\documentclass{amsart}
\usepackage[utf8]{inputenc}
\usepackage{amssymb,amsmath,mathtools,latexsym}
\usepackage{amsthm}
\usepackage{setspace}
\usepackage{graphicx}
\usepackage{mathrsfs}
\usepackage{mdwlist}
\usepackage{color}
\usepackage{bbm}
\usepackage{url}
\usepackage[colorlinks=true]{hyperref}
\usepackage{breakurl}
\usepackage[all,cmtip]{xy}
\usepackage[shortlabels]{enumitem}
\usepackage{tikz-cd}

\newtheorem{theorem}{Theorem}
\newtheorem{lemma}[theorem]{Lemma}
\newtheorem{definition}[theorem]{Definition}
\newtheorem{proposition}[theorem]{Proposition}

\newtheorem{conj}[theorem]{Conjecture}

\newtheorem{cor}[theorem]{Corollary}
\theoremstyle{remark}\newtheorem{remark}[theorem]{Remark}

\newtheorem*{claim*}{Claim}

\numberwithin{theorem}{section}
\numberwithin{equation}{section}
\newcommand{\G}{\mathbf{G}} % alg. group G
\renewcommand{\H}{\mathbf{H}} % alg. group G
\newcommand{\Gm}{\mathbb{G}_m}
\newcommand{\GL}{\operatorname{GL}}
\newcommand{\PGL}{\operatorname{PGL}}
\newcommand{\SL}{\operatorname{SL}}
\newcommand{\SO}{\operatorname{SO}}
\newcommand{\Spin}{\operatorname{Spin}}
\newcommand{\torus}{\mathbf{T}} % algebraic torus
\newcommand{\Mat}{\operatorname{Mat}}

\newcommand{\Aut}{\operatorname{Aut}}

\newcommand{\Nr}{\mathrm{Nr}} % norm
\newcommand{\A}{\mathbb{A}}
\newcommand{\Res}{\operatorname{Res}} % restriction of scalars
\newcommand{\disc}{\operatorname{disc}}
\renewcommand{\mod}{\, \mathrm{mod}\,} % mod without weird spacing
\newcommand{\de}{\,\mathrm{d}} % integral
\newcommand{\vol}{\mathrm{vol}} % volume
\newcommand{\height}{\mathrm{ht}} % height
\newcommand{\norm}[1]{\|#1\|} % height
\newcommand{\Ad}{\mathrm{Ad}}

\DeclareMathOperator{\Tr}{Tr}
\DeclareMathOperator{\gen}{gen}
\DeclareMathOperator{\spn}{spn}

\DeclareMathOperator{\diag}{diag}
\DeclareMathOperator{\pr}{pr}
\DeclareMathOperator{\Av}{Av}

\newcommand{\Bow}{\mathrm{Bow}}

\newcommand\rquot[2]{
  \mathchoice
  {% \displaystyle
    \text{\raise0.5ex\hbox{$#1$}\big/\lower0.5ex\hbox{$#2$}}%
  }
  {% \textstyle
    #1\,/\,#2
  }
  {% \scriptstyle
    #1\,/\,#2
  }
  {% \scriptscriptstyle
    #1\,/\,#2
  }
}

\newcommand\lquot[2]{
  \mathchoice
  {% \displaystyle
    \text{\lower0.5ex\hbox{$#1$}\big\backslash\raise0.5ex\hbox{$#2$}}%
  }
  {% \textstyle
    #1\,\backslash\,#2
  }
  {% \scriptstyle
    #1\,\backslash\,#2
  }
  {% \scriptscriptstyle
    #1\,\backslash\,#2
  }
}

\newcommand\lrquot[3]{
  \mathchoice
  {% \displaystyle
    \text{\lower0.5ex\hbox{$#1$}\big\backslash\raise0.5ex\hbox{$#2$}\big/
      \lower0.5ex\hbox{$#3$}}%
  }
  {% \textstyle
    #1\,\backslash\,#2\,/\,#3
  }
  {% \scriptstyle
    #1\,\backslash\,#2\,/\,#3
  }
  {% \scriptscriptstyle
    #1\,\backslash\,#2\,/\,#3
  }
}

\newcommand{\Dcal}{\mathcal{D}}

\newcommand{\Ocal}{\mathcal{O}}

\newcommand{\Scal}{\mathcal{S}}

\newcommand{\bA}{\mathbb{A}}

\newcommand{\C}{\mathbb{C}}

\newcommand{\N}{\mathbb{N}}

\newcommand{\Q}{\mathbb{Q}}

\newcommand{\R}{\mathbb{R}}

\newcommand{\Z}{\mathbb{Z}}

\newcommand{\Bbf}{\mathbf{B}}

\newcommand{\Hbf}{\mathbf{H}}

\newcommand{\Lbf}{\mathbf{L}}
\newcommand{\Mbf}{\mathbf{M}}

\newcommand{\Tbf}{\mathbf{T}}

\newcommand{\ord}{\mathrm{ord}}

\newcommand{\id}{\mathrm{id}}

\DeclarePairedDelimiter\abs{\lvert}{\rvert}%

\makeatletter
\let\oldabs\abs
\def\abs{\@ifstar{\oldabs}{\oldabs*}}

\title[Representations of binary forms by quaternary forms]{Representations of binary quadratic forms by quaternary quadratic forms}
\author{Wooyeon Kim}
\address[W.K.]{Korea institute for Advanced Study, 85 Hoegi-ro, Dongdaemun-gu, Seoul, Republic of Korea}
\email{wooyeonkim@kias.re.kr}

\author{Andreas Wieser}
\address[A.W.]{Institute for Advanced Study, 1 Einstein Drive, Princeton, NJ 08540, USA}
\email{awieser@ias.edu}

\author{Pengyu Yang}
\address[P.Y.]{Morningside Center of Mathematics, Academy of Mathematics and Systems Science, Chinese Academy of Sciences, Beijing, China}
\email{yangpengyu@amss.ac.cn}

\date{\today}
\thanks{
P.Y. is supported by National Key R\&D Program of China 2022YFA1007500.
W.K.~is supported by a Korea Institute for Advanced Study individual grant no.~HP101301.
A.W.~acknowledges the support of the Swiss National Science Foundation, grant no.~217944.
This material is based upon work supported by a grant from the Institute for Advanced Study School of Mathematics.
}

\begin{document}

\begin{abstract}
We prove a local-global principle for primitive representations of binary quadratic forms by quaternary quadratic forms.
Our method is a variant of Linnik's ergodic method showing density for certain homogenous toral sets.
The central ingredient is a measure classification result of Einsiedler and Lindenstrauss for actions of rank two diagonalizable groups on quotients of products of $\SL_2$.
This rigidity result together with an application of the Siegel mass formula reduce the density problem to a counting problem on a certain affine variety. 
We solve that counting problem using the determinant method of Bombieri-Pila and Heath-Brown.
\end{abstract}
\maketitle
\section{Introduction}

Consider two integral quadratic forms $q:\Z^m \to \Z$ and $Q: \Z^n \to \Z$ for $m < n$.
We write $\disc(Q)$ for the discriminant of $Q$ i.e.~the determinant of any matrix representation of $Q$; $\disc(q)$ is defined analogously.
A \emph{representation} of $q$ by $Q$ is a linear map $\iota: \Z^m \to \Z^n$ with $Q\circ \iota = q$.
Furthermore, the representation $\iota$ is \emph{primitive} if $\iota(\Z^m) = (\iota(\Z^m)\otimes\Q) \cap \Z^n$.
Lastly, we will say that $q$ is (primitively) representable by $Q$ if a (primitive) representation of $q$ by $Q$ exists.

The question how to decide whether a (primitive) representation of $q$ by $Q$ exists is a very classical one.
Certainly, there is a necessary local criterion: there should exist (primitive) representations $\iota_p: \Z_p^m \to \Z_p^n$ of $q$ by $Q$ for every prime $p$. 
It is also necessary that the minimum
\begin{align*}
\min(q) := \min_{x \in \Z^m\setminus\{0\}} q(x)
\end{align*}
of $q$ be sufficiently large in terms of $Q$.
These two necessary conditions together turn out to be sufficient for primitive representations when the codimension $n-m$ is at least three. 
This was established only very recently in work of Einsiedler, Lindenstrauss, Mohammadi and A.W. \cite{effectivesemisimple} and has a very long history; we refer to \cite{effectivesemisimple} and the announcement of the present article \cite{KWY-2in4announcement}.
We note that these advances for $n-m\geq 3$ go hand-in-hand with progress in unipotent dynamics (namely effective equidistribution results for semisimple adelic periods).

The question remains what can be proven when $n-m=2$ where these results in effective unipotent dynamics are inapplicable.
The above two necessary criteria (on local representations and on the minimum) turn out to be insufficient in codimension $n-m=2$ due to so-called spinor obstructions.
Thus, the following is typically conjectured in codimension two.

\begin{conj}\label{conj:localglobal n-m=2}
Suppose $n-m=2$.
There exists $C(Q)>0$ with the following property.
If $q$ is primitively represented by (an element of) the spin genus of $Q$ and
\begin{align*}
\min(q) \geq C(Q),
\end{align*}
then $q$ is primitively represented by $Q$.
\end{conj}

We note that primitive representability by the spin genus is equivalent to local primitive representability whenever the discriminant $\disc(q)$ is outside of finitely many square classes of integers.
Moreover, if the discriminant $\disc(Q)$ is not divisible by a `high' prime power (see \cite[Ch.~11, Thm.~1.3]{cassels}), the spin genus of $Q$ is equal to the genus of $Q$ and primitive representability by the spin genus is again equivalent to local primitive representability (by the Hasse-Minkowski theorem).

When $m=1$, i.e.~for primitive representations of integers, the above conjecture has been established by Duke and Schulze-Pillot \cite{DukeSchulzePillot} based on estimates of Duke \cite{duke88} and Iwaniec \cite{Iwaniec-halfintegral} for Fourier coefficients of forms of half-integral weight.
Cogdell, Piatetski-Shapiro, and Sarnak \cite{Cogdell-squares} resolved the analogous problem over totally real number fields.
We remark that Linnik's ergodic method \cite{linnik} (see also \cite{ELMV-Ens,W-Linnik}) can be applied to prove a version of Conjecture~\ref{conj:localglobal n-m=2} for $m=1$ assuming a splitting condition; his arguments are central to the approach of the current article.
Lastly, we point out that, in contrast to \cite{effectivesemisimple} for $n-m \geq 3$, the constant $C(Q)$ established in \cite{DukeSchulzePillot} is ineffective due to possible Landau-Siegel zeros.

When $m>1$, Conjecture~\ref{conj:localglobal n-m=2} has seen very little progress.
In the following, we will focus on $m=2$.
Using analytic methods, Schulze-Pillot \cite{SchulzePillot-2in4} showed that a positive proportion of binary quadratic forms of discriminant $\disc(q)$ are primitively represented by $Q$.
In the present article, we prove Conjecture~\ref{conj:localglobal n-m=2} under two Linnik-type splitting conditions:

\begin{theorem}\label{thm:main}
Let $p_1,p_2$ be two distinct odd primes and let $Q$ be a positive definite quaternary integral quadratic form.
Then there exists $C=C(p_1,p_2,Q)>1$ with the following property.

Let $q$ be a primitive binary integral quadratic form such that $-\disc(q)\disc(Q)$ is a non-zero square modulo $p_1$ and $p_2$.
If $q$ is primitively represented by the spin genus of $Q$ and $\min(q) \geq C$ then $q$ is primitively represented by $Q$.
\end{theorem}

The assumption that $q$ be primitive should not be seen as central to our approach and can be relaxed.
On the other hand, the two auxiliary primes $p_1,p_2$ are an artifact of our method; they make Conjecture~\ref{conj:localglobal n-m=2} accessible to deep measure rigidity results for diagonalizable actions on homogeneous spaces established by Einsiedler and Lindenstrauss \cite{EL-joiningsPIHES,EL-nonmaximal}.

In contrast to other applications of these rigidity results, see e.g.~\cite{IlyaAnnals,AES3D,AEW-2in4}, the underlying result for adelic toral periods is a \emph{density} result rather than an equidistribution result.
In fact, we show lower bounds rather than asymptotics for representation numbers.
Let $r(q,Q)$ be the number of primitive representations of $q$ by $Q$.
Define
\begin{align*}
\omega_{\spn(Q)} &= \sum_{Q' \in \spn(Q)} |\SO_{Q'}(\Z)|^{-1}
\end{align*}
and
\begin{align*}
r(q,\spn(Q))     = \omega_{\spn(Q)}^{-1} \sum_{Q' \in \spn(Q)} r(q,Q') |\SO_{Q'}(\Z)|^{-1}.
\end{align*}
The following is a finer version of Theorem~\ref{thm:main}.

\begin{theorem}\label{thm:main2}
There exists $\delta>0$ with the following property.
Let $p_1,p_2$ be two distinct primes and let $q,Q$ be as in Theorem~\ref{thm:main}.
Then
\begin{align*}
r(q,Q) \geq r(q,\mathrm{spin}(Q)) \big( \delta + \varepsilon(q))
\end{align*}
where $\varepsilon(q)$ is a function depending on $Q,p_1,p_2$ and on $q$ with $\varepsilon(q) \to 0$ when $\min(q)$ goes to infinity.
\end{theorem}

If $q$ is primitively represented by the spin genus of $Q$, then one can show \cite[Lemma 2.1]{SchulzePillot-survey04} using the Siegel mass formula \cite{Siegel-I} and lower bounds for local representation densities that
\begin{align*}
r(q,\mathrm{spin}(Q)) \gg_{\varepsilon} \disc(q)^{\frac{1}{2}-\varepsilon}
\end{align*}
where the implicit constant is not effective (due to Landau-Siegel zeros).
Together with Theorem~\ref{thm:main2}, this implies that, whenever $\min(q)$ is sufficiently large, we have
\begin{align*}
r(q,Q) \gg_{\varepsilon} \disc(q)^{\frac{1}{2}-\varepsilon}.
\end{align*}

\begin{remark}[Mixing conjecture]
The local-global principle in Conjecture~\ref{conj:localglobal n-m=2} for $m=2$ and $n=4$ is closely related to the mixing conjecture of Michel and Venkatesh \cite{MVIHES}.
Tremendous progress has been achieved towards the mixing conjecture in recent years --- see the work of Khayutin \cite{IlyaAnnals}, Blomer and Brumley \cite{BlomerBrumley-mixing}, Blomer, Brumley, Khayutin \cite{BlomerBrumleyKhayutin}, as well as the very recent work of Blomer, Brumley and Radziwi\l\l\, \cite{BlomerBrumleyRadziwill}.
The methods in these works do not seem to apply to our setting, particularly when $\disc(Q)$ is not a square.
The present article's approach has implications towards the mixing conjecture, which will be made explicit in an upcoming work.
\end{remark}

As mentioned, the approach pursued in this article is strongly inspired by Linnik's ergodic method \cite{linnik,ELMV-Ens,W-Linnik}) and has few analytic ingredients.
In particular, there is justified hope that this soft method might extend to related problems in higher-rank (density of toric periods).

\subsection*{Acknowledgments}
The authors would like to thank Manfred Einsiedler, Elon Lindenstrauss, Philippe Michel, Peter Sarnak, and Akshay Venkatesh for their interest in this project and for their encouragement.
We are also grateful to Menny Aka, Farrell Brumley, Zhizhong Huang, Ilya Khayutin, Rainer Schulze-Pillot, Per Salberger, Ye Tian, and Katherine Woo for many helpful discussions on a range of related topics.
Finally, we thank the Forschungsinstitut f\"ur Mathematik at ETH Z\"urich, the Institute for Advanced Study, the Korea Institute for Advanced Study, the Morningside Center of Mathematics at CAS, and the Simons–Laufer Institute for their hospitality and for providing an excellent research environment.

\subsection{Structure of the article}\label{sec:structure}
For an outline of the proof, we refer to the research announcement \cite{KWY-2in4announcement} of the present article.
Our proof below follows this outline, and we recommend consulting it.
\begin{itemize}
\item In \S\ref{sec:translation}, we translate our main theorems into a problem for certain adelic toral periods. 
More specifically, these periods are given by rank one tori $\torus < \Spin_Q$ (pointwise stabilizers groups of representations).
This discussion is classical and, for instance, already contained in \cite[\S2]{effectivesemisimple}.
\item In \S\ref{sec:subpolynomial}, we establish our main theorems when $\min(q) \leq \disc(q)^{\delta_1}$ for an explicit $\delta_1>0$. In this low range, our argument is based on equidistribution in `stages' where one uses that the torus $\mathbf{T}$ in question is contained in an intermediate group of much smaller complexity. 
\end{itemize}
In all other sections, we assume that $\min(q) > \disc(q)^{\delta_1}$.
\begin{itemize}
\item In \S\ref{sec:measurerigidity}, we apply the aforementioned measure classification results for diagonalizable actions of positive entropy due to Einsiedler and Lindenstrauss \cite{EL-joiningsPIHES,EL-nonmaximal}; this reduces the problem to showing that certain weak${}^\ast$-limits of homogeneous toral measures have entropy bigger than a half of the maximal entropy.
\item In \S\ref{sec:reduce to counting}, we reduce the above entropy claim to a counting problem for integer points on a rational hypersurface of $\mathbf{A}^5$ of degree four (where the points also satisfy archimedean and non-archimedean restrictions).
\item In \S\ref{sec:solve counting}, we solve the counting problem using the determinant method of Bombieri-Pila \cite{BombieriPila} and Heath-Brown \cite{HeathBrown-annals} (see also Appendix~\ref{sec:det method}). 
\end{itemize}

The appendix is structured as follows. In Appendix~\ref{sec:det method}, we establish a variant of the determinant method for counting in planar (non-square) rectangles.
In Appendix~\ref{sec:Siegel mass}, we apply the Siegel mass formula to estimate the number of representations of a ternary or quaternary quadratic form by a quaternary quadratic form.
In Appendix~\ref{sec:Duke}, we present an effective equidistribution result for toral periods in periods of anisotropic inner forms of $\PGL_2$.
We note that all results established in the appendix are well-known to experts, we present them merely for lack of explicit reference.

\subsection{Some notations}
Given two quantities $A,B>0$ we write $A \ll B$ to mean that there exists an implicit constant $c>0$ with $A \leq cB$. 
If the constant depends on an object $W$, we write $A \ll_W B$ unless the dependency is clear from context.
Moreover, $A \asymp B$ if $A \ll B$ and $B \ll A$.
If there exists $c>0$ with $A \leq B^c$ we write $A \leq B^\star$.

Given a natural number $N$ and a prime power $p^\ell$, we write $p^\ell \Vert N$ to mean $p^\ell \mid N$ and $p^{\ell+1}\nmid N$.

\section{Translation into density of toral packets}\label{sec:translation}

In this section, we shall interpret our main theorems --- Theorems~\ref{thm:main}, \ref{thm:main2} --- as a strengthening of density result for certain homogeneous toral sets.

\subsection{General notation}\label{sec:notation qf}
Let $Q$ be a positive definite quaternary integral quadratic form in four variables.
Let $\langle\cdot,\cdot\rangle_Q$ be the symmetric bilinear form such that $Q(w)=\langle w,w\rangle_Q$ for every $w\in \Q^4$.
We denote by $\G = \Spin_Q$ the spin group of the quadratic form $Q$.
By definition, $\G$ is the universal cover of $\overline{\G}=\SO_Q$ and we denote by $\rho: \G \to \overline{\G}$ the covering map.
We also write $\rho(g)x = g.x$ for the standard representation of $\G$.
Denote
\begin{align}\label{eq:localglobal-ambientspaces}
X = \overline{\G}(\Q)\rho(\G(\A)) \subset \overline{X} = [\overline{\G}(\bA)] = \lquot{\overline{\G}(\Q)}{\overline{\G}(\A)}
\end{align}
and the $\rho(\G(\A))$ resp.~$\overline{\G}(\A)$-invariant probability measure on $X$ resp.~$\overline{X}$ by $\mu_X$ resp.~$\mu_{\overline{X}}$.

Note that $\G$ is an $\R$-anisotropic $\Q$-form of $\SL_2\times \SL_2$.
One can show that $\G$ is an outer form if and only if the discriminant $\disc(Q)$ is not a rational square.
Recall that $\disc(Q)$ denotes the determinant of a matrix representative of $Q$.

For an integral binary quadratic form $q(x,y) = Ax^2+Bxy+Cy^2$ we similarly denote by $\disc(q) = AC -\frac{1}{4}B^2$ its discriminant.
Note that $\disc(q)$ differs by a factor of $-4$ from the alternative convention $B^2-4AC$ more common in algebraic number theory.
Representations of $q$ by $Q$ (or other elements of the spinor genus of $Q$) will be denoted by $\iota$.
Recall that $\iota$ is primitive if $\iota(\Z^2) = (\iota(\Z^2) \otimes \Q)\cap \Z^4$.
Note that if $-4\disc(q)$ is fundamental (i.e.~it is equal to the discriminant of the quadratic field $\Q(\sqrt{-\disc(q)})$), then any representation is primitive.
One may check that $-4\disc(q)$ is fundamental if and only if either $\disc(q) \not\in\Z$, $-4\disc(q)$ is square-free and congruent to $1\mod 4$ or $\disc(q) \in\Z$, $\disc(q)$ is square-free, and congruent to $1$ or $2\mod 4$.
We write $\min(q) = \min_{x \in \Z^2\setminus\{0\}} q(x)$ as before.

\subsection{Genus, spin genus, and primitive representations}

The following discussion is similar to \cite[\S2]{effectivesemisimple}.
We recall the action of the group $\GL_4(\A_f)$ acts on the set of (full rank) lattices $\Lambda$ in $\Q^4$ via
\begin{align*}
g_\star\Lambda = \bigcap_p \big(g_p.(\Lambda \otimes \Z_p) \cap \Q^4\big).
\end{align*}
Notice that the restriction of this action to $\GL_4(\Q)$ induces the standard action of $\GL_4(\Q)$ on lattices in $\Q^4$.

\subsubsection{The genus and spin genus of $Q$}
The genus resp.~spin genus of $Q$ are defined to be
\begin{align*}
\gen(Q) = \overline{\G}(\A_f)_\ast \Z^4\supset 
\spn(Q) = \big((\overline{\G}(\Q)\rho(\G(\A_f))\big)_\star \Z^4.
\end{align*}
Notice that the quadratic form $Q$ restricts to an integral form on each lattice in the genus; the so-obtained set of quadratic forms is, for simplicity, also called the genus and consists precisely of quadratic forms locally equivalent to $Q$.
The same convention applies to the spin genus, though the corresponding quadratic forms are harder to characterize with referring to lattices --- see e.g.~\cite[Ch.~11]{cassels}.
The genus decomposes into finitely many $\SO_Q(\Q)$-orbits which we call equivalence classes.
Define
\begin{equation}\label{eq:maximalcompact}
\begin{split}
K_f &= \big\{g \in \overline{\G}(\Q)\rho(\G(\A_f)): g_\star \Z^4 =\Z^4\big\},\\
\overline{K}_f &= \big\{g \in \overline{\G}(\A_f): g_\star \Z^4 =\Z^4\big\}.
\end{split}
\end{equation}
In particular, $\gen(Q) \simeq \rquot{\overline{X}}{(\overline{K}_f}\overline{\G}(\R))$.

\subsubsection{Primitive representations by the genus and spin genus}
In the following it will be convenient to view a
primitive representations of $q$ by an element of the genus of $Q$ as pairs $(\mathfrak{L},\iota)$ where $\mathfrak{L} \in \gen(Q)$ and $\iota:\Q^2 \to \Q^4$ is an isometry (i.e.~$Q\circ \iota =q$) with $\mathfrak{L} \cap \iota(\Q^2) = \iota(\Z^2)$.
We view primitive representation by the spin genus analogously.
Note that, in this viewpoint, $\SO_Q(\Q)$ acts on the set of primitive representations by the genus via $\gamma.(\mathfrak{L},\iota) = (\gamma\mathfrak{L},\gamma\iota)$; an orbit will be called an equivalence class of representations.
We correspondingly introduce the following notation:
\begin{itemize}
\item $\tilde{\mathcal{R}}(q,\gen(Q))$ (resp.~$\tilde{\mathcal{R}}(q,\spn(Q))$) is the set of equivalence classes of primitive representations of $q$ by the genus (resp.~spin genus) of $Q$.
\item $\tilde{\mathcal{R}}(q,Q)$ is the set of $\SO_Q(\Z)$-orbits of primitive representations of $q$ by $Q$.
\end{itemize}
If $Q=Q_1,\ldots,Q_r$ are representatives of the quadratic forms in the spin genus of $Q$, we clearly have
\begin{align*}
\tilde{\mathcal{R}}(q,\spn(Q)) \simeq \bigsqcup_i \tilde{\mathcal{R}}(q,Q_i);
\end{align*}
a similar statement is true for the genus.
We have a forgetful map onto the set of equivalence classes in the spin genus
\begin{align*}
\tilde{\mathcal{R}}(q,\spn(Q)) \to \SO_Q(\Q) \backslash \spn(Q);
\end{align*}
with Theorem~\ref{thm:main} in mind, one of our goals is to show surjectivity of this forgetful map.

\subsection{Periods of stabilizer groups}\label{sec:period stabilizer}

For any isometry $\iota: (\Q^2,q) \to (\Q^4,Q)$ set
\begin{equation}\label{eq:stabtorus}
\begin{split}
\torus_\iota &= \{g \in \G: g.\iota(x) = \iota(x) \text{ for all }x \in \Z^2\},\\
\overline{\torus}_{\iota} &= \{g \in \overline{\G}: g\iota(x) = \iota(x) \text{ for all }x \in \Z^2\}.
\end{split}
\end{equation}
These are rank one tori in $\G$ resp.~$\overline{\G}$ and $\overline{\torus}_{\iota}$ it is isomorphic to the group $\SO_{Q|_{\iota(\Z^2)^\perp}}$.
We choose $g_{\iota,\infty}\in \G(\R)$ such that $g_{\iota,\infty}^{-1}\torus_\iota(\R)g_{\iota,\infty}$ does not depend on $\iota$ (for instance, take $g_{\iota,\infty}$ to rotate the subspace $\iota(\Z^2) \otimes \R \subset \R^n$ to a fix reference subspace).

For any primitive representation $(\mathfrak{L},\iota)$ where $\mathfrak{L}=g_\star \Z^4 \in \gen(Q)$ we define the homogeneous toral subsets of $\overline{X}$ 
\begin{align}\label{eq:homtor-generalized}
Y_{\mathfrak{L},\iota} = \overline{\G}(\Q)\rho(\torus_{\iota}(\A)g_{\iota,\infty})g,\quad
\overline{Y}_{\mathfrak{L},\iota} = \overline{\G}(\Q)\overline{\torus}_{\iota}(\A)\rho(g_{\iota,\infty})g.
\end{align}
One may check that $Y_{\mathfrak{L},\iota},\overline{Y}_{\mathfrak{L},\iota}$ depend only on the class $\SO_Q(\Q).(\mathfrak{L},\iota) \in \tilde{\mathcal{R}}(q,\gen(Q))$ of the primitive representation $(\mathfrak{L},\iota)$.
If $\mathfrak{L}=\Z^4$ we simplify \eqref{eq:homtor-generalized} to
\begin{align}\label{eq:deftoralsets}
Y_{\iota} = \overline{\G}(\Q) \rho(\torus_\iota(\A)g_{\iota,\infty}),\quad
\overline{Y}_\iota = \overline{\G}(\Q) \overline{\torus}_\iota(\A)\rho(g_{\iota,\infty}).
\end{align}
We write $\mu_{\mathfrak{L},\iota}= \mu_{Y_{\mathfrak{L},\iota}}$ for the $g^{-1}\rho(g_{\iota,\infty}^{-1}\torus_\iota(\A)g_{\iota,\infty})g$-invariant probability measure on $Y_\iota$ and define $\overline{\mu}_{\mathfrak{L},\iota}= \mu_{\overline{Y}_{\mathfrak{L},\iota}}$ similarly.
We again simplify this to $\mu_{\iota}$ resp.~$\overline{\mu}_{\iota}$ if $\mathfrak{L}=\Z^4$.

\subsubsection{Generating primitive representations}
Suppose now that $(\mathfrak{L},\iota)$ is a primitive representation by the spin genus of $Q$.
We have a map
\begin{equation}\label{eq:from torus to rep}
\begin{split}
\Phi_{\mathfrak{L},\iota}:    \overline{\torus}_{\iota}(\Q) \backslash (\overline{\torus}_{\iota}(\Q)\rho(\torus_{\iota}(\A_f))_\star \mathfrak{L}) 
&\to \tilde{\mathcal{R}}(q,\spn(Q)),\\
\overline{\torus}_{\iota}(\Q)t_\star \mathfrak{L}
&\mapsto \SO_Q(\Q).(t_\star \mathfrak{L},\iota).
\end{split}
\end{equation}
and similarly
\begin{equation}\label{eq:from torus to rep-genus}
\begin{split}
\overline{\Phi}_{\mathfrak{L},\iota}:    \overline{\torus}_{\iota}(\Q) \backslash \overline{\torus}_{\iota}(\A_f)_\star \mathfrak{L} 
&\to \tilde{\mathcal{R}}(q,\gen(Q)),\\
\overline{\torus}_{\iota}(\Q)t_\star \mathfrak{L}
&\mapsto \SO_Q(\Q).(t_\star \mathfrak{L},\iota).
\end{split}
\end{equation}
As in \cite[Lemma 2.7]{effectivesemisimple}, one may verify that the maps in \eqref{eq:from torus to rep} and \eqref{eq:from torus to rep-genus} are injective; surjectivity fails in general (see e.g.~\cite[Rem.~2.8]{effectivesemisimple}).

We record the following simple observation (not used for Theorem~\ref{thm:main}):

\begin{lemma}\label{lem:partition image}
Suppose that $(\mathfrak{L}_1,\iota_1)$ and $(\mathfrak{L}_2,\iota_2)$ are two primitive representations. Then the images of the maps $\Phi_{\mathfrak{L}_1,\iota_1}$ and $\Phi_{\mathfrak{L}_2,\iota_2}$ are either equal or disjoint.
The analogous statement applies to the maps in \eqref{eq:from torus to rep-genus}.
\end{lemma}

\begin{proof}
Suppose that $\SO_Q(\Q).(\mathfrak{L},\iota)$ is in the intersection of the images of the maps $\Phi_{\mathfrak{L}_1,\iota_1}$ and $\Phi_{\mathfrak{L}_2,\iota_2}$. By symmetry, it is enough to show that the images of $\Phi_{\mathfrak{L}_1,\iota_1}$ and $\Phi_{\mathfrak{L},\iota}$ agree.
Here, note that the image of $\Phi_{\mathfrak{L},\iota}$ is independent of the choice of representative within $\SO_Q(\Q).(\mathfrak{L},\iota)$.
Since $\SO_Q(\Q).(\mathfrak{L},\iota)$ is in the image of $\Phi_{\mathfrak{L}_1,\iota_1}$, there exists $\gamma \in \SO_Q(\Q)$ and $t \in \rho(\torus_{\iota}(\A_f))$ such that $\gamma.(\mathfrak{L},\iota) = (t_\star \mathfrak{L}_1,\iota_1)$. We may assume, by replacing $(\mathfrak{L},\iota)$ within its class $\SO_Q(\Q).(\mathfrak{L},\iota)$, that $\gamma = \id$.
Correspondingly, $\iota=\iota_1$ and the tori for $\iota$ and $\iota_1$ agree which proves the lemma.
\end{proof}

\subsection{From distributional properties of toral sets to local-global principles}

Distributional properties of the homogeneous toral sets defined in \eqref{eq:homtor-generalized} and \eqref{eq:deftoralsets} imply corresponding statements for representations of quadratic forms.
Here, we aim to phrase a technical proposition capturing these implications.

Given a subset $X'\subset X$ we say that a sequence of subsets $Y_i \subset X'$ is \emph{asymptotically dense} in $X'$ if for any open subset $\mathcal{O}\subset X'$ we have $Y_i \cap \mathcal{O}\neq \emptyset$ for all $i$ sufficiently large.
Also, recall that the character spectrum in $L^2(X)$ consists of functions invariant under $\rho(\G(\A))$ i.e.~arising through the factor map $\overline{X} \to \overline{X}/\rho(\G(\A))$. 
Here, note that $\overline{X}/\rho(\G(\A))$ is a compact abelian group.

\begin{proposition}\label{prop:localglobalgiven?}
Let $q_i$ be a sequence of integral binary quadratic forms with $\min(q_i) \to \infty$ for $i \to \infty$.
Then the following statements hold:
\begin{enumerate}[label=\textnormal{(\Alph*)}]
\item\label{item:localglobalgiven-density} \emph{(Asymptotic density)} Suppose that there exists a sequence of primitive representations $\iota_i$ of $q_i$ by $Q$ such that the homogeneous toral sets $Y_{\iota_i}$ are asymptotically dense in $X$.
Then all $q_i$ for $i$ sufficiently large are primitively represented by \emph{all} elements of the spin genus of $Q$.
\item\label{item:localglobalgiven-Haarcomp} \emph{(Positive Haar component)} Suppose that there exists $\delta>0$ so that for any $\mathfrak{L}\in \spn(Q)$ and for any sequence of primitive representations $(\mathfrak{L},\iota_i)$ of $q_i$ the following holds:
\begin{quote}
Any weak${}^\ast$-limit $\mu$ of the measures $\mu_{\mathfrak{L},\iota_i}$ satisfies $\mu \geq \delta \mu_X$.
\end{quote}
Then as $i\to \infty$
\begin{align*}
r(q_i,Q) \geq r(q_i,\mathrm{spin}(Q)) \big( \delta + o(1)).
\end{align*}
\item\label{item:localglobalgiven-equi1} \emph{(Equidistribution)} 
Suppose that for any $\mathfrak{L}\in \spn(Q)$ and for any sequence of primitive representations $(\mathfrak{L},\iota_i)$ of $q_i$ the measures $\mu_{\mathfrak{L},\iota_i}$ converge to $\mu_X$.
Then as $i\to \infty$
\begin{align*}
r(q_i,Q) = r(q_i,\mathrm{spin}(Q)) \big(1 + o(1)).
\end{align*}
\end{enumerate}
Furthermore,
\begin{enumerate}[label=$(\overline{\textnormal{\Alph*}})$]
\setcounter{enumi}{2}
\item\label{item:localglobalgiven-equi2}
    Suppose that for any $\mathfrak{L}\in \spn(Q)$ and for any sequence of primitive representations $(\mathfrak{L},\iota_i)$ of $q_i$ the following holds:
    \begin{quote}
    For any $f \in C^\infty(\overline{X})$ orthogonal to the character spectrum we have $\overline{\mu}_{\mathfrak{L},\iota_i}(f) \to \overline{\mu}_X(f)$.
    \end{quote}
Then as $i\to \infty$
\begin{align*}
r(q_i,Q) = r(q_i,\mathrm{spin}(Q)) \big(1 + o(1)).
\end{align*}
\end{enumerate}
\end{proposition}

The above proposition is certainly standard (see e.g.~\cite[\S2]{effectivesemisimple}) and we do not fully prove it. (Also, note that \ref{item:localglobalgiven-equi1} is not be used in this article.)

\begin{remark}
Since $\mu_{\mathfrak{L},\iota_i}$ is merely a translate of $\mu_{\iota_i}$ (where $\iota_i$ is notably not a primitive representation of $q$ by $Q$), the assumptions in \ref{item:localglobalgiven-Haarcomp} and \ref{item:localglobalgiven-equi1} may as well be imposed on the measures $\mu_{\iota_i}$.
The analogous statement applies in \ref{item:localglobalgiven-equi1}.
\end{remark}

\begin{proof}[Proof of Proposition~\ref{prop:localglobalgiven?} \ref{item:localglobalgiven-density}]
Let $\iota$ be a primitive representation of a binary quadratic form $q$ by $Q$.
In view of the assumptions in \ref{item:localglobalgiven-density}, we may additionally assume that $Y_{\iota}$ intersects all $\rho(\G(\R))\overline{K}_f$-orbits in $X$ where $\overline{K}_f$ is as in \eqref{eq:maximalcompact}.

To see that $\tilde{\mathcal{R}}(q,\spn(Q)) \to \SO_Q(\Q) \backslash \spn(Q)$ is surjective, we may as well show that its precomposition with $\Phi_{\mathfrak{L},\iota}$ (cf.~\eqref{eq:from torus to rep}) is surjective or, equivalently, that the natural map
\begin{align*}
Y_{\iota} \to X/(\rho(\G(\R))\overline{K}_f) = \SO_Q(\Q) \backslash \spn(Q)
\end{align*}
is surjective. The latter is true by assumption.
\end{proof}

\begin{proof}[Proof of Proposition~\ref{prop:localglobalgiven?} \ref{item:localglobalgiven-equi2}]
Let $\mathfrak{L} = g_{\star}\Z^4 \in \spn(Q)$ and suppose that $(\mathfrak{L},\iota_i)$ is a sequence of primitive representations of $q$ by $\mathfrak{L}$.
By assumption, the probability measures $\overline{\mu}_{\iota_i}$ on the homogeneous toral sets $\overline{Y}_{\mathfrak{L},\iota_i}$ satisfy $\overline{\mu}_{\mathfrak{L},\iota_i}(f) \to \overline{\mu}_X(f)$ for any $f \in C^\infty(\overline{X})$ orthogonal to the character spectrum.

Consider the composed homomorphism
\begin{align*}
\overline{\torus}_{\iota_i}(\A) \hookrightarrow \overline{\G}(\A) \to \overline{\G}(\Q)\backslash \overline{\G}(\A)/\rho(\G(\A))\overline{K}_f=\overline{X}/\rho(\G(\A))\overline{K}_f = \mathcal{F}.
\end{align*}
Since $\mathcal{F}$ is a finite abelian group, its kernel is a finite index subgroup of $\torus_{\iota_i}(\A)$ of index at most $\#\mathcal{F}$.
We denote the kernel by $T_{\iota_i}'$ and let 
\begin{align*}
Y_{\iota_i}' = \overline{\G}(\Q)T_{\iota_i}' \rho(g_{\iota,\infty})g = \overline{Y}_{\mathfrak{L},\iota_i} \cap X\overline{K}_f \subset X\overline{K}_f
\end{align*}
By construction, if $\psi$ is a $\overline{K}_f$-invariant function supported on $X\overline{K}_f \subset \overline{X}$ and $x \in \overline{Y}_{\mathfrak{L},\iota_i}$ satisfies $\psi(x) \neq 0$ then $x \in Y_{\iota_i}'$.

\begin{claim*}
For any $\overline{K}_f$-invariant function $\psi$ supported on $X\overline{K}_f$ we have
\begin{align*}
\int_{Y_{\iota_i}'} \psi \to \int_{X\overline{K}_f} \psi,
\end{align*}
where the integrals are taken against the respective invariant probability measures.
\end{claim*}

\begin{proof}[Proof of Claim]
If $\psi$ is in addition orthogonal to the character spectrum, then $\int_{Y_{\iota_i}'} \psi \to 0$ as $\int_{\overline{Y}_{\mathfrak{L},\iota_i}}\psi \to 0$.
On the other hand, if $\psi$ is in the character spectrum then $\psi$ is a multiple of the delta-function on $\mathcal{F}$ supported at the identity. In this case, $\int_{Y_{\iota_i}'} \psi = \int_{X\overline{K}_f} \psi$ and the claim follows.
\end{proof}

The claim applied to characteristic functions on $\overline{K}_f$-orbits on $X\overline{K}_f$ implies Proposition~\ref{prop:localglobalgiven?}\ref{item:localglobalgiven-equi2} by invoking the same finishing argument as in \cite[\S2]{effectivesemisimple}.
Indeed, one partitions $\tilde{\mathcal{R}}(q,\gen(Q))$ according to Lemma~\ref{lem:partition image}, restricts each set to fiber over the spin genus of $Q$ (whence the definition of the above sets $Y_{\iota_i}'$), and computes weights to reduce to equidistribution of these restricted sets (which is the content of the above claim).
\end{proof}

\section{The subpolynomial case}\label{sec:subpolynomial}

In this section, we show that if $\min(q)$ is small, our main theorems reduce to a simpler equidistribution argument in stages, establishing stronger results directly (see Theorem~\ref{thm:small min} below).
We follow the notation in \S\ref{sec:translation}.

\subsection{Orthogonal complements}
Let $v\in\Z^4$ be a primitive vector and let $W$ be its orthogonal complement with respect to $Q$.
We write $W(\Z) = W \cap \Z^4$ and note that $\Z^4 = W(\Z) + \Z w_0$ for some $w_0\in\Z^4$.
Then by orthogonal projection of $w_0$ to the line spanned by $v$ we have $\disc(Q) = \disc(Q|_{W(\Z)}) \frac{\langle v,w_0\rangle_Q^2}{Q(v)}$.
Notice that $2\langle v,w_0\rangle_Q$ generates the ideal $\{2\langle v,w\rangle_Q:w\in\Z^4\}$ in $\Z$, and this ideal divides $\disc(2Q)$ as $v$ is primitive.
Hence $\langle v,w_0\rangle_Q\asymp_Q 1$, and it follows that 
\begin{equation}\label{eq:complement of vector}
    \disc(Q|_{W(\Z)})\asymp_Q Q(v).
\end{equation}
More generally, we have the following well-known lemma, a special case of e.g.~\cite[Prop.~5.1]{AMW-higherdim}.

\begin{lemma} \label{lem:orthocomplement disc}
    Let $W$ be a subspace of $\Q^4$, and let $W^\perp$ denote the orthogonal complement. Then we have
    \begin{equation*}
        \frac{\disc(Q|_{W^\perp(\Z)})}{\disc(Q)} \leq \disc(Q|_{W(\Z)}) \leq \disc(Q)\disc(Q|_{W^\perp(\Z)}).
    \end{equation*}
\end{lemma}

\begin{proof}
For simplicity, we set $M = W(\Z)$ and $M^\perp = W^\perp(\Z)$.
    Consider the linear map $f:\Z^4\to M^\vee$ defined by $f(v)(w)=2\langle v,w\rangle_Q$ for every $v\in \Z^4$ and $w\in M$, where $M^\vee:=\mathrm{Hom}(M,\Z)$ denotes the dual lattice of $M$. 
    Then we have $\ker f=M^\perp$, $\Z^4/M^\perp\cong f(\Z^4)$, and $f(M)\cong M$. 
    Thus, $f(\Z^4)/f(M)\cong \Z^4/(M\oplus M^\perp)$. Since $f(\Z^4)/f(M)$ is a subgroup of $M^\vee/M$, we have $\abs{\Z^4/(M\oplus M^\perp)}=\abs{f(\Z^4)/f(M)}\leq\abs{M^\vee/M}=\disc(Q|_{M})$. Hence 
    \begin{align*}
    \disc(Q|_M)\disc(Q|_{M^\perp})=\abs{\Z^4/(M\oplus M^\perp)}^2\disc(Q)\leq\disc(Q|_M)^2\disc(Q)
    \end{align*}
    and thus we have $\disc(Q|_{M^\perp})\leq\disc(Q|_M)\disc(Q)$. The same argument shows that $\disc(Q|_{M})\leq\disc(Q|_{M^\perp})\disc(Q)$ and we conclude.
\end{proof}

\subsection{Equidistribution in stages}
Let $q$ be a binary quadratic form and suppose that $\iota$ is a primitive representation of $q$ by $Q$.
Let $W = \iota(\Q^2)$ and let $v \in \iota(\Z^2) = W(\Z)$ be a primitive vector (where we may assume that $Q(v)$ is minimal).
By Lemma~\ref{lem:orthocomplement disc} we have $\disc(Q|_{W^\perp(\Z)}) \asymp_Q \disc(q)$.

Define $\overline{\torus}=\overline{\torus}_{\iota}$ as in \eqref{eq:stabtorus}.
Let $\Hbf = \{g\in \G: g.v = v\}$, $\overline{\Hbf} = \{g\in \overline{\G}: g.v = v\}$, and note that $\overline{\torus}< \overline{\Hbf}$.
Let $L^2_{00}([\overline{\G}(\bA)])$ denote the orthogonal complement of the subspace of $L^2([\overline{\G}(\bA)])$ consisting of $\rho(\G(\bA))$-invariant square integrable functions (the character spectrum).
We use the Sobolev norms $(\mathcal{S}_d(\cdot))_{d\geq 1}$ defined in \cite[(A.3)]{EMMV}.

The following theorem combines effective equidistribution results for semisimple adelic periods established by Einsiedler, Margulis, Mohammadi and Venkatesh in \cite{EMMV} with the effective equidistribution result for toral periods discussed in Appendix~\ref{sec:Duke}.

\begin{theorem}\label{thm:two step equi}
    There exist $d,\kappa>0$ such that for $f\in C^{\infty}([\overline{\G}(\bA)])\cap L^2_{00}([\overline{\G}(\bA)])$
    \begin{equation*}
        \bigg|\int_{[\overline{\Tbf}(\bA)]} f\bigg|
        \ll_Q Q(v)^{\frac{1}{\kappa}}\disc(q)^{-\kappa}\Scal_{10}(f)+Q(v)^{-\kappa}\Scal_d(f).
    \end{equation*}
\end{theorem}

\begin{proof}
    We show that the periods $[\overline{\Tbf}(\A)] \subset [\overline{\Hbf}(\A)]$ and $[\overline{\Hbf}(\A)] \subset [\overline{\G}(\A)]$ are effectively equidistributed beginning with the latter.
    
    By \cite[Thm.~1.5]{EMMV}, there exist an absolute constant $d>0$ such that for every $h\in\overline{\Hbf}(\bA)$ we have
    \begin{equation}\label{eq:applying EMMV}
        \bigg|\int_{[\rho(\Hbf(\bA))h]}f\bigg| \ll \vol([\rho(\Hbf(\bA))h])^{-\star}\Scal_d(f).
    \end{equation}
    where $[\rho(\Hbf(\bA))h] = \overline{\G}(\Q)\rho(\Hbf(\bA))h \subset [\overline{\G}(\bA)]$.
    Moreover, \cite[App.~A]{effectivesemisimple} defines a notion of complexity $\mathrm{cpl}(\cdot)$ for semisimple adelic periods and shows that
    \begin{equation*}
        \vol([\rho(\Hbf(\bA))h])\gg\mathrm{cpl}([\rho(\Hbf(\bA))h])^{\star}= \mathrm{cpl}([\rho(\Hbf(\bA))])^{\star}.
    \end{equation*}
    In view of \cite[\S2.1]{effectivesemisimple}, we have $Q(v)^\star \ll_Q \mathrm{cpl}([\rho(\Hbf(\bA))])\ll_Q Q(v)^\star$.
    Combining this with \eqref{eq:applying EMMV} we get
    \begin{equation}\label{eq:bounding Htilde period by disc}
        \bigg|\int_{[\rho(\Hbf(\bA))h]}f\bigg| \ll Q(v)^{-\star}\Scal_d(f).
    \end{equation}

    We now show that $[\overline{\Tbf}(\A)] \subset [\overline{\Hbf}(\A)]$ is effectively equidistributed in an appropriate sense. 
    Here, we wish to invoke Appendix~\ref{sec:Duke}.
    To do so, we need to compare the different level structures used here and in Appendix~\ref{sec:Duke}.
    For $\overline{K}_f$ as in \eqref{eq:maximalcompact}, the level structure used here is defined with respect to the compact open subgroup $\overline{K}_f \cap \overline{\Hbf}(\A_f)$, which in view of the above complexity and volume estimates has index $\ll Q(v)^\star$ in a maximal compact open subgroup containing it.
    Note also that $\overline{\Hbf}\simeq \mathbf{PB}^\times$ for $B$ a definite quaternion algebra over $\Q$ whose discriminant is $\ll Q(v)^\star$ (cf.~\eqref{eq:complement of vector}).
    Let $\pi^+_\Hbf$ denote the orthogonal projection from $L^2(\overline{\Hbf}(\Q)\backslash\overline{\Hbf}(\bA))$ to the subspace of $\rho(\Hbf(\bA))$-invariant functions so that $f-\pi_\Hbf^+(f)\in L^2_{00}(\overline{\Hbf}(\Q)\backslash\overline{\Hbf}(\bA))$. 
    By the definition of $\pi_\Hbf^+$ we have
    \begin{equation} \label{eq:writing piH+ as Htilde period}
        \pi_\Hbf^+(f)(x) = \int_{[\rho(\Hbf(\bA))]}f(x\widetilde{h})\de \widetilde{h},\quad \forall x\in[\overline{\Hbf}(\bA)].
    \end{equation}
    
    By Theorem~\ref{thm:uniform Duke} we have
    \begin{equation}\label{eq:inner equidistribution}
        \bigg|\int_{[\overline{\Tbf}(\bA)]}\left( f-\pi_\Hbf^+(f) \right)\bigg| 
        \ll Q(v)^{\star}\disc(q)^{-\delta+\varepsilon}\Scal^H_{14}(f).
    \end{equation}
    Here, the Sobolev norm $\Scal^H_{10}(\cdot)$ is defined in \eqref{eq:sob H} and satisfies by the above index estimates $\Scal_{10}^H(f)\ll Q(v)^\star\Scal_{10}(f)$ for all $f$ smooth.
    Combining \eqref{eq:bounding Htilde period by disc}\eqref{eq:writing piH+ as Htilde period}, we have for every $x\in[\overline{\Hbf}(\bA)]$
    \begin{equation} \label{eq:outer equidistribution}
        \abs{\pi_\Hbf^+(f)(x)} \ll Q(v)^{-\star}\Scal_d(f).
    \end{equation}
    By \eqref{eq:inner equidistribution}\eqref{eq:outer equidistribution} we have
    \begin{align*}
        \bigg|\int_{[\overline{\Tbf}(\bA)]}f\bigg|  
        &\ll Q(v)^{\star}\disc(q)^{-\delta+\varepsilon}\Scal^H_{10}(f)
        +Q(v)^{-\star}\Scal_d(f)\\
        &\ll Q(v)^{\star}\disc(q)^{-\delta+\varepsilon}\Scal_{10}(f)+Q(v)^{-\star}\Scal_d(f).
    \end{align*}
    This proves the theorem.
\end{proof}

\subsection{Representing forms with small minimum}

Using Theorem~\ref{thm:two step equi}, we obtain the following.

\begin{theorem}\label{thm:small min}
There exist absolute constants $A>1$ and $\delta_0>0$ as well as constants $c_1,c_2$ depending on $Q$ with the following property.
Suppose that
\begin{align*}
c_1 \leq \min(q) \leq c_2\disc(q)^{\delta_0}
\end{align*}
and that $q$ is primitively represented by the spin genus of $Q$. Then
\begin{align*}
\Big| \frac{r(q,Q)}{r(q,\spn(Q))} - 1 \Big| \ll_Q \min(q)^{-1/A}.
\end{align*}
\end{theorem}

Here, $\delta_0$ can be taken to be $\kappa^2-\varepsilon$ for any $\varepsilon>0$ where $\kappa$ is as in Theorem~\ref{thm:two step equi}.
The qualitative version of Theorem~\ref{thm:small min} is readily obtainable through Proposition~\ref{prop:localglobalgiven?}\ref{item:localglobalgiven-equi2} and Theorem~\ref{thm:two step equi}.
We do not require the above effective version for Theorems~\ref{thm:main} and \ref{thm:main2}, but note that it follows from the arguments in \cite[\S2]{effectivesemisimple}.
In the remainder of the article, we will treat the case where $\min(q) \geq C_2\disc(q)^{\delta_0}$.

\section{Measure rigidity and reduction to entropy $>\frac{1}{2}$}\label{sec:measurerigidity}

In this section, we turn to rephrasing our main theorems in terms of entropy.
Let $p_1,p_2$ be distinct odd primes. 
In Proposition~\ref{prop:localglobalgiven?}, we already expressed our main theorems in terms of certain properties of homogeneous toral sets.
Let $(q_i)_i$ be a sequence of primitive binary quadratic forms with $-\disc(q_i)\disc(Q)$ a non-zero square modulo $p_1$ and $p_2$.
Let $(\iota_i)_i$ be a sequence of primitive representations of $q_i$ by $Q$ (without loss of generality; one may replace below $Q$ by an element of the spin genus of $Q$).
In view of Theorem~\ref{thm:small min}, we will henceforth additionally assume 
\begin{align*}
\min(q_i) \geq c_0|\disc(q_i)|^{\delta_0}.
\end{align*}
for some $c_0\in (0,\frac{1}{4})$ and $\delta_0\in (0,\frac{1}{2})$.

For each $i$, let $\mu_i=\mu_{\iota_i}$ denote the uniform probability measure on $Y_i = Y_{\iota_i}$ (see \eqref{eq:deftoralsets}).
We also recall that $X=\overline{\G}(\Q)\rho(\G(\A)) \subset \overline{X}$ and that $\mu_X$ denotes the uniform probability measure on $X$ (see \eqref{eq:localglobal-ambientspaces}). 
In view of Proposition~\ref{prop:localglobalgiven?}\ref{item:localglobalgiven-Haarcomp}, Theorems~\ref{thm:main} and \ref{thm:main2} are implied by the following:

\begin{theorem}\label{thm:Haarcomponent}
There exists an absolute constant $\delta>0$ so that any weak${}^\ast$-limit $\mu$ of the measures $\mu_i$ satisfies $\mu \geq \delta \mu_X$.
\end{theorem}

The goal of this section is to reduce Theorem~\ref{thm:Haarcomponent} to a statement about entropy of the partial limit of the measures $\mu_i$.
Here, the crucial ingredient is the work of Einsiedler and Lindenstrauss \cite{EL-nonmaximal,EL-joiningsPIHES} on measure rigidity for higher rank diagonalizable actions on (arithmetic) homogeneous spaces.

\subsection{Local conjugacy of tori}\label{sec:localconjugacy}
Note that the homogeneous sets $Y_i$ have, as given, no common invariance.
The general discussion in this subsection will address this issue.

Let $p$ be an odd prime with $p \nmid \disc(Q)$.
For quadratic lattices over $\Z_p$, a procedure mimicking Gram-Schmidt exists --- see e.g.~\cite[Ch.~8]{cassels}.
For instance, if $v\in \Z_p^4$ satisfies $p \nmid Q(v)$ then the orthogonal projection $w \mapsto w - \frac{\langle w,v\rangle}{Q(v)}v$ is well-defined and hence $\Z_p^4 = \Z_p v \oplus \{w \in \Z_p^4: w \perp v\}$.
Using this, one can show that $Q$ is $\Z_p$-equivalent to $x_1^2+x_2^2+x_3^2+\alpha x_4^2$ for some $\alpha \in \Z_p^\times$.
Similarly, we prove the following integral version of the Witt extension theorem:

\begin{lemma}\label{lem:orbits over Zp}
Let $q$ be a binary quadratic form over $\Z_p$ with $p \nmid \disc(q)$.
Let $\iota_1, \iota_2: \Z_p^2 \to \Z_p^4$ be two (automatically primitive) representations of $q$ by $Q$. Then there exists $k\in \SO_Q(\Z_p)$ with $k\iota_1 = \iota_2$.
\end{lemma}

\begin{proof}
Diagonalizing $q$ over $\Z_p$ we may assume that $q(x,y)= x^2+\beta y^2$ since $p \nmid \disc(q)$.
By the aforementioned version of Gram-Schmidt, we have for $i=1,2$
\begin{align}\label{eq:orthogonaldecomp}
\Z_p^4 = \iota_i(\Z_p^2) \oplus \{w\in \Z_p^4:w\perp\iota_i(\Z_p^2)\} 
=: \iota_i(\Z_p^2) \oplus \Lambda_i.
\end{align}
Explicitly, taking $v_1 = \iota_i(e_1)$ and $v_2 = \iota_i(e_2)$ one may consider the orthogonal projection $w \mapsto w - \frac{\langle w,v_1\rangle}{Q(v_1)}v_1 - \frac{\langle w,v_2\rangle}{Q(v_2)}v_2$.
Since $p \nmid \disc(Q)$, we have $p \nmid \disc(Q|_{\Lambda_i})$ by \eqref{eq:orthogonaldecomp}. Thus, $Q|_{\Lambda_i}$ is $\Z_p$-equivalent to $x^2+\beta_i y^2$ for some $\beta_i \in \Z_p^\times$ and, taking the discriminant in \eqref{eq:orthogonaldecomp}, $\beta\beta_2 \in \disc(Q)(\Z_p^\times)^2 = \beta\beta_1(\Z_p^\times)^2$. This shows that $\beta_1,\beta_2$ differ by a square and so $Q|_{\Lambda_1},Q|_{\Lambda_2}$ are $\Z_p$-equivalent.
This proves that there exists $k\in \mathrm{O}_Q(\Z_p)$ with $k\iota_1 = \iota_2$. If $u_1,u_2$ is an orthogonal basis of $\Lambda_2$, the map $u_1\mapsto u_1$, $u_2\mapsto-u_2$ preserves the quadratic form; post-composing $k$ with this orientation flip, if needed, yields the lemma.
\end{proof}

Recall that $Q$ is $\Z_p$-equivalent to $x_1^2+x_2^2+x_3^2+\alpha x_4^2$ for some $\alpha \in \Z_p^\times$.
We fix a binary quadratic form $q(x,y) = x^2+\beta y^2$ and a representation $\iota:\Z_p^2 \to \Z_p^4$ of $q$ by $Q$.
We assume that $-\disc(q)\disc(Q) = -\alpha \beta$ is a square in $\Z_p^\times$.

\subsubsection{Reducible case}\label{sec:diaggroup split}
Suppose that $\alpha$ is a square in $\Z_p^\times$.
Then $Q$ is in fact equivalent to the determinant form 
\begin{align*}
x_1x_4-x_2x_3 = \det\begin{pmatrix}
x_1 & x_2 \\ x_4 & x_4
\end{pmatrix}
\end{align*}
(e.g.~because the discriminants agree up to squares).
The action of $\SL_2 \times \SL_2$ on $2\times2$-matrices via $(g_1,g_2).A = g_1 A g_2^{-1}$ preserves the determinant and hence yields an isogeny $\SL_2\times \SL_2 \to \overline{\G}$ defined over $\Q_p$. The isogeny lifts to an isomorphism 
\begin{align}\label{eq:localisoSL2xSL2}
\SL_2\times \SL_2 \xrightarrow{\sim} \G.
\end{align}

Thus, consider $\iota_0:\Z_p^2 \to \Mat_2(\Z_p) \simeq \Z_p^4$ sending $(x,y)$ to $\mathrm{diag}(x,y)$. We let
\begin{align}\label{eq:stddiagonal for split}
A_p = \{g \in \G(\Q_p): g.\iota_0 = \iota_0\}
\end{align}
be the pointwise stabilizer group of $\iota_0$.
Under the isomorphism in \eqref{eq:localisoSL2xSL2} one may readily verify that
\begin{align*}
A_p \simeq 
\left\{ \left(\begin{pmatrix}
t & \\ & t^{-1}
\end{pmatrix},\begin{pmatrix}
t & \\ & t^{-1}
\end{pmatrix}\right): t\in \Q_p^\times\right\}.
\end{align*}
We fix
\begin{align}\label{eq:a_p-split}
\left(\begin{pmatrix}
p & \\ & p^{-1}
\end{pmatrix},\begin{pmatrix}
p & \\ & p^{-1}
\end{pmatrix}\right) \hat{=}\ a_p \in A_p.
\end{align}
The eigenvalues of $a_p$ acting on the Lie algebra of $\overline{\G}$ are $1,p^2,p^{-2}$, each with multiplicity $2$. 
 
Note that, since $\alpha$ is a square, so is $-\beta$ and, thus, $q$ is equivalent to the binary form $xy$.
By Lemma~\ref{lem:orbits over Zp} there exists $k_{\iota,p} \in \SO_Q(\Z_p)$ with $k_{\iota,p}.\iota_0 = \iota$ and, in particular,
\begin{align}\label{eq:conjugatingelement-splitcase}
k_{\iota,p}^{-1} \rho(\torus_\iota(\Q_p)) k_{\iota,p} = \rho(A_p).
\end{align}

\subsubsection{Irreducible case}\label{sec:diaggroup non-split}
Suppose now that $\alpha$ is not a square.
Then $\Q_p(\sqrt{\alpha})$ is the unramified quadratic extension of $\Q_p$ with ring of integers $\Z_p[\sqrt{\alpha}]$.
Consider now the space of Hermitian matrices over $\Q_p(\sqrt{\alpha})$
\begin{align*}
V= 
\left\{ 
\begin{pmatrix}
x_1 & x_2+\sqrt{\alpha}x_3 \\ x_2-\sqrt{\alpha}x_3 & x_4
\end{pmatrix}
: x_1,x_2,x_3,x_4 \in \Q_p\right\}
\end{align*}
and the $\Z_p$-lattice $\Lambda =V \cap \Mat_2(\Z_p[\alpha])$. The determinant on $\Lambda$ is, in the above coordinates, given by $x_1x_4 -x_2^2 + \alpha x_3^2$.
By determinant considerations, $Q$ is $\Z_p$-equivalent to this form and we may identify the quadratic lattices $(\Z_p^4,Q)$ and $(\Lambda,\det)$.

Note that $\SL_2(\Q_p(\sqrt{\alpha}))$ acts on $V$ via $g.v = gv\Bar{g}^t$ where $\Bar{g}$ denotes the Galois conjugate; the action preserves the determinant.
We obtain a homomorphism $\SL_2(\Q_p(\sqrt{\alpha})) \to \SO_Q(\Q_p)$.
More generally, we have an isogeny $ \Res_{\Q_p(\sqrt{\alpha})/\Q_p}(\SL_2) \to \SO_Q=\overline{\G} $ defined over $\Q_p$ and, by lifting, an isomorphism
\begin{align}\label{eq:localisoResSL2}
\Res_{\Q_p(\sqrt{\alpha})/\Q_p}(\SL_2) \to \G.
\end{align}

Let $\iota_0:\Z_p^2 \to \Lambda$ be given by  
\begin{align*}
\iota_0(x,y)=
\begin{pmatrix}
 & x+\sqrt{\alpha}y \\ x-\sqrt{\alpha}y & 
\end{pmatrix}.
\end{align*}
We let
\begin{align}\label{eq:stddiagonal for nonsplit}
A_p = \{g \in \G(\Q_p): g.\iota_0 = \iota_0\}
\end{align}
be the pointwise stabilizer group of $\iota_0$.
Under the isomorphism in \eqref{eq:localisoResSL2} one may readily verify that
\begin{align*}
A_p \simeq 
\left\{ \begin{pmatrix}
t & \\ & t^{-1}
\end{pmatrix}: t\in \Q_p^\times\right\}.
\end{align*}
We fix
\begin{align}\label{eq:a_p-nonsplit}
\begin{pmatrix}
p & \\ & p^{-1}
\end{pmatrix}\hat{=}\ a_p  \in A_p.
\end{align}
The eigenvalues of $a_p$ acting on the Lie algebra of $\overline{\G}$ are $1,p^2,p^{-2}$, each with multiplicity $2$. 

Since $\alpha$ and $-\beta$ differ by a square by assumption, $q$ is equivalent to $x^2-\alpha y^2$ or to $-x^2+\alpha y^2$, which is the quadratic form on $\iota_0(\Z_p^2)$.
By Lemma~\ref{lem:orbits over Zp} there exists $k_{\iota,p} \in \SO_Q(\Z_p)$ with $k_{\iota,p}.\iota_0 = \iota$ and, in particular,
\begin{align}\label{eq:conjugatingelement-nonsplitcase}
k_{\iota,p}^{-1} \rho(\torus_\iota(\Q_p)) k_{\iota,p} = \rho(A_p).
\end{align}

\subsection{Lower bounds on the entropy of limit measures}

We now phrase a theorem for the self-correlation of the homogeneous toral sets $Y_\iota$ or rather shifted versions thereof (to ensure common invariance, we use the discussions of \S\ref{sec:localconjugacy}).
We fix an odd prime $p$ and write $a=\rho(a_p)$ for simplicity where $a_p$ was defined in \S\ref{sec:localconjugacy}.
Define
\begin{align*}
\Bow_p(n) = \bigcap_{-n\leq i \leq n} a^i\{k \in\overline{\G}(\Z_p): k \equiv \id \mod p\}a^{-i}
\end{align*}
and the two-sided Bowen ball
\begin{align*}
\Bow(n) = \overline{\G}(\R) \times \prod_{q \neq p}\overline{\G}(\Z_q) \times \Bow_p(n).
\end{align*}
Here, $\overline{\G}(\Z_q) = \{g \in \overline{\G}(\Q_p):g\Z_q^4 =\Z_q^4\}$.
For any $x \in X$, the points in $x\Bow(n)$ can be thought of as the points having comparable trajectory to $x$ for times between $-n$ and $n$.
Indeed, if $y \in x\Bow(n)$ then $ya^i \in xa^i \overline{\G}(\R\times \hat{\Z})$ or, more informally, $ya^i$ is at distance $\ll 1$ from $xa^i$ for all $-n \leq i\leq n$.

\begin{theorem}[Linnik-type basic lemma]\label{thm:selfcorrelation}
For any $\delta_0>0$ there exists $\delta_1 >0$ (depending only on $\delta_0$) with the following property.

Let $p$ be an odd prime and $c_0>0$.
Let $q(x,y)$ be a primitive integral binary quadratic form with the following properties:
\begin{itemize}
    \item $-\disc(q)\disc(Q)$ is a non-zero square modulo $p$.
    \item $\min(q) \geq c_0|\disc(q)|^{\delta_0}$.
\end{itemize}
Let $\iota$ be a primitive representation of $q$ by $Q$ and let $Y_\iota' = Y_{\iota}k_{\iota,p}$ where $k_{\iota,p}$ is as in \eqref{eq:conjugatingelement-splitcase} resp.~\eqref{eq:conjugatingelement-nonsplitcase}.
Then $Y_\iota'$ is $a$-invariant and the uniform measure $\mu_{\iota}'$ on $Y_\iota'$ satisfies
\begin{align*}
\mu_\iota' \times \mu_\iota'(\{(x,y): y \in x \Bow(n)\})
\ll_{c_0,\delta_0,\delta} p^{-(4+\delta)n}
\end{align*}
for some $n\in \N$ with $|\disc(q)|^{\frac{1}{10}} \leq p^n \leq |\disc(q)|^{\frac{1}{8}}$ and any $\delta<\delta_1$. 
\end{theorem}

The main goal for the rest of the article will be to prove Theorem~\ref{thm:selfcorrelation}.
In the remainder of this section, we will see how Theorem~\ref{thm:selfcorrelation} implies Theorem~\ref{thm:Haarcomponent} (and, thus, also the main theorems).
It is a standard fact that self-correlation bounds as in Theorem~\ref{thm:selfcorrelation} imply lower bounds on entropy --- see for example \cite[\S4.2]{ELMV-Ens} or \cite[Prop.~3.2]{ELMV-DukeJ}.
The following corollary proves what we need.

\begin{cor}\label{cor:entropy>1/2}
Let $(q_i),(\iota_i)$ be a sequence satisfying the assumptions of Theorem~\ref{thm:selfcorrelation} and with $\disc(q_i) \to \infty$.
Then any weak${}^\ast$-limit $\mu'$ of the measures $\mu_{\iota_i}'$ satisfies
\begin{align*}
h_{\mu'}(a) \geq (2+\tfrac{1}{2}\delta_1)\log(p) = (1+\tfrac{1}{4}\delta_1) \tfrac{1}{2} h_{\mu_X}(a).
\end{align*}
Here, $h_{\nu}(a)$ denotes the measure-theoretic (Kolmogorov-Sinai) entropy of an $a$-invariant probability measure $\nu$.
\end{cor}

\begin{proof}[Proof of Corollary~\ref{cor:entropy>1/2}]
Let $\mathcal{P}$ be a finite measurable partition of $X$ such that
\begin{align*}
[x]_{\mathcal{P}_{-n}^n} \subset x \mathrm{Bow}(n)
\end{align*}
where $\mathcal{P}_{-n}^n$ denotes the refinement of $\mathcal{P}$ for times between $-n$ and $n$. That is, $x,y$ belong to the same partition element of $\mathcal{P}_{-n}^n$ if $xa^i,ya^i$ belong to the same partition element of $\mathcal{P}$ for all $-n\leq i \leq n$.
We note that constructing such a partition is simple in our context; one may take $\mathcal{P}$ to consist of orbits of a sufficiently small open subgroup $\Omega = \prod_v \Omega_v \subset \rho(\G(\A))$ where $\Omega_\infty=\rho(\G(\R))$.
In particular, we may assume that the partition elements of $\mathcal{P}_{-n}^n$ are open and closed.

To simplify notation, write $\mu_i' = \mu_{\iota_i}'$ and assume without loss of generality that $\mu_i' \to \mu'$.
Let $n_i$ be as in Theorem~\ref{thm:selfcorrelation} applied to $\iota_i$ so that for $\delta<\delta_1$
\begin{align}\label{eq:cor-boundcorrelation}
\mu_i' \times \mu_i'(\{(x,y): y \in [x]_{\mathcal{P}_{-n_i}^{n_i}}) \leq C p^{-(4+\delta)n_i}
\end{align}
for some $C=C(\delta)>0$ and $n_i \to \infty$.
By convexity of the logarithm, 
\begin{align*}
H_{\mu_i'}(\mathcal{P}_{-n_i}^{n_i})
= - \sum_{P\in \mathcal{P}_{-n_i}^{n_i}} \mu_i'(P) \log(\mu_i'(P))
\geq -\log \Big( \sum_{P\in \mathcal{P}_{-n_i}^{n_i}} \mu_i'(P)^2\Big).
\end{align*}
Moreover, for any partition element $P\in \mathcal{P}_{-n_i}^{n_i}$ we have that $x,y \in P$ is equivalent to $x \in P$ and $y \in [x]_{\mathcal{P}_{-n_i}^{n_i}}$ and so
\begin{align*}
\sum_{P\in \mathcal{P}_{-n_i}^{n_i}} \mu_i'(P)^2
= \sum_{P\in \mathcal{P}_{-n_i}^{n_i}} \mu_i' \times \mu_i'(P \times P)
= \mu_i' \times \mu_i'\big(\{(x,y): y \in [x]_{\mathcal{P}_{-n_i}^{n_i}}\big).
\end{align*}
Therefore, 
\begin{align*}
H_{\mu_i'}(\mathcal{P}_{-n_i}^{n_i})
\geq -\log\Big(\mu_i' \times \mu_i'\big(\{(x,y): y \in [x]_{\mathcal{P}_{-n_i}^{n_i}}\big)\Big)
\end{align*}
and, since $-\log$ is monotonely decreasing, \eqref{eq:cor-boundcorrelation} implies
\begin{align}\label{eq:cor-partitioning into smaller int}
H_{\mu_i'}(\mathcal{P}_{-n_i}^{n_i})
\geq (4+\delta)\log(p) n_i - \log(C).
\end{align}

Let $n \geq 1$. If $n_i \geq n$, we may partition the interval $[-n_i,n_i]$ into intervals of length $n$ together with a single interval of length at most $n$.
As $\mu_i'$ is $a$-invariant, we have $H_{\mu_i'}(\mathcal{P}_{\ell}^{\ell+n-1}) = H_{\mu_i'}(\mathcal{P}_{0}^{n-1})$ for any $\ell \in \Z$. 
This together with subadditivity implies
\begin{align*}
H_{\mu_i'}(\mathcal{P}_{-n_i}^{n_i}) \leq \lfloor \tfrac{2n_i+1}{n}\rfloor H_{\mu_i'}(\mathcal{P}_{0}^{n-1}) + n H_{\mu_i'}(\mathcal{P}).
\end{align*}
Inserting into \eqref{eq:cor-partitioning into smaller int} we obtain
\begin{align*}
H_{\mu_i'}(\mathcal{P}_{0}^{n-1}) \geq 
(4+\delta)\log(p) \tfrac{n_i}{2n_i+1} n - \tfrac{1}{2n_i+1} \big(n\log(C)+n^2H_{\mu_i'}(\mathcal{P})\big).
\end{align*}
Note that $H_{\mu_i'}(\mathcal{P})$ is bounded from above by $\log$ of the number of partition elements in $\mathcal{P}$. Using additionally that the partition elements of $\mathcal{P}_{0}^{n-1}$ are open and closed, we may take $i \to \infty$ to obtain
\begin{align*}
H_{\mu'}(\mathcal{P}_{0}^{n-1}) \geq 
\tfrac{1}{2}(4+\delta)\log(p)n.
\end{align*}
Finally, letting $\delta\to \delta_1$ this proves the corollary.
\end{proof}

\subsection{Setup for the proof of Theorem~\ref{thm:Haarcomponent}}\label{sec:setup measureclassification}
In the rest of this section, we shall deduce Theorem \ref{thm:Haarcomponent} from the entropy lower bound in Corollary \ref{cor:entropy>1/2}, in combination with measure classification theorems of Einsiedler and Lindenstrauss \cite{EL-joiningsPIHES,EL-nonmaximal}.
We use the notation and the assumptions in Theorem~\ref{thm:Haarcomponent}.
For each $\iota_i$, let $k_{\iota_i,p_1}\in \overline{\G}(\Z_{p_1})$ and $k_{\iota_i,p_2}\in \overline{\G}(\Z_{p_2})$ be elements as in \eqref{eq:conjugatingelement-splitcase} resp.~\eqref{eq:conjugatingelement-nonsplitcase}.
After switching to a subsequence, we may further assume that there exist $k_1 \in \overline{\G}(\Z_{p_1})$ and $k_2 \in \overline{\G}(\Z_{p_2})$ such that $\rho(\G(\Q_{p_j}))k_{\iota_i,p_j} = \rho(\G(\Q_{p_j}))k_{j}$ for $j=1,2$.
Here, we merely used that $\rho(\G(\Q_{p_j}))$ has finite index in $\overline{\G}(\Q_{p_j})$.
Let $g_i \in \G(\A)$ with $\rho(g_i) = k_{\iota_i,p_1}k_{\iota_i,p_2}(k_1k_2)^{-1}$.

Define
\begin{align*}
Y_i'    &= Y_{\iota_i} \rho(g_i) \subset X = \overline{\G}(\Q)\rho(\G(\A)),\\
\tilde{Y}_i &= \G(\Q) \torus_{\iota_i}(\A) g_i \subset \G(\Q)\backslash \G(\A) = \tilde{X}.
\end{align*}
By definition, $\rho(\tilde{Y}_i) = Y_i'$ where we write by abuse of notation $\rho:\tilde{X}\to X$ for the map induced by the isogeny $\rho:\G \to \overline{\G}$.
We let $\mu_i'$ resp.~$\tilde{\mu}_i$ be the uniform probability measure on $Y_i'$ resp.~$\tilde{Y}_i$ and note that $\rho_\ast \tilde{\mu}_i = \mu_i'$.
In view of the statement of Theorem~\ref{thm:Haarcomponent}, we may assume that $\tilde{\mu}_i \to \tilde{\mu}$ and
$\mu_i' \to \mu' = \rho_\ast \tilde{\mu}$ in the weak${}^\ast$-topology and show that $\mu' \geq \delta \mu_{X}$ for some $\delta>0$.

By definition of the bounded shifts $g_i$, the measures $\mu_i'$ are invariant under $k_1\rho(A_{p_1})k_1^{-1}$ and $k_2\rho(A_{p_2})k_2^{-1}$ and therefore also under $k_1 \rho(a_{p_1})k_1^{-1},k_2\rho(a_{p_2})k_2^{-1}$ where $a_{p_1},a_{p_2}$ are as in \eqref{eq:a_p-split} resp.~\eqref{eq:a_p-nonsplit}.
For $j\in {1,2}$, since $\rho(\G(\Q_{p_j})) < \overline{\G}(\Q_{p_j})$ is a normal subgroup, $k_j \rho(a_{p_j})k_j^{-1} = \rho(b_j)$ for some $b_j \in \G(\Q_{p_j})$.
Necessarily, the measures $\tilde{\mu}_i$ are invariant under $b_1,b_2$.
For simplicity, we let $A$ be the rank two free abelian group generated by $b_1,b_2$.

\subsubsection{Entropy bounds}
Observe that the sets $Y_i'$ are slightly modified from Theorem~\ref{thm:selfcorrelation} for $p=p_1$ as they are shifted in the two places $p_1,p_2$. Theorem~\ref{thm:selfcorrelation} in the form of Corollary~\ref{cor:entropy>1/2} applies regardless and we conclude that
\begin{align}\label{eq:ap1 entropy}
h_{\mu'}(\rho(b_1)) \geq (1+\tfrac{1}{4}\delta_1) \tfrac{1}{2} h_{\mu_X}(\rho(b_1)).
\end{align}
The same statement is true for $b_2$ and hence for any element in $A$.

Note that $h_{\mu_X}(\rho(a)) = h_{\mu_{\tilde{X}}}(a)$ for any $a\in A$ where $\mu_{\tilde{X}}$ is the uniform measure on $\tilde{X}$.
By the Abramov-Rokhlin formula, we thus also have for any $a\in A$
\begin{align}\label{eq:ap1 entropy sc}
h_{\tilde{\mu}}(a) \geq (1+\tfrac{1}{4}\delta_1) \tfrac{1}{2} h_{\mu_{\tilde{X}}}(a).
\end{align}
We will use this to show $\tilde{\mu} \geq \tfrac{1}{8}\delta_1 \mu_{\tilde{X}}$, which implies the theorem with $\delta= \tfrac{1}{8}\delta_1$.
We remark that the above passage to the simply connected cover will allow us to directly invoke the rigidity results of Einsiedler and Lindenstrauss \cite{EL-nonmaximal,EL-joiningsPIHES}.

\subsubsection{Passage to an $S$-arithmetic quotient}
In what follows, we consider projections of the given adelic measures onto $S$-arithmetic quotients. 
Let $S$ be a set of places of $\Q$ containing $\{\infty,p_1,p_2\}$ and let (omitting $\tilde{\cdot}$ for simplicity)
\begin{align*}
X^S:=\G(\mathbb{Z}^S)\backslash \G(\mathbb{{Q}}_S).
\end{align*}
Write $\pi^S$ for the projection from $\tilde{X}$ onto $X^S$ (which exists by strong approximation with respect to $S$, see \cite[Ch.~7]{platonov} or \cite{Rapinchuk-survey}). 
Then the projection $\pi^S_*\tilde{\mu}$ is an $A$-invariant probability measure on $X^S$.
By the Abramov-Rokhlin formula and \eqref{eq:ap1 entropy} we have for any $a \in A$
\begin{align}\label{eq:entropy in S factor}
h_{\pi^S_*\tilde{\mu}}(a) \geq (1+\tfrac{1}{4}\delta_1) \tfrac{1}{2} h_{\mu_{X^S}}(a).
\end{align}

\subsubsection{Ergodic decomposition}
For the $A$-invariant probability measure $\pi^S_*\tilde{\mu}$, we have the following ergodic decomposition:
\begin{equation}\label{eq:ergodicdecomposition}
    \pi^S_*\tilde{\mu}=\int_{X^S} \nu(\xi)\de\pi^S_*\tilde{\mu}(\xi),
\end{equation}
where $\nu(\xi)$ is an ergodic $A$-invariant probability measure for $\pi^S_*\tilde{\mu}$-a.e.~$\xi$. It is a standard fact that the entropy of an invariant measure is the average of the entropies of its ergodic components: for any $a\in A$,
\begin{equation}\label{eq:entropydecomposition}
    h_{\pi^S_*\tilde{\mu}}(a)=\int h_{\nu(\xi)}(a)\de\pi^S_*\tilde{\mu}(\xi).
\end{equation}

\subsection{Irreducible case}
We first consider the case where $\G$ is $\mathbb{Q}$-almost simple. 
Recall from \S\ref{sec:notation qf} that, equivalently, $\disc(Q)$ is not a square and hence this case should be considered `typical'.
We apply the following measure classification theorem of Einsiedler-Lindenstrauss for non-maximal torus actions on irreducible quotients of $\SL_2$'s.

\begin{theorem}[Special case of {\cite[Thm.~1.3]{EL-nonmaximal}}]\label{thm:nonmaximaltorusrigidity}
    Let $\Mbf$ be an algebraic group over $\Q$ that is $\Q$-almost simple and a form of $\SL_2^k$ with $k\geq 1$, and $\Gamma$ be an arithmetic lattice in $M=\Mbf(\Q_S)$ arising from $\Mbf$. Suppose $X=\Gamma \backslash M$ is compact and let $A<M$ be a closed abelian class-$\mathcal{A}'$ subgroup of higher rank. Let $\nu$ be an $A$-invariant and ergodic probability measure on $X$ such that $h_{\nu}(a)>0$ for some $a\in A$. Then the measure $\nu$ is \emph{algebraic}. 
    More precisely, there exists a connected semisimple subgroup $\Lbf < \Mbf$ defined over $\Q$, a finite index subgroup $L\leq \mathbf{L}(\Q_S)$, and some $g\in M$ so that $\nu$ is the uniform measure on $\Gamma Lg $ and so that $A<g^{-1}\mathbf{L}(\Q_S)g$.
\end{theorem}

\begin{proof}[Proof of Theorem \ref{thm:Haarcomponent} assuming Corollary \ref{cor:entropy>1/2} -- the irreducible case]

 We apply Theorem~\ref{thm:nonmaximaltorusrigidity} for $X^S=\G(\mathbb{Z}^S)\backslash\G(\mathbb{{Q}}_S)$ and $A$ the group defined above. 
 Let $\nu(\xi)$ be an ergodic component in \eqref{eq:ergodicdecomposition}. If $h_{\nu(\xi)}(a)>0$, then by Theorem \ref{thm:nonmaximaltorusrigidity} there exists a semisimple $\mathbb{Q}$-subgroup $\mathbf{L}_{\xi}<\G$ and a finite index subgroup $L_{\xi}<\mathbf{L}_{\xi}(\mathbb{Q}_S)$ such that $\nu(\xi)$ is the normalized Haar measure on a single orbit $\G(\mathbb{Z}^S)L_\xi g_\xi$ for some $g_\xi\in \G(\mathbb{Q}_S)$. Moreover, $\mathbf{L}_{\xi}$ is semisimple, $A<g_\xi^{-1}\mathbf{L}_{\xi}(\mathbb{Q}_S)g_\xi$, and $L_{\xi}=\mathbf{L}_{\xi}(\mathbb{Q}_S)$ (as necessarily $\mathbf{L}_{\xi}$ is simply connected). 

If $\Lbf_\xi$ is proper, then by mere dimension considerations it must be isogenous to a form of $\SL_2$.
Moreover, since $g_\xi^{-1}\Lbf_\xi(\Q_{p_1})g_\xi$ contains $b_1$, the Lie algebra of $\Ad(g_\xi)\mathrm{Lie}(\Lbf_{\xi})$ over $\Q_{p_1}$ is a sum of eigenspaces of $\Ad(b_1)$, necessarily with simple eigenvalues $1,p_1^2,p_1^{-2}$.
Thus, for such $\xi$ we have $h_{\nu(\xi)}(a)=\frac{1}{2}h_{\mu_{X^S}}(a)$ for all $a \in A$.

If $\Lbf_\xi= \G$, then $\nu(\xi)=\mu_{X^S}$.
 
To summarize, there are the following three scenarios:
\begin{enumerate}
    \item $\mathbf{L}_{\xi}=\G$ and $\nu(\xi)=\mu_{X^S}$,
    \item $\mathbf{L}_{\xi}$ is a $\Q$-form of $\SL_2$ and $h_{\nu(\xi)}(a)=\frac{1}{2}h_{\mu_{X^S}}(a)$ for all $a\in A$,
    \item $h_{\nu(\xi)}(a)=0$ for all $a\in A$.
\end{enumerate}
Let $\Xi_1,\Xi_2,\Xi_3\subseteq X^S$ denote the sets of $\xi$ satisfying (1),(2),(3), respectively. Then Corollary \ref{cor:entropy>1/2} (in the form of \eqref{eq:entropy in S factor}) and \eqref{eq:entropydecomposition} imply
\begin{align*}
\pi_*^S\tilde{\mu}(\Xi_1)\cdot h_{\mu_{X^S}}(a)+\pi_*^S\tilde{\mu}(\Xi_2)\cdot \tfrac{1}{2}h_{\mu_{X^S}}(a)\geq (1+\tfrac{1}{4}\delta_1) \tfrac{1}{2} h_{\mu_{X^S}}(a).
\end{align*}
It follows that $\pi_*^S\tilde{\mu}(\Xi_1)\geq \tfrac{1}{8}\delta_1$, hence $\pi_*^S\tilde{\mu}\geq \pi_*^S\tilde{\mu}(\Xi_1)\mu_{X^S}\geq \tfrac{1}{8}\delta_1\mu_{X^S}$. 

Since any continuous function on $X$ factors through an $S$-arithmetic quotient, the above shows that $\tilde{\mu} \geq \tfrac{1}{8}\delta_1 \mu_{\tilde{X}}$.
The theorem follows with $\delta=\tfrac{1}{8}\delta_1$.
\end{proof}

\subsection{Reducible case}
We now consider the case where $\G$ is not $\mathbb{Q}$-almost simple.
In this case, $\G$ is isomorphic over $\Q$ to $\SL_1(\Bbf) \times \SL_1(\Bbf)$ for some quaternion algebra $\mathbf{B}$. 
We write $\G_1,\G_2$ for the factors of $\G$ (both isomorphic to $\SL_1(\Bbf)$) and $\pi_i: \G\to \G_i$ for the projections for $i=1,2$.
We remark that $\SL_1(\Bbf)$ is equivalently the spin group for the norm form of $\Bbf$ restricted to the subspace $\Bbf^0$ of traceless quaternions in~$\Bbf$ (indeed, $\SL_1(\Bbf)$ acts by conjugation on $\Bbf^0$ preserving the norm).

\begin{lemma}\label{lem:applying AEW}
Suppose that $\iota:\Z^2 \to \Z^4$ is a primitive representation of a primitive integral binary quadratic form $q$ by $Q$. Denote by $\torus_\iota < \G$ as in \eqref{eq:stabtorus} the pointwise stabilizer. Then $\pi_i(\torus_\iota) < \G_i$ is the stabilizer of a rational line in $\Bbf^0$ of height $\gg |\disc(q)|$.
\end{lemma}

Here, the notion of height is given as follows:
Fix a maximal order $\mathcal{O} \subset \Bbf(\Q)$. Given a rational line in $\Bbf^0(\Q)$ let $v$ be a non-zero shortest vector in the line which is contained in $\mathcal{O}^0$ and then take the norm of $v$.

\begin{proof}
It is clear by the considerations e.g.~in \S\ref{sec:diaggroup split} that $\torus_\iota$ projects non-trivially onto both factors. Moreover, any non-trivial torus in $\SL_1(\Bbf)$ is the stabilizer of a rational line in $\Bbf^0$ (as the conjugation action on $\Bbf^0$ can be identified with the adjoint action on the Lie algebra).
Therefore, we only need to verify the claim regarding the height.
To that end, we will make heavy use of arguments in \cite{AEW-2in4}.

By our assumption, $Q$ is rationally equivalent to a multiple of the norm form on $\Bbf$. 
Integrally, this means that $Q$ is equivalent to a multiple of the norm form when restricted to a full-rank sublattice $\Lambda$ of $\Bbf(\Q)$. 
Note that $\Lambda$ is commensurable to $\mathcal{O}$.
Therefore, we may view $\iota(\Z^2)$ is a rank two lattice in $\Bbf(\Q)$ for which there is a $2$-dimensional subspace $W$ such that $\iota(\Z^2) \subset W$ and $\iota(\Z^2)$ is commensurable to $W \cap\mathcal{O}$.
In particular, $\tilde{q} = \Nr|_{W \cap\mathcal{O}}$ satisfies that the greatest common divisor of its coefficients (in any basis) is $\ll 1$.
Also, $|\disc(\tilde{q})| \gg |\disc(q)|$.
Moreover,
\begin{align*}
\torus_\iota = \{(g_1,g_2) \in \SL_1(\Bbf)^2: g_1wg_2^{-1} =w \text{ for all } w\in W\}.
\end{align*}

Let $w_1,w_2$ be a basis of $W \cap\mathcal{O}$.
Following \cite{AEW-2in4} we define
\begin{align*}
a_1 = 2w_1 \overline{w_2} - \Tr(w_1 \overline{w_2}),\quad
a_2 = 2 \overline{w_2}w_1 - \Tr( \overline{w_2}w_1).
\end{align*}
By straight-forward calculations (e.g.~as in \cite[\S2]{AEW-2in4}) one shows that $a_1,a_2$ are traceless vectors in $\mathcal{O}$ which are fixed under $\pi_1(\torus_{\iota})$ and $\pi_2(\torus_{\iota})$, respectively, and which satisfy $\Nr(a_1)=\Nr(a_2) = 4\disc(\tilde{q})$.

It remains to examine primitivity properties of $a_1,a_2$. 
For any odd prime $p$ at which $\Bbf$ is unramified, we have by \cite[\S3.3]{AEW-2in4} that
\begin{align*}
\ord_p(\tilde{q}) = \max\{\ord_p(a_1),\ord_p(a_2)\}
\end{align*}
where $\ord_p(a_i) = \max\{j: p^{-j}a_i \in \mathcal{O}\}$.
If $\Bbf$ is unramified at $2$, one can show analogously that $|\ord_2(\tilde{q}) - \max\{\ord_2(a_1),\ord_2(a_2)\}|$ is bounded by an absolute constant.
If $\Bbf$ is ramified at a prime $p$, then
\begin{align*}
\mathcal{O}\otimes \Z_p = \{x \in \Bbf(\Q_p):\Nr(x) \in \Z_p\}
\end{align*}
is the unique maximal order, see e.g.~\cite[Prop.~13.3.4]{voight}.
In particular, $a \in \mathcal{O}\otimes \Z_p$ is not invertible if and only if $p \mid \Nr(a)$.
Applying this to $a=a_1,a_2$ we see that either $\ord_p(a_1)=\ord_p(a_2)=0$ or $\ord_p(a_1)=\ord_p(a_2)>0$. On the other hand, the linear map
\begin{align*}
w_1\wedge w_2 \in \textnormal{$\bigwedge$}^2\mathcal{O} \mapsto (a_1,a_2) \in (\Bbf^0 \cap \mathcal{O})^2 
\end{align*}
has trivial kernel and hence $\min\{\ord_p(a_1),\ord_p(a_2)\} \ll \ord_p(w_1\wedge w_2)$ (see also \cite[Lemma 2.4]{AEW-2in4}).
Overall, these arguments show that the lines through $a_1$ and $a_2$ have height $\gg \Nr(a_1)=\Nr(a_2) \gg |\disc(\tilde{q})| \gg |\disc(q)|$ which concludes the lemma.
\end{proof}

The following proposition, which is an application of a version of Duke's theorem and its generalizations, is phrased in the simply connected cover.
Recall that we fixed a sequence of primitive representations $(\iota_i)_i$ of $q_i$ by $Q$.

\begin{proposition}\label{prop:joining}
Let $p$ be an odd prime.
Let $(g_i)_i$ be a bounded sequence of elements in $\G(\A_f)$ such that $g_{i,p}\torus_{\iota_i}g_{i,p}^{-1}$ is a torus independent of $p$ and split over $\Q_p$.
For each $i$, let $\nu_i$ be the uniform probability measure on $\G(\Q) \torus_{\iota_i}(\A)g_i \subset \G(\Q)\backslash\G(\A)$. 
Then for $j=1,2$ the measures $(\pi_j)_\ast \nu_i$ are equidistributed in the factor $\G_j(\Q)\backslash \G_j(\A)$ with respect to the Haar measure.
\end{proposition}

Without the given splitting condition at the prime $p$, the above proposition is certainly folklore, though to the authors' knowledge not verbatim contained in the existing literature.
With the splitting condition, we can invoke refinements of Linnik's ergodic method \cite{linnik} (see \cite{ELMV-Ens,W-Linnik}).

\begin{proof}
Fix $j \in \{1,2\}$ and let $\torus_{i,j} = \pi_j(\torus_{\iota_i})$ as well as $g_{i,j} = \pi_j(g_i)$.
The measure $(\pi_j)_\ast \nu_i$ is the uniform probability measure on the period $\G_j(\Q)\torus_{i,j}(\A)g_{i,j}$ in $\G_j(\Q) \backslash \G_j(\A)$.
By Lemma~\ref{lem:applying AEW}, the torus $\torus_{i,j}$ is the stabilizer group of a rational line whose height goes to infinity as $i \to \infty$.
The claim follows from \cite[\S7]{W-Linnik}.
\end{proof}

\begin{proof}[Proof of Theorem~\ref{thm:Haarcomponent} assuming Corollary \ref{cor:entropy>1/2} -- the reducible case]

We proceed as in the previous case, though we appeal to the classification of joinings of Einsiedler and Lindenstrauss \cite{EL-joiningsPIHES}.
By Proposition~\ref{prop:joining} applied to \S\ref{sec:setup measureclassification}, the measure $\pi^S_\ast \tilde{\mu}$ is a convex combination of joinings $\nu(\xi)$ for the respective Haar measures on the factors of $X^S$.
Each of these measures $\nu(\xi)$ is either (a) the Haar probability measure on $X^S$ or (b) there exists a proper semisimple $\mathbb{Q}$-subgroup $\mathbf{L}_{\xi}<\G$ and a finite index subgroup $L_{\xi}<\mathbf{L}_{\xi}(\mathbb{Q}_S)$ such that $\nu(\xi)$ is the  Haar probability measure on a single orbit $\G(\mathbb{Z}^S)Lg_\xi$ for some $g_\xi\in \G(\mathbb{Q}_S)$. 
Moreover, in case (b) $\mathbf{L}_{\xi}$ projects surjectively on each of the factors $\G_1,\G_2$ of $\G$ and $A<g_\xi^{-1}\mathbf{L}_{\xi}(\mathbb{Q}_S)g_\xi$. 
The entropy of $\nu(\xi)$ in case (b) is half of the entropy of the Haar probability measure on $X^S$ and therefore one can conclude the proof as before (though we note that there are no potential zero entropy ergodic components).
\end{proof}

\section{Reducing the Linnik-type basic lemma to a point counting problem}\label{sec:reduce to counting}

This section marks the beginning of our proof of Theorem~\ref{thm:selfcorrelation}.
Here, we will reduce Theorem~\ref{thm:selfcorrelation} to a point counting problem for a certain hypersurface in $\mathbf{A}^5$ of degree four.

Throughout the section we fix an odd prime $p$.
We also adhere to the notations and assumptions of Theorem~\ref{thm:selfcorrelation}.
In particular, $q$ denotes a primitive integral binary quadratic form of discriminant $D$.
Moreover, we assume that $q(x,y) = Ax^2+Bxy+ Cy^2$ is Minkowski reduced and $A \geq c_0 D^{\delta_0}$.
Note that $D = AC-\frac{B^2}{4}>0$ and that $p \nmid D$ by assumption.
For simplicity, we will ignore $c_0,\delta_0$ in the implicit (multiplicative) constants henceforth.

\subsection{Reduction to counting nearby pairs of representations}

We begin with a volume estimate for the periods $Y_\iota$. We refer to e.g.~\cite[\S4.3]{ELMVAnn} for the notion of volume we use here.

\begin{lemma}[Volume estimates]\label{lem:volume}
Let $\iota:\Z^2 \to \Z^4$ be a primitive representation of $q$ by $Q$. Then
\begin{align*}
\vol(Y_\iota')=\vol(Y_\iota) \asymp |D|^{\frac{1}{2}+o(1)}.
\end{align*}
\end{lemma}

\begin{proof}
The first equality is clear and we prove the growth rate for $\vol(Y_\iota)$.
For simplicity, write $\torus = \torus_{\iota}$ and $Y = Y_{\iota}$. 
In estimating the volume, the main step consists of determining the compact open subgroups $\torus(\Z_{\ell}) = \{t\in \torus(\Q_{\ell}): t.\Z_{\ell}^4 =\Z_{\ell}^4\}$ and $\overline{\torus}(\Z_{\ell}) = \{t\in \overline{\torus}(\Q_{\ell}): t.\Z_{\ell}^4 =\Z_{\ell}^4\}$ for any prime $\ell$.

To that end, we will use the Clifford algebra of $Q$. 
We refer to \cite{Knus} for a thorough treatment and limit ourselves to introducing what we need.
The Clifford algebra $\mathcal{C}$ for $Q$ can be defined to be the tensor algebra over $\Q^4$ modulo the ideal generated by the relations $v^2 = Q(v)$ for $v \in V=\Q^4$.
The vector space $V$ is naturally contained in $\mathcal{C}$ and for any $v,w\in V$ we have $vw+wv = 2\langle v,w\rangle_Q$.
By definition, the even Clifford algebra $\mathcal{C}^e$ consists of even tensors in $\mathcal{C}$.
Note that the subspace $W = \iota(\Q^2)^\perp$ has its own Clifford algebra (for the restriction of $Q$ to $W$), whose even part forms a quadratic subalgebra $K \subset \mathcal{C}^e$; it is generated by $1$ and $w_1w_2$ for $w_1,w_2 \in W$.
We write $\sigma: \mathcal{C}^e \to \mathcal{C}^e,\ vw \mapsto wv$ for the standard involution, which restricts to the non-trivial Galois automorphism on $K$. 
Set $\Nr(v) = v \sigma(v)$.
We note that these constructions are well-behaved with respect to field extensions.
The group $\G$ can be seen as the group of norm one units in the even Clifford algebra (where we use that $Q$ has four variables) and the action on $V$ is by conjugation within the Clifford algebra.
Similarly, there are identifications $\Res_{K/ \Q}(\Gm)^1 \simeq \torus$ and $\Res_{K/ \Q}(\Gm)/\Gm \simeq  \overline{\torus}$, which will be implicit in the following.

Let $w_1,w_2$ be an integral basis of $W\cap \Z^4$. 

\begin{claim*}
Let $\ell$ be a prime. Then
\begin{align*}
\torus(\Z_{\ell}) &\supseteq \big\{t = t_1+ t_2 w_1w_2: t_1,t_2 \in \Z_{\ell},\ \Nr(t) =1 \big\},\\
\overline{\torus}(\Z_{\ell}) &\supseteq \big\{t = t_1+ t_2 w_1w_2: t_1,t_2 \in \Z_{\ell},\  \Nr(t) \in \Z_{\ell}^\times \big\}\Q_{\ell}^\times/\Q_{\ell}^\times\\
&\simeq \big\{t = t_1+ t_2 w_1w_2: t_1,t_2 \in \Z_{\ell},\  \Nr(t) \in \Z_{\ell}^\times \big\}/\Z_{\ell}^\times
\end{align*}
with equality if ${\ell}$ is odd with ${\ell} \nmid \disc(Q)$.
Moreover, if ${\ell}=2$ or ${\ell}\mid \disc(Q)$ the index is bounded depending only on $Q$.
\end{claim*}

\begin{proof}[Proof of Claim]
It is clear that the inclusion `$\supseteq$' holds and it remains to prove the claims regarding the index.
For simplicity, we assume ${\ell}\neq 2$; the argument for ${\ell}=2$ is analogous to the case ${\ell} \mid \disc(Q)$ below.
We may replace $w_1,w_2$ by an integral basis of $W(\Q_{\ell})\cap\Z_{\ell}^4$ which we can choose to be orthogonal and satisfying $\alpha \mid \beta$ where $\alpha = Q(w_1)$ and $\beta = Q(w_2)$.
Note that $\mathcal{O}_{\alpha\beta} =\{t_1+ t_2 w_1w_2: t_1,t_2 \in \Z_{\ell}\}$ is an order isomorphic to $\Z_{\ell}[\sqrt{\alpha\beta}]$ in the quadratic \'etale algebra $K \otimes \Q_{\ell} \simeq \Q_{\ell}[\sqrt{\alpha\beta}]$.
As any $t\in \torus(\Z_{\ell})\subset K \otimes \Q_{\ell}$ preserves $\Z_{\ell}^4$ under $\rho$, $t$ is necessarily an integral element so that the claim is obvious if $\Z_{\ell}[\sqrt{\alpha\beta}]$ is the maximal order in $\Q_{\ell}[\sqrt{\alpha\beta}]$ i.e.~if ${\ell}^2 \nmid \alpha\beta$ (the same argument applies to $\overline{\torus}(\Z_{\ell})$).
In the following, we assume ${\ell} \mid \beta$ and let $t = t_1+t_2w_1w_2 \in K \otimes \Q_{\ell}$ be such that $\rho(t)$ preserves $\Z_{\ell}^4$.

We assume first ${\ell} \nmid \disc(Q)$.
In this case, the argument in \cite[Prop.~B.6]{AMW-higherdim} together with the fact that the quadratic form $q$ (on the complement of $W$) is assumed to be primitive implies that $Q(xw_1+yw_2)$ is primitive or equivalently that $\ell \nmid \alpha$.
Then
\begin{align*}
\Z_{\ell}^4\ni t w_1 t^{-1} = \tfrac{1}{\Nr(t)} t w_1 \sigma(t) 
= \tfrac{1}{\Nr(t)} \big((t_1^2-t_2^2 \alpha\beta)w_1 -2 \alpha t_1t_2 w_2\big)
\end{align*}
so that $t_1^2-t_2^2 \alpha\beta,\ t_1t_2 \in \Nr(t) \Z_{\ell}$.
As $w_2$ is isotropic modulo ${\ell}$ and $\ell \nmid \disc(Q)$, there exists an auxiliary vector $w_3 \in \Z_{\ell}^4$ with $\langle w_1,w_3\rangle = 0$ and $\langle w_2,w_3\rangle = 1$.
In this case, we have $w_2w_3 = -w_3w_2 + 2$ and so
\begin{align*}
\Z_{\ell}^4 \ni
t w_3 t^{-1} 
&= \tfrac{1}{\Nr(t)} t w_3 \sigma(t) 
= \tfrac{1}{\Nr(t)} (t_1 w_3 + t_2w_3w_1w_2 + 2t_2w_1) \sigma(t)\\
&= w_3 + \tfrac{1}{\Nr(t)} 2t_2w_1 \sigma(t)
= w_3 + \tfrac{1}{\Nr(t)} 2t_1t_2w_1 + \tfrac{1}{\Nr(t)}2 \alpha t_2^2 w_2 
\end{align*}
concluding $t_1^2,t_2^2 \in \Nr(t)\Z_{\ell}$.
If $\Nr(t) =1$, this shows $t_1,t_2 \in \Z_{\ell}$ as claimed. Otherwise, we multiply $t$ by a scalar so that $(t_1,t_2) \in \Z_{\ell}^2 \setminus \ell \Z_{\ell}^2$.
Thus, $t_1^2,t_2^2 \in \Nr(t)\Z_{\ell}$ shows $\Nr(t) \in \Z_{\ell}^\times$ and the lemma follows for ${\ell}$ not dividing $\disc(Q)$.

Assume now $\ell \mid \disc(Q)$. 
Using \cite[Prop.~B.6]{AMW-higherdim} again, the valuation of $\alpha$ is bounded absolutely in terms of $Q$.
As before, we have $t_1^2,\ \alpha t_1t_2 \in \Nr(t) \Z_{\ell}$.
Moreover, one can find an auxiliary vector $w_3 \in \Z_{\ell}^4$ with $\langle w_1,w_3\rangle = 0$ and $\langle w_2,w_3\rangle = \ell^n$ for some $n$ bounded absolutely in terms of $Q$.
Then $t w_3 t^{-1}\in \Z_{\ell}^4$ implies $\ell^n t_1t_2,\ell^n\alpha t_2^2 \in \Nr(t)\Z_{\ell}$.
If $\Nr(t) =1$ then $t_1\in \Z_{\ell}$ and $t_2 \in \ell^{-\lfloor(\ord_{\ell}(\alpha)+n)/2\rfloor}\Z_{\ell}$ which shows that $t$ is contained in an order $\mathcal{O}'\supset \mathcal{O}_{\alpha\beta}$ with index $[\mathcal{O'}:\mathcal{O}_{\alpha\beta}]$ an absolutely bounded power of ${\ell}$; this implies the lemma in this case.
For the claim regarding $\overline{\torus}(\Z_{\ell})$ we rescale $t$ such that $t_1\in \Z_{\ell}^\times$. Then $\Nr(t) \in \Z_{\ell}^\times$ and so $t_2 \in \ell^{-n}\Z_{\ell}$.
This proves the lemma if $\ell^{2n} \mid \alpha\beta$. 
Otherwise the index of $\mathcal{O}_{\alpha\beta}$ in the maximal order is absolutely bounded which also implies the claim.
\end{proof}

Let $\mathcal{O}$ be the order generated by $w_1w_2$ in $K$ and set $\mathcal{O} \otimes \hat{\Z} = \hat{\mathcal{O}}$.
Let $\torus(\hat{\Z}) = \prod_{\ell}\torus(\Z_{\ell})$ and similarly for $\overline{\torus}$.
By the above claim, we have surjections
\begin{equation}\label{eq:class groups for T,Tbar}
\begin{split}
\lrquot{K^1}{\A_K^1}{\hat{\mathcal{O}}^1} \to 
\lrquot{{\torus}(\Q)}{{\torus}(\A_f)}{{\torus}(\hat{\Z})},\\
\lrquot{K^\times}{\A_K^\times}{\hat{\mathcal{O}}^\times \A_\Q^\times} \to 
 \lrquot{\overline{\torus}(\Q)}{\overline{\torus}(\A_f)}{\overline{\torus}(\hat{\Z})} 
\end{split}
\end{equation}
with kernel of size $O_Q(1)$.
As $\Q$ has class number one, $\lrquot{K^\times}{\A_K^\times}{\hat{\mathcal{O}}^\times \A_\Q^\times} \simeq \lrquot{K^\times}{\A_K^\times}{\hat{\mathcal{O}}^\times}$.
Moreover, by genus theory (see e.g.~\cite[\S7.1]{W-Linnik}), the natural map induced by inclusion 
\begin{align}\label{eq:genustheory}
 \lrquot{K^1}{\A_K^1}{\hat{\mathcal{O}}^1} \to \lrquot{K^\times}{\A_K^\times}{\hat{\mathcal{O}}^\times}   
\end{align}
is injective with cokernel of size $D^{o(1)}$.
Lastly, the group $\lrquot{K^\times}{\A_K^\times}{\hat{\mathcal{O}}^\times}$ is isomorphic to the Picard group of the order $\mathcal{O}$ (because proper $\mathcal{O}$-ideals are locally principal, see e.g.~\cite[Prop.~2.1]{ELMV-Ens}). By the class number formula and Siegel's lower bound, it is of size $D^{\frac{1}{2}+o(1)}$.
By \eqref{eq:class groups for T,Tbar}, the same estimate holds for $\#(\lrquot{{\torus}(\Q)}{{\torus}(\A_f)}{{\torus}(\hat{\Z})})$ and $\#(\lrquot{\overline{\torus}(\Q)}{\overline{\torus}(\A_f)}{\overline{\torus}(\hat{\Z})})$.

We finally turn to proving the claimed volume estimates. Up to a proportionaliy constant, the volume of $Y_\iota$ is given by the cardinality of the finite group
\begin{align*}
\mathcal{F} = Y_\iota / (\rho(\torus(\A)) \cap \overline{\G}(\R)\overline{K}_f) 
\simeq \G(\Q) \backslash \G(\Q)\rho(\torus(\A)) / (\rho(\torus(\A)) \cap \overline{\G}(\R)\overline{K}_f).
\end{align*}
where $\overline{K}_f$ is as in \eqref{eq:maximalcompact}.
We have maps
\begin{align*}
\lrquot{{\torus}(\Q)}{{\torus}(\A_f)}{{\torus}(\hat{\Z})} 
\to \mathcal{F} \to \lrquot{\overline{\torus}(\Q)}{\overline{\torus}(\A_f)}{\overline{\torus}(\hat{\Z})} 
\end{align*}
where the first map is surjective and the compositum has kernel of size $O_Q(1)$ by \eqref{eq:class groups for T,Tbar} and \eqref{eq:genustheory}.
This proves $\#\mathcal{F} = D^{\frac{1}{2}+o(1)}$ and hence the lemma.
\end{proof}

For any positive definite integral quaternary quadratic form $Q'$ and $n\in \N$ consider the set $\mathcal{X}(Q',n)$ of tuples $(\iota_1,\iota_2)$ with the following properties:
\begin{itemize}
\item $\iota_1,\iota_2$ are primitive representations of $q$ by $Q'$ with distinct images.
    \item There exists $k \in \SO_{Q'}(\Z_p)$ with $k|_{\iota_1(\Q^2)^\perp}=\id$ such that
    \begin{align}\label{eq:rot calX}
    \iota_2(x) \equiv k\iota_1(x) \mod p^{n}
    \end{align}
    for all $x\in \Z^2$.
\end{itemize}
Whenever $Q',Q''$ are equivalent, we have $\#\mathcal{X}(Q',n)=\#\mathcal{X}(Q'',n)$.

\begin{lemma}\label{lem:pair correlation to vectors}
We have for any $n \geq 1$
\begin{align*}
\mu_\iota' \times \mu_\iota'&(\{(x,y): y \in x \Bow(n)\})\\
&\ll_\varepsilon D^{-\frac{1}{2}+\varepsilon} + D^{-1+\varepsilon} \max_{Q'\in \spn(Q)} \#\mathcal{X}(Q',2n)
\end{align*}
\end{lemma}

\begin{proof}
The proof of the lemma is arguably standard and a variant of it is, for instance, already contained in \cite{ELMV-Ens,W-Linnik}. For the readers' convenience, we provide a proof nevertheless.
We follow the notation set up in \S\ref{sec:translation}.
We also fix a set of representatives $\Z^4=\mathfrak{L}_1,\mathfrak{L}_2,\ldots,\mathfrak{L}_r$ for the $\SO_Q(\Q)$-orbits on $\spn(Q)$.
(As usual, each of these lattices gives rise to a quadratic form by representing $Q$ on a basis of the lattice.)
By weak approximation, we may assume $\mathfrak{L}_i \otimes \Z_p = \Z_p^4$ for all $i$.

First, we show how to construct from any a pair $(x,y) \in (Y_{\iota}')^2$ with  $y \in x \Bow(n)$ a pair of representation; much of this discussion is already contained in \S\ref{sec:period stabilizer}.
Write
\begin{align*}
x = \SO_Q(\Q)t_x\rho(g_{\iota,\infty})k_{\iota,p}, \quad
y = \SO_Q(\Q)t_y\rho(g_{\iota,\infty})k_{\iota,p}
\end{align*}
for $t_x,t_y \in \rho(\torus_{\iota}(\A))$.
Let $i \in \{1,\ldots,r\}$ be such that $(t_{x,f})_\ast \mathfrak{L} \in \SO_Q(\Q)\mathfrak{L}_i$.
Since $y \in x \Bow(n)$, this also implies $(t_{y,f})_\ast \mathfrak{L} \in \SO_Q(\Q)\mathfrak{L}_i$.
We may therefore select $\gamma_x,\gamma_y \in \SO_Q(\Q)$ so that $(\gamma_x t_x)_f,(\gamma_{y}t_y)_f$ map the lattice $\mathfrak{L} = \Z^4$ into the lattice $\mathfrak{L}_i$ and so that $\gamma_y t_y \rho(g_{\iota,\infty}) k_{\iota,p} \in \gamma_x t_x \rho(g_{\iota,\infty}) k_{\iota,p} \Bow(n) $.
As was established in \eqref{eq:from torus to rep}, $\iota_x=\gamma_x\iota$ and $\iota_y = \gamma_y \iota$ are two primitive representations of $q$ by $\mathfrak{L}_i$ (or, equivalently, the quadratic form on $\mathfrak{L}_i$).
One can check that the pair of primitive representations $(\iota_x,\iota_y)$ is uniquely defined modulo the diagonal left action of $\Aut(\mathcal{L}_i)$.
Moreover, the map $(x,y) \mapsto \Aut(\mathcal{L}_i).(\iota_x,\iota_y)$ satisfies the following injectivity property: whenever $(x,y),(x',y')$ have the same image, they differ by an element of the compact group
\begin{align*}
T^{\mathrm{cpt}} = k_{\iota,p}^{-1}(\overline{\torus}_\iota(\Q)\torus_\iota(\A)\cap \overline{\G}(\R)K_f)k_{\iota,p}
\end{align*}
in each coordinate.
Similarly, $\iota_x = \iota_y$ if $x,y$ differ by an element of $T^{\mathrm{cpt}}$.

Next, we record the meaning of the Bowen ball displacement in terms of the representations.
Any $g \in k_{\iota,p} \Bow_p(n) k_{\iota,p}^{-1}$ can be written as $g = g' t t'$ where $g'\in \overline{\G}(\Z_p)$ satisfies $g' \equiv \id \mod p^{2n}$ and where $t\in \rho(\torus_{\iota}(\Q_p))\cap \overline{\G}(\Z_p)$, $t'\in \rho(\torus_{\iota^\perp}(\Q_p)) \cap \overline{\G}(\Z_p)$ satisfy $t\equiv t' \equiv \id \mod p$.
Here, $\torus_{\iota^\perp}<\G$ is the spin group of $\iota(\Q^2)$.
We may thus write $(\gamma_y t_y)_p \in (\gamma_x t_x)_p g'tt'$ for $g',t,t'$ as above. Therefore, for all $v\in \Z^2$
\begin{align*}
\iota_y(v) = (\gamma_yt_y)_p \iota(v) = (\gamma_x t_x)_p g'tt' \iota(v)
\equiv \big((\gamma_x t_x)_pt'(\gamma_x t_x)_p^{-1}\big) \iota_x(v) \mod p^{2n}
\end{align*}
where we used that $t,t',t_x$ commute and that $(\gamma_x t_x)_p$ is integral.
Notice that $(\gamma_x t_x)t'(\gamma_x t_x)^{-1}$ is an integral rotation fixing the orthogonal complement of $\iota_x(\Q_p^2)$.
This proves that $(\iota_x,\iota_y) \in \mathcal{X}(\mathfrak{L}_i,2n)$ unless $\iota_x=\iota_y$ or equivalently $x,y$ differ by an element in $T^{\mathrm{cpt}}$.
As $\torus_\iota$ is abelian and using Lemma~\ref{lem:volume}, we have $\mu_\iota'(xT^{\mathrm{cpt}})\ll \vol(Y_{\iota})^{-1}\ll_\varepsilon D^{-\frac{1}{2}+\varepsilon}$ for any $x \in Y_\iota$.
By Fubini, the set $\{(x,y): y\in xT^{\mathrm{cpt}}\}$ has $\mu_\iota'\times \mu_\iota'$-measure $\ll_\varepsilon D^{-\frac{1}{2}+\varepsilon}$.
Similarly, for any $(x,y) \in Y_\iota$ we have
\begin{align*}
\mu_\iota' \times \mu_\iota'\big(\{
(x',y'):(\iota_{x'},\iota_{y'})=(\iota_{x},\iota_{y})
\}\big) 
= \mu_\iota' \times \mu_\iota'( xT^{\mathrm{cpt}} \times  yT^{\mathrm{cpt}}) 
\ll_\varepsilon D^{-1+\varepsilon}.
\end{align*}
This proves the lemma.
\end{proof}

\subsection{Reduction to a counting problem on a hypersurface}

Henceforth, our aim is to establish (non-trivial) upper bounds on $\max_{Q'\in \spn(Q)} \#\mathcal{X}(Q',2n)$.
Since $\mathcal{X}(Q',2n)=\mathcal{X}(Q'',2n)$ for all $Q',Q''$ equivalent, it suffices to establish an upper bound for an individual set $\#\mathcal{X}(Q',2n)$.
For simplicity of the exposition, we will consider below only $Q'=Q$.

First, we apply a corollary of the Siegel mass formula to reduce the counting problem for the sets $\mathcal{X}(Q,2n)$ to a more manageable counting problem on a degree $4$ hypersurface in $\mathbf{A}^5$.
We begin with the following elementary, though maybe surprisingly intricate, lemma which captures the congruence conditions at the prime~$p$.

\begin{lemma}
Let $(\iota_1,\iota_2) \in \mathcal{X}(Q,2n)$ for some $n\in \N$.
Define the half-integers
\begin{equation}\label{eq:xi's from reps}
\begin{split}
x_1 = \langle \iota_1(e_1),\iota_2(e_1)\rangle_{Q},\ x_2 = \langle \iota_1(e_1),\iota_2(e_2)\rangle_{Q},\\
x_3 = \langle \iota_1(e_2),\iota_2(e_1)\rangle_{Q},\ x_4 = \langle \iota_1(e_2),\iota_2(e_2)\rangle_{Q}.
\end{split}
\end{equation}
Then we have
\begin{align*}
x_1x_4-x_2x_3 &\equiv D \mod p^{4n},\\
Cx_1-Ax_4 &\equiv 0 \mod p^{4n},\\
C(x_2+x_3)-Bx_4 &\equiv 0 \mod p^{4n},\\
A(x_2+x_3)-Bx_1 &\equiv 0 \mod p^{4n}.
\end{align*}
\end{lemma}

\begin{proof}
It will be convenient to define an inner product on $\wedge^2 \Z^4$ through
\begin{align}\label{eq:disc}
\langle v_1\wedge v_2,u_1\wedge u_2\rangle_Q 
= \det\begin{pmatrix}
\langle v_1,u_1 \rangle_Q & \langle v_1,u_2 \rangle_Q \\
\langle v_2,u_1 \rangle_Q & \langle v_2,u_2 \rangle_Q
\end{pmatrix}
\end{align}
and similarly on $\wedge^3 \Z^4$.
The quadratic form $Q$ extends analogously (this is the discriminant).

Let $w_1 = \iota_1(e_1)$, $w_2 = \iota_1(e_1)$, $w_1' = \iota_2(e_1)$, and $w_2' = \iota_2(e_2)$. Let $k$ be a rotation as in \eqref{eq:rot calX} so that $w_i' \equiv kw_i \mod p^{2n}$ for $i=1,2$.
We begin by proving the quadratic congruence equation.
For this, notice that
\begin{align*}
Q(w_1\wedge w_2 - w_1' \wedge w_2') =D - 2\langle w_1 \wedge w_2,w_1'\wedge w_2'\rangle_Q +D
= 2D - (x_1x_4-x_2x_3)
\end{align*}
and on the other hand
\begin{align*}
Q(w_1\wedge w_2 - w_1' \wedge w_2') &= Q(w_1\wedge w_2 - k.w_1 \wedge k.w_2+O(p^{2n})) = O(p^{4n}).
\end{align*}
This verifies the quadratic congruence condition in the lemma.

The linear congruence conditions are more involved.
We begin by noting that they are valid mod $p^{2n}$ (instead of $p^{4n}$).
Indeed,
\begin{align*}
Cx_1-\tfrac{1}{2}Bx_2 &=
\det \begin{pmatrix}
\langle w_1,w_1'\rangle & \langle w_1,w_2'\rangle \\
 \langle w_2',w_1'\rangle & \langle w_2',w_2'\rangle
\end{pmatrix}
\equiv \det \begin{pmatrix}
\langle w_1,w_1\rangle & \langle w_1,w_2\rangle \\
 \langle w_2',w_1\rangle & \langle w_2',w_2\rangle
\end{pmatrix}\\
&= Ax_4 - \tfrac{1}{2}Bx_2 \mod p^{2n}
\end{align*}
and similarly for the other linear congruence equations.

We now aim to upgrade the above congruence equations mod $p^{2n}$ to the desired linear congruence equations mod $p^{4n}$.
For this, it will be convenient to consider for $t = 0,\ldots,p-1$ the vectors
\begin{align*}
w_1(t) = w_1+tw_2,\ w_2(t)= w_2,\ w_1'(t) = w_1'+tw_2',\ w_2'(t)= w_2'.
\end{align*}
Clearly, $w_1(t)' \equiv k.w_1(t) \mod p^{2n}$. We define $x_i(t)$ in analogy to the previous inner products $x_i$ so that, in particular, $x_i(0)=x_i$.
Furthermore, set $A(t) = \langle w_1(t), w_1(t)\rangle$ and define $B(t),C(t)$ similarly.
It is easy to see that 
\begin{align*}
B(t)^2-4A(t)C(t) =-4D,\quad x_1(t)x_4(t)-x_2(t)x_3(t) = x_1x_4-x_2x_3
\end{align*}
and that the analogous linear congruence conditions are valid mod $p^{2n}$.
Also, notice that $x_1(\cdot),A(\cdot)$ are quadratic polynomials in $t$, $x_2(\cdot),x_3(\cdot),B(\cdot)$ are linear polynomials, and $C(\cdot),x_4(\cdot)$ are constant.

We now find additional quadratic congruence conditions.
Notice that
\begin{align*}
O(p^{4n}) &= Q(w_1(t)\wedge w_2(t) \wedge w_1'(t)) \\
&= -C(t)x_1(t)^2 +B(t)x_1(t)x_3(t)-Ax_3(t)^2 + A(t)D.
\end{align*}
Proceeding similarly with the pure wedges $w_1(t)\wedge w_2(t) \wedge w_2'(t),w_1(t)\wedge w_1'(t) \wedge w_2'(t)$, and $w_2(t)\wedge w_1'(t) \wedge w_2'(t)$, we obtain
\begin{equation}\label{eq:quadcongeqfrom3wedge}
\begin{split}
C(t)x_1(t)^2 -B(t)x_1(t)x_3(t) + A(t)x_3(t)^2 \equiv A(t)D \mod p^{4n},\\
C(t)x_2(t)^2 -B(t)x_2(t)x_4(t) + A(t)x_4(t)^2 \equiv C(t)D \mod p^{4n},\\
C(t)x_1(t)^2 -B(t)x_1(t)x_2(t) + A(t)x_2(t)^2 \equiv A(t)D \mod p^{4n},\\
C(t)x_3(t)^2 -B(t)x_3(t)x_4(t) + A(t)x_4(t)^2 \equiv C(t)D \mod p^{4n}.
\end{split}
\end{equation}
In particular, taking the difference of the first and the third resp.~the second and the fourth equation in \eqref{eq:quadcongeqfrom3wedge}, we deduce
\begin{equation}\label{eq:quadcongeqfrom3wedge-2}
\begin{split}
(x_2-x_3)\big(A(t)(x_2(t)+x_3(t))-B(t)x_1(t)\big) \equiv 0 \mod p^{4n},\\
(x_2-x_3)\big(C(t)(x_2(t)+x_3(t))-B(t)x_4(t)\big) \equiv 0 \mod p^{4n}.
\end{split}
\end{equation}
using $x_2(t)-x_3(t)=x_2-x_3$.
Similarly, taking $C(t)$ times the first minus $A(t)$ times the fourth equation in \eqref{eq:quadcongeqfrom3wedge} (or similarly with the second and third), we obtain
\begin{equation}\label{eq:quadcongeqfrom3wedge-3}
\begin{split}
\big(C(t)x_1(t)-A(t)x_4(t)\big)\big(C(t)x_1(t)+A(t)x_4(t)-B(t)x_3(t)\big) \equiv 0 \mod p^{4n},\\
\big(C(t)x_1(t)-A(t)x_4(t)\big)\big(C(t)x_1(t)+A(t)x_4(t)-B(t)x_2(t)\big) \equiv 0 \mod p^{4n}.
\end{split}
\end{equation}

\textsc{Case 1:} $p \nmid (x_2-x_3)$.

In particular, we have $A(x_2+x_3)-Bx_1 \equiv C(x_2+x_3)-Bx_4 \equiv0 \mod p^{4n}$ by \eqref{eq:quadcongeqfrom3wedge-2} at $t=0$.
Applying the first congruence equation in \eqref{eq:quadcongeqfrom3wedge-2} at arbitrary $t$ and expanding the polynomials in $t$ we also obtain
\begin{align*}
0 &\equiv \big(B x_{4}-C(x_{2} + x_{3})\big) t^{2} + 2(- C x_{1} + A x_{4}) t - B x_{1} + A x_{2} + A x_{3}\\
&\equiv 2(- C x_{1} + A x_{4}) t - B x_{1} + A x_{2} + A x_{3} \mod p^{4n}
\end{align*}
Since $p \neq 2$, this shows also $ A x_{4} - C x_{1}\equiv 0\mod p^{4n}$ concluding the proof in this case.

\textsc{Case 2:} $p \mid (x_2-x_3)$.

Since the linear congruence conditions are valid mod $p$, we have $x_2(t) \equiv x_3(t)$, $C(t)x_1(t)\equiv A(t) x_4(t)$, $C(t)x_2(t) \equiv \frac{B(t)}{2}x_4(t)$, and $A(t)x_2(t)\equiv \frac{B(t)}{2}x_1(t)$ modulo $p$.
Since $p \nmid D$, at least one of $A(t),B(t),C(t)$ is not divisible by $p$ and, therefore, the vector $(x_1(t),x_2(t),x_3(t),x_4(t))$ is a multiple of $(A(t),\frac{B(t)}{2},\frac{B(t)}{2},C(t))$.
By the quadratic congruence condition $x_1x_4-x_2x_3 \equiv D\mod p$, the multiple is non-zero, say $(x_1,x_2,x_3,x_4) \equiv \lambda_t(A,\frac{B}{2},\frac{B}{2},C)\mod p$ for some $\lambda_t \in \{1,\ldots,p-1\}$.
Thus,
\begin{align*}
C(t)x_1(t)+A(t)x_4(t)-B(t)x_2(t)
\equiv \lambda_t (2A(t)C(t) -\tfrac{1}{2}B(t)^2) \not\equiv 0 \mod p
\end{align*}
and by \eqref{eq:quadcongeqfrom3wedge-3} we deduce $C(t)x_1(t)-A(t)x_4(t)\equiv 0 \mod p^{4n}$.
Equivalently,
\begin{align*}
\big(C (x_{2} + x_{3}) - B x_{4}\big) t + C x_{1} - A x_{4} \equiv 0 \mod p^{4n}
\end{align*}
and we obtain $C (x_{2} + x_{3}) - B x_{4}\equiv C x_{1} - A x_{4} \equiv 0 \mod p^{4n}$.
If $p \nmid C$, this implies the lemma.

Assume $p \mid C$.
We replace the shear used above by the shear $(w_1,w_2)\mapsto (w_1,w_2+tw_1)$ and similarly for $(w_1',w_2')$. Following the same argument we arrive at $C(t)x_1(t)-A(t)x_4(t)\equiv 0 \mod p^{4n}$ where $A(\cdot),\ldots$ are defined analogously. Expanding again in $t$ we find 
\begin{align*}
\big(B x_{1} - A (x_{2} +x_{3})\big) t + C x_{1} - A x_{4}\equiv 0 \mod p^{4n}
\end{align*}
and conclude.
\end{proof}

Suppose that $\iota_1,\iota_2$ are two representations of $q$ by $Q$ with distinct images and define $x_1,x_2,x_3,x_4$ as in \eqref{eq:xi's from reps}.
If the lattice $\iota_1(\Z^2)+\iota_2(\Z^2)$ has full rank, i.e.~rank $4$, then the discriminant on this lattice is, clearly, the index squared times the discriminant of $Q$. 
Denote this index by $x_0$.
Then
\begin{align*}
\disc(Q)x_0^2 &= \det \begin{pmatrix}
A & B/2 & x_1 & x_2 \\
B/2 & C & x_3 & x_4 \\
x_1 & x_3 & A & B/2 \\
x_2 & x_4 & B/2 & C
\end{pmatrix} \\
&= (x_1x_4-x_2x_3 -D)^2 - (Cx_1-Ax_4)^2 \\
&\qquad- (Bx_1-A(x_2+x_3))(Bx_4-C(x_2+x_3))
\end{align*}
When the lattice spanned by $\iota_1(\Z^2)$ and $\iota_2(\Z^2)$ has rank $3$, the above matrix has determinant zero and we obtain a solution to the above equation with $x_0 = 0$.
The above equation cuts out the aforementioned hypersurface in $\mathbf{A}^5$.
If $(\iota_1,\iota_2) \in \mathcal{X}(Q,2n)$ then $x_0 \equiv 0 \mod p^{4n}$.

Note that the Cauchy-Schwarz inequality implies the Archimedean bounds
\begin{align*}
|x_0|\ll D,\quad
|x_1| \leq A,\quad |x_4| \leq  C,\quad |x_2|,|x_3| \leq \sqrt{AC}.
\end{align*}
The specific bound on $x_0$ will be inconsequential to us and we note that the above degree $4$ equation implies $|x_0|\ll D$ as well.

Also, whenever $x_1 =A$ then $\iota_1(e_1)=\iota_2(e_1)$ and so $x_2=x_3=\tfrac{B}{2}$. The same conclusion is reached when $x_4 =C$.
Also, when $x_1=-A$ or $x_4 = -C$ then $x_2=x_3=-\tfrac{B}{2}$.
In all of these cases, $x_0=0$.

The following definitions collects all the properties of the tuples $(x_0,\ldots,x_4)$.

\begin{definition}\label{def:set S}
For any $n\in \N$ the subset $\mathcal{S}(n) \subset \frac{1}{2}\Z^5$ is the set of tuples $(x_0,\ldots,x_4)$ satisfying
\begin{equation}\label{eq:degree4 equation}
\begin{split}
\disc(Q)x_0^2 &= (x_1x_4-x_2x_3 -D)^2 - (Cx_1-Ax_4)^2 \\
&\qquad- (Bx_1-A(x_2+x_3))(Bx_4-C(x_2+x_3))
\end{split}
\end{equation}
as well as the following constraints:
\begin{enumerate}[label=\textnormal{(S\arabic*)}]
    \item\label{item:S1} $|x_1| \leq A$, $|x_4| \leq C$, and $|x_2|,|x_3| \leq \sqrt{AC}$.
    \item\label{item:S2} We have
\begin{equation}\label{eq:originalcongeqforx}
    \begin{aligned}
x_1x_4-x_2x_3 &\equiv D \mod p^{4n},\\
Cx_1-Ax_4 &\equiv 0 \mod p^{4n},\\
C(x_2+x_3)-Bx_4 &\equiv 0 \mod p^{4n},\\
A(x_2+x_3)-Bx_1 &\equiv 0 \mod p^{4n}.
\end{aligned}
\end{equation}
\item\label{item:S3} If $x_1 =A$ or $x_4 = C$ then $x_0 = 0$, $x_2=x_3=\frac{B}{2}$.
If $x_1 =-A$ or $x_4 = -C$ then $x_0 = 0$, $x_2=x_3=-\frac{B}{2}$.
\end{enumerate}
\end{definition}

\begin{proposition}\label{prop:from vectors to inner products}
For any $n\in \N$ with $p^{4n} < \sqrt{D}$ and any $\delta\in (0,\frac{1}{4})$ we have
\begin{align*}
\# \mathcal{X}(Q,2n) \ll_{Q,\varepsilon} D^{\delta+\varepsilon} \#\mathcal{S}(n) + \frac{D^{\frac{3}{2}-\delta + \varepsilon}}{p^{8n}} + \frac{D^{7/8+\varepsilon}}{p^{5n}}  +D^{\frac{1}{2}+\varepsilon}.
\end{align*}
Moreover, if $D$ is fundamental, we have for any $n\in \N$
\begin{align*}
\# \mathcal{X}(Q,2n) \ll_{Q,\varepsilon} D^{\varepsilon} \#\mathcal{S}(n).
\end{align*}
The analogous bounds hold for any $Q' \in \spn(Q)$.
\end{proposition}

We note that, when applying Proposition~\ref{prop:from vectors to inner products}, $\delta>0$ will be very small and $p^{4n}$ will be very close to $\sqrt{D}$ so that all terms other than $D^{\delta+\varepsilon} \#\mathcal{S}(n)$ will be controlled by $D^{\frac{1}{2}+\varepsilon}$.

\begin{proof}[Proof of Proposition~\ref{prop:from vectors to inner products} for fundamental discriminants]
Suppose that $D$ is fundamental.
For any  $(\iota_1,\iota_2) \in \mathcal{X}(Q',2n)$ define $x_1,x_2,x_3,x_4$ as in \eqref{eq:xi's from reps}.
Moreover, define $x_0$ to be the index of $\iota_1(\Z^2)+\iota_2(\Z^2)$ in $\Z^4$ whenever the rank of $\iota_1(\Z^2)+\iota_2(\Z^2)$ is four. When the rank is three, set $x_0 = 0$.
As already established, we have defined in this way a map
\begin{align*}
\Phi: \mathcal{X}(Q',2n) \to \mathcal{S}(n).
\end{align*}

We wish to bound the size of the fiber $\Phi^{-1}(x)$ above a point $x \in \mathcal{S}(n)$.
In other words, we count the number of pairs of representations $(\iota_1,\iota_2)$ for which the quadratic form on $\Z \iota_1(e_1)+\Z \iota_1(e_2)+\Z \iota_2(e_1)+\Z \iota_2(e_2)$ is given by the form $Q_x$ with matrix representation
\begin{align}\label{eq:Qx}
\begin{pmatrix}
A & B/2 & x_1 & x_2 \\
B/2 & C & x_3 & x_4 \\
x_1 & x_3 & A & B/2 \\
x_2 & x_4 & B/2 & C
\end{pmatrix}.
\end{align}

If $x_0 \neq 0$, the quadratic form $Q_x$ is non-degenerate.
As $D$ is fundamental, the greatest common divisor of the entries of $2Q_x$ is $\ll 1$ and $\mathsf{m}_2(Q_x) \ll 1$ (cf.~\eqref{eq:m_j}).
Thus, by the application of the Siegel mass formula in Theorem~\ref{thm:app-equalrank}, we have $\#\Phi^{-1}(x) \ll_\varepsilon D^\varepsilon$.

So suppose that $x_0=0$.
Without loss of generality, we may assume that the rank of $\iota_1(\Z^2)+\Z\iota_2(e_1)$ is three and write $Q_x'$ for the quadratic form on this sublattice.
As $D$ is fundamental, the greatest common divisor of the coefficients of $Q_x'$ is $\ll 1$ and $\mathsf{m}_2(Q_x') \ll 1$.
By the application of the Siegel mass formula in Corollary~\ref{cor:app-codim1} the point $x \in \mathcal{S}(n)$ determines $\iota_1$ and $\iota_2(e_1)$ up to a factor $\ll D^\varepsilon$.
Moreover, the vector $\iota_2(e_2)$ is contained in the three-dimensional subspace spanned by $\iota_1(\Q^2)$ and $\iota_2(e_1)$ (since $x_0=0$) and has determined inner products with each of the vectors $\iota_1(e_1)$, $\iota_1(e_2)$, and $\iota_2(e_1)$. 
Thus, $\iota_2(e_2)$ is also determined.
Overall, this shows $\#\Phi^{-1}(x) \ll_\varepsilon D^{\varepsilon}$ and proves the claim.
\end{proof}

For $D$ non-fundamental the proof of Proposition~\ref{prop:from vectors to inner products} is more involved and we will require a few preliminaries.
The first two of the following lemmas will also be used later for other purposes.

\begin{lemma}\label{lem:Pell}
Let $f(x,y) \in \Z[x,y]$ be of total degree two so that its degree two homogeneous constituent $q'(x,y)$ (a binary quadratic form) is non-degenerate.
Then the number of solutions to $f(x,y) = a$ with $|x|,|y| \leq R$ is $(|a|\,R\,\height(f))^{o(1)}$.
\end{lemma}

Here, the height $\height(f)$ of $f$ is the maximum absolute value of its coefficients.

\begin{proof}
By a linear change of coordinates and using the assumption that $q'$ is non-degenerate, one readily reduces the lemma to the case $f=q'$.
We assume without loss of generality that $q'$ is primitive throughout.
We further reduce the statement to the following claim regarding the number of solutions to an equation of Pell-type.
Let $K =\Q(\sqrt{d})$ and write $\norm{x}$ for the Euclidean norm of $x\in K$ in the image under $K \hookrightarrow \R^2$ or $K \hookrightarrow \C$.

\begin{claim*}
The number of $x \in \mathcal{O}_K^\times$ with $\norm{x} \leq R$ is $\ll \log(R)$.
\end{claim*}

If $K$ is imaginary, the statement is trivial as $\#\mathcal{O}_K^\times \ll 1$, so assume $K$ is real.
In that case, it is a standard observation that any fundamental unit $\varepsilon \in \mathcal{O}_K^\times$ satisfies $|\log(|\varepsilon|)| \geq \kappa$ for some absolute $\kappa$. Thus the claim follows.

Let $\mathcal{O}$ be the quadratic order of discriminant $d$; if $d$ is fundamental, $\mathcal{O}$ is the ring of integers in $K = \Q(\sqrt{d})$.
Recall that a proper fractional ideal $I$ for $\mathcal{O}$ is a rank two $\Z$-submodule of $\Q(\sqrt{d})$ with the property that $\mathcal{O} = \{x\in K: xI \subset I\}$.
Any such ideal $I$ gives rise to the primitive binary quadratic form $x \in I \mapsto \tfrac{1}{\mathrm{N}(I)} \mathrm{Nr}_{K/\Q}(x)$ where $\mathrm{N}(I) = \frac{(\mathcal{O}:\mathcal{O}\cap I)}{(I:\mathcal{O}\cap I)}$ is the norm of the ideal $I$.
Moreover, any primitive binary quadratic form of discriminant $d$ is of this shape.

Let $I$ be as above for our binary form $q'$; we may assume $I \subset \mathcal{O}$ and $N(I) \ll |d|^{\frac{1}{2}}$ by Minkowski's theorem.
It thus suffices to estimate the number of solutions $x\in I$ to $\mathrm{Nr}_{K/\Q}(x) = a'$ where $a' = N(I)a \in \Z$ and $\norm{x}$ is bounded linearly in $R$. 
Note that the number of $\mathcal{O}_K$-ideals with norm $a'$ is bounded by $|a'|^{o(1)}$ by standard bounds on divisor functions.
We have reduced to counting for a given $x_0 \in \mathcal{O}_K$ the number of $x \in \mathcal{O}_K$ with norm $a'$ and $\mathcal{O}_Kx = \mathcal{O}_K x_0$ and $\norm{x} \ll \sqrt{d}R$.
Clearly, $x,x_0$ differ by a unit in $\mathcal{O}_K^\times$ and so the above claim implies the lemma.
\end{proof}

\begin{lemma}\label{lem:degenerate S}
For any $n \in \N$ with $p^{4n} \leq \sqrt{D}$ we have
\begin{align}\label{eq:Sdeg-def}
\#\{x \in \mathcal{S}(n): Cx_1-Ax_4=A(x_2+x_3)-Bx_1=0\} \ll \frac{\sqrt{D} }{p^{4n}}.
\end{align}
Moreover, $x \in \mathcal{S}(n)$ is contained in the above set if and only if
\begin{align*}
(Cx_1-Ax_4)^2 + (Bx_1-A(x_2+x_3))(Bx_4-C(x_2+x_3))=0.
\end{align*}
\end{lemma}

\begin{proof}
        Recall that $q$ is primitive so that $\operatorname{gcd}(A,B,C)=1$.
        Thus, for $x$ contained in the set in \eqref{eq:Sdeg-def} we see that $(x_1,x_2+x_3,x_4)$ is a multiple of $(A,B,C)$. 
        Since $|x_1|\leq A$ and $|x_4|\leq C$, this implies that $(x_1,x_2+x_3,x_4)$ is either $(0,0,0)$ or $\pm(A,B,C)$. 
        If $(x_1,x_2+x_3,x_4)=(0,0,0)$ then $x_2$ mod $p^{4n}$ is determined by the quadratic congruence equation $x_2^2\equiv x_1x_4-x_2x_3\equiv D$ mod $p^{4n}$. If $(x_1,x_2+x_3,x_4)=\pm(A,B,C)$ then the last property of $\mathcal{S}(n)$ in Definition \ref{def:set S} determines $x_2=x_3=\pm\frac{B}{2}$. 
        The first part of the lemma follows.

        Suppose now that $x\in \mathcal{S}(n)$ satisfies $(Cx_1-Ax_4)^2 + (Bx_1-A(x_2+x_3))(Bx_4-C(x_2+x_3))=0$.
        Then setting $s= Cx_1-Ax_4$ and $t = Bx_1-A(x_2+x_3)$ we see that $s^2 + t(-Bs+Ct)\frac{1}{A}=0$. Equivalently, $q(s,-t) =0$ and, as $q$ is anisotropic, $s=t=0$ which concludes the second part of the lemma.
\end{proof}

In the following, $\mathsf{m}_1(\cdot), \mathsf{m}_2(\cdot),\ldots$ are the quantities defined in \eqref{eq:m_j} for quadratic forms (these are products of local factors capturing elementary divisors).

\begin{lemma}\label{lem: non-trivial m2}
For $x\in \mathcal{S}(n)$ let $Q_x$ be as in \eqref{eq:Qx} and suppose that $Q_x$ is non-degenerate.
Let $p_0$ be a prime and suppose that $p_0^\ell \Vert \mathsf{m}_2(Q_x)$ for $\ell>0$.
Then $p_0^\ell \mid 4D$.
If $p_0 \neq 2$, there exists $\epsilon \in \{\pm 1\}$ with
\begin{align*}
x_1 \equiv \epsilon A \mod p_0^\ell,\
x_2 \equiv x_3\equiv \tfrac{\epsilon}{2} B \mod p_0^\ell,\
x_4 \equiv \epsilon C \mod p_0^\ell.
\end{align*}
If $p_0 = 2$ and $\ell \geq 2$, there exists $\epsilon \in \{\pm 1\}$ with
\begin{align*}
x_1 \equiv \epsilon A \mod 2^{\ell-2},\
x_2 \equiv x_3\equiv \epsilon B \mod 2^{\ell-2},\
x_4 \equiv \epsilon C \mod 2^{\ell-2}.
\end{align*}
\end{lemma}

\begin{proof}
The assertion $p_0^\ell \mid 4D$ is immediate from the definitions.
We prove the lemma for $p_0\neq 2$ first.
Since $q$ is primitive, we have $\gcd(A,B,C) =1$.
Suppose that $p_0 \nmid A$.
Then $Q_x(e_1) = A$ is a unit mod $p_0$, and after orthogonal projection to the complement of $e_1$ the quadratic form $Q_x$ is equivalent to
\begin{align}\label{eq:proj e1perp}
\begin{pmatrix}
A & 0 & 0 & 0 \\
0 & \frac{D}{A}&
    -\frac{B x_1}{2A} + x_3 &
    -\frac{B x_2}{2A} + x_4 \\
0 & -\frac{B x_1}{2A} + x_3 &
    A - \frac{x_1^2}{A} &
    \frac{B}{2} - \frac{x_1 x_2}{A} \\
0 & -\frac{B x_2}{2A} + x_4 &
    \frac{B}{2} - \frac{x_1 x_2}{A} &
    C - \frac{x_2^2}{A}
\end{pmatrix}.
\end{align}
The greatest common divisor of the so-obtained $3\times3$ symmetric matrix in the bottom right corner is precisely $p_0^\ell$ by definition of the elementary divisors (cf.~Definition~\ref{def:elementary divisors}). 
In particular, $x_1^2 \equiv A^2 \mod p_0^\ell$.
Since $p_0 \nmid A$ we have $x_1 \equiv \epsilon A\mod p_0^\ell$ for $\epsilon \in \{\pm1\}$.
Inserting this into the other entries of the matrix the lemma follows.

If $p_0 \nmid C$ we may apply a coordinate switch $e_1 \leftrightarrow e_2$ and $e_3 \leftrightarrow e_4$ which exchanges $x_1 \leftrightarrow x_4$ and $x_2\leftrightarrow x_3$ so that the above discussion yields the desired congruence conditions.
If $p_0\nmid B$ and $p_0 \mid A,C$, we apply a shear $e_1 \mapsto e_1 + e_2$, $e_2 \mapsto e_2$, $e_3 \mapsto e_3 + e_4$, $e_4 \mapsto e_4$ and again use the previous discussion.

Suppose now that $p_0 =2$ and, without loss of generality, that $2 \nmid A$. 
In this case, the above orthogonal projection to $e_1^\perp$ is well-defined on $\Lambda = \Z_2 e_1 \oplus 2 \bigoplus_{i>1} \Z_2 e_i$, where $Q_x' =Q_x|_\Lambda$ is equivalent to \eqref{eq:proj e1perp} after multiplying the lower right $3\times 3$ block by $4$.
We denote the latter ternary form by $Q_x''$.
From Definition~\ref{def:elementary divisors}, we see that $k_1(2,Q_x)=k_1(2,Q_x')=1$ and $k_2(2,Q_x) \leq k_2(2,Q_x')$ where $k_2(2,Q_x)=\ell$ by assumption.
Moreover, notice that $k_2(2,Q_x') \leq k_1(2,Q_x'')$ so that overall $\ell \leq k_1(2,Q_x'')$. From this, we conclude that $4(x_1^2-A^2) \equiv 0 \mod 2^\ell$
which implies $x_1 \equiv \epsilon A \mod 2^{\ell-2}$ for $\epsilon \in \{\pm1\}$. The other congruence conditions follow similarly.
\end{proof}

\begin{proof}[Proof of Proposition~\ref{prop:from vectors to inner products} for non-fundamental discriminants]
As in the fundamental case, define $x_0,\ldots,x_4$ for any $(\iota_1,\iota_2) \in \mathcal{X}(Q,2n)$ as well as a map
\begin{align*}
\Phi: \mathcal{X}(Q',2n) \to \mathcal{S}(n).
\end{align*}
Recall the definition of $\mathsf{m}_2(\cdot),\mathsf{m}_2'(\cdot)$ from Appendix~\ref{sec:Siegel mass}.
In the following we estimate contributions to $\mathcal{X}(Q',2n)$ arising through different subsets of $\mathcal{S}(n)$.

\medskip
\textsc{Case 1:} \textit{$x_0 \neq 0$.}

(For this proof, this is a short-hand for estimating the number of $(\iota_1,\iota_2)$ for which $x = \Phi(\iota_1,\iota_2)$ satisfies $x_0\neq 0$. Similar short-hands are applied throughout the proof.)

\textsc{Case 1.1:} \textit{$\mathsf{m}_2'(Q_x) \leq D^\delta$.}

In this case, the same argument as for $D$ fundamental (based on Theorem~\ref{thm:app-equalrank}) applies to show that the contribution of such points is $\ll D^\delta \#\mathcal{S}(n)$.

\textsc{Case 1.2}: \textit{$\mathsf{m}_2'(Q_x) > D^\delta$.}

Let $\mathsf{m}$ be a fixed divisor of $D$. 
We claim that the number of $x\in \mathcal{S}(n)$ so that $\mathsf{m}_2(Q_x) = \mathsf{m}$ is $\ll \max\{ \tfrac{D^{3/2}}{p^{8n}\mathsf{m}},1\}$.
Combining this claim with Theorem~\ref{thm:app-equalrank} we obtain that the number of pairs of representations $\underline{\iota} = (\iota_1,\iota_2) \in \mathcal{X}(Q',2n)$ so that $x=\Phi(\underline{\iota})$ satisfies $\mathsf{m}_2(Q_x) = \mathsf{m}> D^{2\delta}$ is
\begin{align*}
\ll \max\{ \tfrac{D^{3/2}}{p^{8n}\mathsf{m}},1\}\mathsf{m}_2'(Q_x) \ll  \max\{ \tfrac{D^{3/2}}{p^{8n}\sqrt{\mathsf{m}}},\sqrt{\mathsf{m}}\} \ll \max\{ \tfrac{D^{3/2-\delta}}{p^{8n}},\sqrt{D}\}.
\end{align*}
As $D$ has $D^{o(1)}$ many divisors, this suffices.

To prove the above claim, we leverage the congruence conditions of Lemma~\ref{lem: non-trivial m2}.
Losing a factor of $\ll_\varepsilon D^\varepsilon$, we may also fix the signs $\epsilon$ in that lemma.
Therefore, there exists a positive divisor $\tilde{\mathsf{m}} \mid \mathsf{m}$ with $\frac{\mathsf{m}}{\tilde{\mathsf{m}}} \leq 4$ so that $x_1,\ldots,x_4$ are determined modulo $\tilde{\mathsf{m}}$.

Let $\ell \geq 0$ be such that $p^\ell \Vert A$.
Then $p^\ell \mid x_1$.
With $x_1$ given, $x_2+x_3$ is determined modulo $p^{4n-\ell}$ and, if $x_2+x_3$ is also given, $x_4$ is determined modulo $p^{4n}$; see \eqref{eq:originalcongeqforx}.
Thus, the total number of choices for $x_1,x_2+x_3,x_4$ is
\begin{align*}
\ll \max\{\tfrac{A}{\mathsf{m}p^{\ell}},1\}\max\{\tfrac{\sqrt{D}}{p^{4n-\ell}\mathsf{m}},1\} \max\{\tfrac{C}{p^{4n}\mathsf{m}},1\} 
\ll \max\{ \tfrac{D^{3/2}}{p^{8n}\mathsf{m}},1\}
\end{align*}
where we used $p^{4n} < \sqrt{D}$.
With $x_1,x_2+x_3,x_4$ determined, we wish to apply Lemma~\ref{lem:Pell}.
Let $a=(Cx_1-Ax_4)^2 + (Bx_1-A(x_2+x_3))(Bx_4-C(x_2+x_3))$.
If $a \neq 0$, Lemma~\ref{lem:Pell} shows that $x$ is determined by $x_1,x_2+x_3,x_4$ up to $\ll D^\varepsilon$ many choices. 
By Lemma~\ref{lem:degenerate S} the number of $x\in \mathcal{S}(n)$ with $a=0$ is $\ll \frac{\sqrt{D}}{p^{4n}}$.
Overall, this concludes the above claim and hence this Case 1.2.

\medskip
\textsc{Case 2:} $x_0 =0$.

In this case, the quadratic form $Q_x$ is degenerate (i.e.~has discriminant zero).
We shall denote by $Q_x^i$ the non-degenerate quadratic form obtained by omitting the $i$-th row and column in $Q_x$.
For any $x$ in the image of $\Phi$, one of the quadratic forms $Q_x^3$ and $Q_x^4$ is non-degenerate and similarly for $Q_x^1$ and $Q_x^2$ (because $\mathcal{X}(Q',2n)$ consists of pairs of primitive representations with distinct images).

\textsc{Case 2.1:} \textit{There exists $i$ so that $Q_x^i$ is non-degenerate and $\mathsf{m}_2'(Q_x^i) \leq D^\delta$.}

In this case, the same argument as for $D$ fundamental (based on Corollary~\ref{cor:app-codim1}) applies to show that the contribution of such points is $\ll D^\delta \#\mathcal{S}(n)$.

\textsc{Case 2.2:} \textit{The form $Q_x^i$ is degenerate for two $i\in \{1,\ldots,4\}$ and otherwise $\mathsf{m}_2'(Q_x^i) > D^\delta$.}

For concreteness, suppose $Q_x^i$ is degenerate for $i=2,4$. Then 
\begin{align*}
Ax_2^2 - Bx_1x_2 + C x_1^2 = AD,\\
Ax_3^2 - Bx_1x_3 + C x_1^2 = AD.
\end{align*}
In particular, $(x_1,x_2,x_3)$ are determined up to $\ll_\varepsilon D^\varepsilon$ choices by Lemma~\ref{lem:Pell}.
Moreover, since $Q_x^1$ is non-degenerate and $\mathsf{m}_2'(Q_x^1) > D^\delta$, one can show as in the proof of Lemma~\ref{lem: non-trivial m2} that $x_4$ is determined modulo $\mathsf{m}$ for some $\mathsf{m} \mid \mathsf{m}_2(Q_x^1)$ with $\frac{\mathsf{m}_2(Q_x^1)}{\mathsf{m}} \leq 4$. 
Therefore, $x_4$ has $\ll \max\{\frac{C}{p^{4n}\mathsf{m}},1\}$ choices. 
Together with Corollary~\ref{cor:app-codim1} this concludes this case when $Q_x^i$ is degenerate for $i=2,4$.
All other cases are treated similarly.

\textsc{Case 2.3:} \textit{The form $Q_x^i$ is degenerate for at most one $i\in \{1,\ldots,4\}$ and otherwise $\mathsf{m}_2'(Q_x^i) > D^\delta$.}

For any of the three non-exceptional $i$, we can apply the argument in the proof of Lemma~\ref{lem: non-trivial m2} to the form $Q_x^i$ to find congruence conditions on two of the variables $x_1,\ldots,x_4$ (for instance, $Q_x^4$ yields congruence conditions for $x_1,x_3$).
Fixing the signs $\epsilon$ in each of these congruence conditions, we obtain that, after a loss of $D^{o(1)}$, each $x_i$ is determined modulo $\mathsf{m}_i$ for some divisor $\mathsf{m}_i$ of $4D$ with $\mathsf{m}_i\gg D^{2\delta}$.
In particular, $x_2+x_3$ is determined modulo $\mathsf{m}_{23} = \gcd(\mathsf{m}_2,\mathsf{m}_3)$.
As in Case 1.2 one argues that the total number of choices for $x\in \mathcal{S}(n)$ is
\begin{align*}
&\ll_\varepsilon D^\varepsilon \max\{\tfrac{A}{\mathsf{m}_1},1\} \max\{\tfrac{\sqrt{D}}{\mathsf{m}_{23}p^{4n}},1\} \max\{ \tfrac{C}{\mathsf{m}_4 p^{4n}},1\}
\\
&\ll_\varepsilon  D^\varepsilon\max\big\{\tfrac{D^{3/2}}{\mathsf{m}_4 p^{8n}}, \tfrac{D}{\mathsf{m}_1 p^{4n}}, \tfrac{\sqrt{D}}{\mathsf{m}_{23}p^{4n}},1\big\}
\end{align*}
using $p^{4n} < \sqrt{D}$.
We now distinguish cases (I), (II), (III), (IV) according to where the latter maximum is attained (ordered linearly). 
In Case (I), we apply Corollary~\ref{cor:app-codim1} to either $Q_x^1$ or $Q_x^3$ (at least one of these is non-degenerate) and obtain a total count of $\ll_\varepsilon \tfrac{D^{3/2+\varepsilon}}{\sqrt{\mathsf{m}_4} p^{8n}} \leq \tfrac{D^{3/2-\delta+\varepsilon}}{p^{8n}}$. Case (II) is analogous using $Q_x^2$ or $Q_x^4$.
For Case (III) assume without loss of generality that $\mathsf{m}_{2} \leq \mathsf{m}_{3}$. Note that $\tfrac{\sqrt{D}}{\mathsf{m}_{23}p^{4n}} \geq 1$ implies $\mathsf{m}_2^2 \leq \mathsf{m}_2\mathsf{m}_3 \leq \mathrm{lcm}(\mathsf{m}_2,\mathsf{m}_3) \gcd(\mathsf{m}_2,\mathsf{m}_3) \leq D^{3/2} p^{-4n}$.
Applying Corollary~\ref{cor:app-codim1} to $Q_x^2$ or $Q_x^3$ we obtain a total count in this case of $\ll_\varepsilon \sqrt{\mathsf{m}_2}\,\tfrac{D^{1/2+\varepsilon}}{p^{4n}} \leq D^{7/8+\varepsilon} p^{-5n}$.
In Case (IV), the total count is $\ll_\varepsilon \sqrt{\mathsf{m}_1}D^\varepsilon\ll D^{1/2+\varepsilon}$ by Corollary~\ref{cor:app-codim1} and we conclude.
\end{proof}

\subsection{Bounds on $\mathcal{S}(n)$}

With Lemma~\ref{lem:pair correlation to vectors} and Proposition~\ref{prop:from vectors to inner products} we have reduced the claimed self-correlation bounds in Theorem~\ref{thm:selfcorrelation} to a more tractable problem on counting certain (half-)integral points on a hypersurface of degree four. 
Specifically, we are reduced to give bounds on the set $\mathcal{S}(n)$ in Definition~\ref{def:set S}.

\subsubsection{Trivial bounds}
For illustration purposes, the following lemma establishes what could be considered a `trivial' bound on $\mathcal{S}(n)$.

\begin{lemma}\label{lem:trivial bound}
We have $\#\mathcal{S}(n) \ll \frac{D^{2}}{p^{12n}}$ for any $n\in \N$ with $p^{4n}< \sqrt{D}$.
\end{lemma}

Applying Lemma~\ref{lem:pair correlation to vectors} and Proposition~\ref{prop:from vectors to inner products} in addition, this shows that for all $n\in \N$
\begin{align*}
\mu_\iota' \times \mu_\iota'&(\{(x,y): y \in x \Bow(n)\}) 
\ll_{\varepsilon} D^{-\frac{1}{2}+\varepsilon} + D^{1+\varepsilon}p^{-12n}.
\end{align*}
Choosing $n$ to roughly balance the two terms on the right, we have
\begin{align*}
\mu_\iota' \times \mu_\iota'&(\{(x,y): y \in x \Bow(n)\}) 
\ll_{\varepsilon} p^{-(4-\varepsilon)n}.
\end{align*}
which falls just short of Theorem~\ref{thm:selfcorrelation}.

\begin{proof}[Sketch of proof of Lemma~\ref{lem:trivial bound}]
As we will not use Lemma~\ref{lem:trivial bound} later, we will assume $p \nmid A$.
In this case, we make a point of \emph{not} using the degree $4$ equation in \eqref{eq:degree4 equation}, but only the weak information that $x_0$ is determined by $(x_1,\ldots,x_4)$ up to $D^{o(1)}$ choices.
The degree four equation will eventually yield a small improvement over the trivial bound in Lemma~\ref{lem:trivial bound}.

If $x \in \mathcal{S}(n)$ satisfies $x_1=\pm A$ then $x_2,x_3,x_0$ are determined by \ref{item:S3} and $x_4$ is determined modulo $p^{4n}$, and so the contribution of such points to $\mathcal{S}(n)$ is $\ll \frac{C}{p^{4n}} \ll \frac{D^{2}}{p^{12n}}$.

Let $4n \geq k \geq 0$ be fixed.
Consider the set $\mathcal{S}_k$ of triples $(x_1,x_2+x_3,x_4)$ satisfying
\begin{itemize}
    \item the linear congruence conditions in \eqref{eq:originalcongeqforx},
    \item the Archimedean bounds $|x_1| < A$, $|x_2+x_3| \leq \sqrt{AC}$, $|x_4| \leq C$,
    \item if $k< 4n$ then either $p^k \Vert (x_1-A)$ or $p^k \Vert (x_1+A)$,
    \item if $k= 4n$ then either $p^k \mid (x_1-A)$ or $p^k \mid (x_1+A)$.
\end{itemize}
Notice that $p^k \leq 2A$ or $\mathcal{S}_k = \emptyset$.
Moreover, we have $\#\mathcal{S}_k \ll D^{\frac{3}{2}}p^{-8n-k}$.
Indeed, if we fix $x_1$ (for which there are $\ll \max\{Ap^{-k},1\}\ll Ap^{-k}$ choices) then $x_2+x_3$ and $x_4$ are determined modulo $p^{4n}$ by the linear congruence conditions, which proves the desired estimate.

We now count $x\in \mathcal{S}(n)$ with given $(x_1,x_2+x_3,x_4)\in \mathcal{S}_k$.
We estimate the number of possible choices of $x_2$.
Note that $x_2$ satisfies the, by now, single-variable polynomial congruence equation
\begin{align*}
x_2^2 -(x_2+x_3)x_2 +x_1x_4-D \equiv 0 \mod p^{4n}.
\end{align*}
The discriminant of this polynomial modulo $p^{4n}$ is congruent to
\begin{align*}
(x_2+x_3)^2-4(x_1x_4-D) \equiv -4D ((x_1/A)^2-1)
\end{align*}
using the linear congruence conditions. In particular, if $k< 4n$, it is divisible by $p^k$ and not by $p^{k+1}$.
Thus, if $\alpha,\beta \in \Z_p$ denote the two roots then $p^{\lfloor k/2\rfloor} \Vert (\alpha-\beta)$. Hence, if $(x_2-\alpha)(x_2-\beta)\equiv 0 \mod p^{4n}$ then either $x_2 \equiv \alpha \mod p^{4n-\lfloor k/2\rfloor-1}$ or $x_2 \equiv \beta \mod p^{4n-\lfloor k/2\rfloor-1}$.
The number of $x_2$ satisfying these congruence conditions and is clearly $\ll \sqrt{D}p^{-4n+k/2}$.
The total contribution for $k < 4n$ is thus $\ll D^{3/2}p^{-8n-k} \sqrt{D}p^{-4n+k/2} \ll D^2p^{-12n-k/2}$ which sums to $\ll \frac{D^{2}}{p^{12n}}$.

If $k=4n$, then $x_1 \equiv A \mod p^{4n}$, $x_2+x_3 \equiv B \mod p^{4n}$, and $x_4 \equiv C\mod p^{4n}$. By the quadratic congruence condition in \eqref{eq:originalcongeqforx}, we have $(x_2-\frac{B}{2})^2 \equiv 0 \mod p^{4n}$ and so $x_2$ is determined modulo $p^{2n}$. This case yields a total count of $\ll D^{2}p^{-14n}$ which suffices.
\end{proof}

\subsubsection{Non-trivial bounds}

Our main counting estimate is the following.
Recall that we assume $A= \min(q) \geq c_0D^{\delta_0}$.

\begin{proposition}\label{prop:maincountingestimate}
There exists $\eta = \eta(\delta_0)\in (0,\frac{1}{4})$ with the following property:
Let $n\in \N$ be maximal with $p^{4n}\leq D^{\frac{1}{2}-\eta}$.
Then 
\begin{align*}
\#\mathcal{S}(n)\ll_{\varepsilon} D^{\frac{1}{2}+\varepsilon-\eta} \leq D^{1+\varepsilon} p^{-\frac{4n}{1-2\eta}}.
\end{align*}
\end{proposition}

\begin{proof}[Proof of Theorem~\ref{thm:selfcorrelation} assuming Proposition~\ref{prop:maincountingestimate}]
By Propositions~\ref{prop:from vectors to inner products} and \ref{prop:maincountingestimate} we have for $n\in \N$ maximal with $p^{4n}\leq D^{\frac{1}{2}-\eta}$ and for any $\delta \in (0,\frac{1}{4})$
\begin{align*}
\max_{Q' \in \spn(Q)} \mathcal{X}(Q',2n)
&\ll_{\varepsilon} D^{\delta+\varepsilon} \#\mathcal{S}(n) + \frac{D^{\frac{3}{2}-\delta + \varepsilon}}{p^{8n}} + \frac{D^{7/8+\varepsilon}}{p^{5n}}  +D^{\frac{1}{2}+\varepsilon}\\
&\ll_{\varepsilon} D^{\frac{1}{2}+\delta-\eta+\varepsilon}+D^{\frac{1}{2}-\delta+2\eta+\varepsilon} +D^{\frac{1}{2}+\varepsilon}.
\end{align*}
We choose $\delta=\frac{3}{2}\eta$ so that $\max_{Q' \in \spn(Q)} \mathcal{X}(Q',2n)\ll_\epsilon D^{\frac{1+\eta}{2}+\epsilon}$. By Lemma \ref{lem:pair correlation to vectors} it follows that
\begin{align*}
\mu_\iota' \times \mu_\iota'(\{(x,y): y \in x \Bow(n)\})&\ll_\varepsilon D^{-\frac{1}{2}+\varepsilon} + D^{-1+\varepsilon} \max_{Q'\in \spn(Q)} \#\mathcal{X}(Q',2n)\\
&\ll D^{-\frac{1}{2}+\varepsilon}+D^{-\frac{1-\eta}{2}+\varepsilon}
\ll D^{-\frac{1-\eta}{2}+\varepsilon}
\end{align*}
Hence, the desired estimate follows for $\delta_1=\frac{4\eta}{1-2\eta}$.
\end{proof}

\section{Solving the counting problem}\label{sec:solve counting}

The goal of this section is to prove Proposition~\ref{prop:maincountingestimate} and hence our main theorems.
We use the notation of the previous section.
In particular, we have a Minkowski-reduced positive definite integral binary quadratic form $q(x,y) = Ax^2+Bxy+Cy^2$ of discriminant $D=AC-\tfrac{1}{4}B^2>0$ and an odd prime for which $q(x,y) \mod p$ is non-degenerate.
We fix a prime power $p^{n} \leq D^{\frac{1}{8}}$ so that $n$ close to $\frac{1}{8}\log_p(D)$ and to be determined.
The goal to estimate the size of $\mathcal{S}(n)$; we write $\mathcal{S}=\mathcal{S}(n)$ to simplify notation (keeping $n$ implicit).

\subsection{Small minimum}
First, we cover `intermediate' ranges of $\min(q)=A$, the minimum of $q$.

\begin{lemma}\label{lem:lowrangeA}
Assume $p^{4n}>2 A$.
    We have
    $$\#\mathcal{S}\ll_{\varepsilon} \frac{AD^{\frac{3}{2}+\varepsilon}}{p^{12n}}+\frac{C}{p^{4n}}.$$
\end{lemma}

\begin{proof}
    We follow a scheme similar in parts to the proof of Lemma~\ref{lem:trivial bound}.
    Assume first that $p \nmid B$ and restrict the count to points with $p^k \Vert (x_2+x_3- B)$ (or similarly $p^k \Vert (x_2+x_3+B)$.
    Let $e< 4n$ be the integer such that $p^e\| A$. Since $p\nmid B$ and $Bx_1-A(x_2+x_3)\equiv0 \textrm{ mod }p^{4n}$, we have $p^e|x_1$ and $(p^{-e}A)(x_2+x_3)-B(p^{-e}x_1)\equiv 0 \textrm{ mod }p^{4n-e}$.
    Using $p^k \Vert (x_2+x_3- B)$ we obtain that either
    \begin{enumerate}[a)]
        \item $4n-e \geq k$ and $p^{-e}x_1\equiv p^{-e}A\mod p^k$ or
        \item $4n-e < k$ and $p^{-e}x_1\equiv p^{-e}A\mod p^{4n-e}$.
    \end{enumerate}
    
    We assume Case a) first.
    If $p^{e+k}> 2A$ we have $x_1 = A$, $x_2=x_3 = \frac{B}{2}$ (by \ref{item:S3}), and $x_4 \equiv C \mod p^{4n}$ (by \eqref{eq:originalcongeqforx}) so that $x$ is determined up to $\frac{C}{p^{4n}}$ choices.
    Otherwise, there are $\ll p^{-e-k}A$ many options for $x_1$.
    Once $x_1$ is fixed, the linear congruence equation $A(x_2+x_3)-Bx_1\equiv 0 \textrm{ mod }p^{4n}$ determines $x_2+x_3$ mod $p^{4n-e}$, hence there are $\ll p^{-(4n-e)}D^{\frac{1}{2}}$ many options for $x_2+x_3$. We also fix $x_2+x_3$.
    By \eqref{eq:originalcongeqforx}, $x_4$ is determined modulo $p^{4n}$.
    
    We claim that the quadratic congruence condition $x_1x_4-x_2x_3\equiv D \textrm{ mod }p^{4n}$, determines $x_2$ mod $p^{4n-k}$.
    Indeed, as in the argument for Lemma~\ref{lem:trivial bound} the polynomial $x_1x_4-x_2x_3\equiv D$ viewed as a polynomial in $x_1,x_2+x_3,x_4$ has discriminant 
    \begin{align*}
    (x_2+x_3)^2-4(x_1x_4-D) \equiv -4D \big(((x_2+x_3)/B)^2-1\big) \mod p^{4n}
    \end{align*}
    which is zero exactly mod $p^{k}$ proving the claim.
    Then there are $\ll p^{-4n+k}D^{\frac{1}{2}}$ many options for $x_2$ and we fix one such option so that $x_1,x_2,x_3$ are determined.
    
    For fixed $x_1,x_2,x_3$, there are $D^{o(1)}$ many $x_4$ satisfying the equation \eqref{eq:degree4 equation} unless $x_1=A,x_2=x_3=\frac{B}{2}$.
    Indeed, this follows from Lemma~\ref{lem:Pell} unless the right-hand side (RHS) of \eqref{eq:degree4 equation} is $\disc(Q)$ times the perfect square of a rational polynomial in $x_4$. 
    Observe that the coefficient of $x_4^2$ in the RHS is $x_1^2-A^2$, hence non-positive. It follows that the RHS can become $\disc(Q)$ times a perfect square only if $x_1=A$. In this case we have $x_0=0,x_2=x_3=\frac{B}{2}$, and $x_4\equiv C \textrm{ mod }p^{4n}$. We thus conclude
    \begin{align*}
        \#\mathcal{S}&\ll_{\varepsilon} (p^{-e-k}A)(p^{-4n+e}D^{\frac{1}{2}})(p^{-4n+k}D^{\frac{1}{2}})D^{\varepsilon}+\frac{C}{p^{4n}}
        \leq \frac{AD^{1+\varepsilon}}{p^{8n}}+\frac{C}{p^{4n}}.
    \end{align*}

    We now assume Case b) so that $x_1\equiv A$ mod $p^{4n}$ for any $x\in \mathcal{S}$.
    As $p^{4n} > 2A$, this implies $x_1 = A$ (as $|x_1| \leq A$) and we conclude as in the beginning of a).

    If $p \mid B$, then $p \nmid A$ (as $p\nmid D$) and we count points subject to $p^k \mid (x_1-A)$ (or $p^k \mid (x_1+A)$). The proof is very similar. With these conditions, $x_1$ is determined up to $\ll Ap^{-k}$ choices (unless $p^k > 2A$ in which case we invoke \ref{item:S3}). Thus, $x_2+x_3$ is determined up to $\ll \frac{\sqrt{D}}{p^{4n}}$ choices and, fixing one such choice, $x_4$ is determined modulo $p^{4n}$.
    Using the quadratic congruence condition, $x_2$ is determined modulo $p^{4n-k}$ and we conclude as before.
\end{proof}

Let $\delta_0':=\min\{2\delta_0,\frac{1}{126}\}$. 
We first prove Proposition~\ref{prop:maincountingestimate} for the range 
\begin{align*}
c_0D^{\delta_0}\leq A=\min(q)< D^{\frac{1}{2}-\delta_0'}.
\end{align*}

\begin{proof}[Proof of Proposition~\ref{prop:maincountingestimate} when $ A< D^{\frac{1}{2}-\delta_0'}$]
Let $\eta=\tfrac{1}{4}\delta_0'$. We choose $n$ so that $p^{4n}\leq D^{\frac{1}{2}-\eta}<p^{4n+1}$.
Note that $p^{4n}>2A$ as $p^{4n+1}>D^{\frac{1}{2}- \eta}> 2pA$. 
We thus have
$$\#\mathcal{S}\ll_{\varepsilon} \frac{AD^{\frac{3}{2}+\varepsilon}}{p^{12n}}+\frac{A^{-1}D}{p^{4n}}\ll D^{\frac{1}{2}+\varepsilon-\eta}$$
by Lemma \ref{lem:lowrangeA} and the proposition follows in this range.
\end{proof}
In the rest of this section, we assume $A\geq D^{\frac{1}{2}-\delta_0'}$.

\subsection{Bad sets}
For the case $A\geq D^{\frac{1}{2}- \delta_0'}$, we shall first exclude the points in
\begin{align*}
    \mathcal{S}_{\mathrm{deg}}=\{\mathbf{x}\in\mathcal{S}: Cx_1-Ax_4=Bx_1-A(x_2+x_3)= 0\}
\end{align*}
and then count the number of points in
$$\mathcal{S}_0:=\mathcal{S}\setminus \mathcal{S}_{\mathrm{deg}}.$$

Recall from Lemma~\ref{lem:degenerate S} that we have
\begin{align}\label{eq:degcount}
\#\mathcal{S}_{\operatorname{deg}}\ll p^{-4n}\sqrt{D}.
\end{align}

We shall also need the following elementary lemma.

\begin{lemma}\label{eq:squarelemma}
    Let 
    \begin{align*}
    Q_1&=x_1x_4-x_2x_3-D,\\
    Q_2 &= (Cx_1-Ax_4)^2+(Bx_1-A(x_2+x_3))(Bx_4-C(x_2+x_3)).
    \end{align*}
    Let $P$ be a real affine plane in $\mathbf{A}^4$.
    \begin{enumerate}
        \item Suppose for $\alpha, \beta\in\R$, $Q_1+\alpha Q_2 + \beta$ vanishes identically on $P$. Then we have $\alpha=(4D)^{-1}$ or $\beta = D$.
        \item Suppose $(Q_1^2-Q_2)|_P$ is a square of a polynomial function on $P$ with real coefficients, then either $P$ is defined by $x_2-x_3=Cx_1+Ax_4-Bx_2=0$, or $P$ is defined by $Cx_1-Ax_4=Bx_4-C(x_2+x_3)=0$.
    \end{enumerate}
\end{lemma}

\begin{proof}
    (1) Write coordinates on the plane as $v_1s_1+v_2s_2+u$ in variables $s_1,s_2$.
    Let $Q'=Q_1'+\alpha Q_2$, a quadratic form, where $Q_1'=Q_1+D=x_1x_4-x_2x_3$.
    If $Q_1+\alpha Q_2 + \beta$ vanishes identically on $P$ or equivalently
    $Q'(v_1s_1+v_2s_2+u)=-\beta+D$ for all $s_1,s_2$, then --- by reading off the coefficients of the polynomial equation in $s_1,s_2$ -- we have $Q'(v_1)=Q'(v_2) =0$ and $\langle v_1,v_2\rangle_{Q'} = \langle v_1,u\rangle_{Q'} = \langle v_2,u\rangle_{Q'} = 0$.
    In particular, the plane spanned by $v_1,v_2$ is a totally isotropic plane with respect to $Q'$ and $u$ is orthogonal to that plane.
    Then either $u$ is contained in the plane or $Q'$ is a degenerate quadratic form i.e. of discriminant zero.
    In the former case, we may assume $u=0$ by changing the coordinates and so $\beta = D$ follows from $s_1=s_2=0$.
    In the latter case, the discriminant $\frac{16D^2\alpha^2-8D\alpha+1}{16}$ equals $0$, and thus $\alpha=(4D)^{-1}$.
    
    (2) Write $(Q_1^2-Q_2)|_P=Q_3^2$ for some $Q_3\in\R[P]$. We have $Q_2|_P=(Q_1|_P+Q_3)(Q_1|_P-Q_3)$. 
    
    We claim that either $Q_1|_P+Q_3$ or $Q_1|_P-Q_3$ is constant. Indeed, suppose both are nonconstant, then both are affine linear. Then $Q_2|_P$ decomposes as a product of two affine linear forms, and $Q_1|_P$ is affine linear. Let $P'$ be the plane through the origin which is parallel to $P$.
    Then, by homogenizing the affine coordinates of $P'$ and specializing to the boundary of $P'$, $Q_2|_{P'}$ is a product of two linear forms, and $P'$ is totally isotropic for $Q_1'$. Notice that the restriction of $Q_2$ to any plane is either positive definite or degenerate. Hence $Q_2|_{P'}$ is degenerate, which implies that $Cx_1-Ax_4$ and $Bx_4-C(x_2+x_3)$ are linearly dependent on $P'$. In other words, there exists $(a,b)\neq (0,0)$ such that $P'$ is contained in the hyperplane $H_{a,b}$ defined by $a(Cx_1-Ax_4)+b(Bx_4-C(x_2+x_3))=0$. However, one can check the restriction of $Q_1'$ to the hyperplane $H_{a,b}$ is always a non-degenerate tenary form. Therefore, $P'$ is not totally isotropic for $Q_1'$, a contradiction.
    
    Now, without loss of generality we may assume that $Q_1|_P-Q_3$ equals a constant $c$, and thus $Q_2|_P=c(2Q_1|_P-c)$. If $c=0$, then $Q_2=0$, which implies that $Cx_1-Ax_4=Bx_4-C(x_2+x_3)=0$. Now we suppose that $c\neq0$. Hence $Q_1|_P=\frac{1}{2}(c^{-1}Q_2|_P+c)$. By (1) we have $c=-2D$, and thus $Q_1|_P=-(4D)^{-1}Q_2|_P-D$. But 
    \begin{align*}
    Q_1+(4D)^{-1}Q_2=(4D)^{-1}(Cx_1+Ax_4-\tfrac{1}{2}B(x_2+x_3))^2+\tfrac{(x_2-x_3)^2-4D}{4}
    \end{align*}
    where we note that $D>0$.
    Hence $x_2-x_3=Cx_1+Ax_4-Bx_2=0$.
\end{proof}

\begin{lemma}\label{lem:square-part1}
Let $F$ be a complex polynomial of the form
\begin{align*}
F(x) = (ax^2+bx+c)^2 - dx^2-ex-f.
\end{align*}
If $F$ is a square in $\C[x]$, the coefficients satisfy
\begin{align}\label{eq:square-rel1}
ae = bd,\quad 4a^2 f = (4ac-d)d
\end{align}
or
\begin{align}\label{eq:square-rel2}
a=0,\quad (4bc-e)e = 4(b^2f + c^2d -df). 
\end{align}
\end{lemma}

\begin{proof}
Write $F(x) = (\alpha x^2 + \beta x + \gamma)^2$ and compare coefficients to obtain
\begin{align*}
&a^2 =  \alpha^2,\ ab =  \alpha \beta,\ b^2+2ac-d=(\beta^2 +2\alpha\gamma),\\
&2bc-e=2 \beta \gamma,\ c^2 - f =  \gamma^2.
\end{align*}

Suppose first $a \neq 0$, in which case $\alpha \neq 0$. Then $ab =  \alpha \beta$ implies $b^2 =  \beta^2$ and so $2ac -d = 2 \alpha \gamma$. Taking the square of the latter equation yields
\begin{align*}
(2ac-d)^2 = 4 a^2 \gamma^2 = 4a^2(c^2-f)
\end{align*}
and hence the second claim equation by canceling the $4a^2c^2$ term on both sides.
Notice also that $\frac{\beta}{\alpha} = \frac{\alpha\beta}{\alpha^2} = \frac{b}{a}$. Thus $2bc- e =  2 \alpha \gamma \frac{\beta}{\alpha} = (2ac-d) \frac{b}{a}$ which proves the first claimed equation.
The case $a=\alpha= 0$ is treated similarly.
\end{proof}

\subsection{Congruence equations}
Let $\mathcal{W}\subset\{0,\ldots,p^{2n}-1\}^4$ denote the reduction of the set $\mathcal{S}_0$ mod $p^{2n}$, forgetting the $x_0$-coordinate.
Then any $(x_1,x_2,x_3,x_4)\in \mathcal{W}$ satisfies
\begin{align*}
x_1x_4-x_2x_3 &\equiv D \mod p^{2n},\\
Cx_1-Ax_4 &\equiv 0 \mod p^{2n},\\
C(x_2+x_3)-Bx_4 &\equiv 0 \mod p^{2n},\\
A(x_2+x_3)-Bx_1 &\equiv 0 \mod p^{2n}.
\end{align*}
The latter three linear equations imply that $(x_1,x_2+x_3,x_4)\equiv (A,B,C)t$ mod $p^{2n}$ for some $t\in\mathbb{Z}/p^{2n}\mathbb{Z}$. Inserting this into the first equation, we have
$$(x_2-Bt)^2-D(t^2-1)\equiv 0 \;\mod p^{2n}.$$
For each $t$ with $p^e\| t\pm 1$, $x_2$ is determined mod $p^{2n-e}$ by the first equation, hence the total count is $\sum_{e=0}^{2n}p^e(p^{2n-e})\ll np^{2n}$.

We split the points in $\mathcal{S}$ into their congruence classes mod $p^{2n}$. For $\mathbf{w}\in \mathcal{W}$ set
    \begin{align*}
    \mathcal{S}(\mathbf{w}):=\{x\in \mathcal{S}_0: (x_1,x_2,x_3,x_4)\equiv \mathbf{w}\; \mod p^{2n}\}.
    \end{align*}
For each congruence class $\mathbf{w}\in\mathcal{W}$, we may choose a representative from $\mathcal{S}(\mathbf{w})$ (otherwise, $\mathcal{S}(\mathbf{w})=\emptyset$ and we discard $\mathbf{w}$) and denote by $\mathbf{w}$ again by abuse of notation, so that $\mathbf{w}=(w_1,w_2,w_3,w_4)$ indeed satisfies \eqref{eq:originalcongeqforx}, i.e,
\begin{equation}\label{eq:congeqforw}
    \begin{aligned}
w_1w_4-w_2w_3 &\equiv D \mod p^{4n},\\
Cw_1-Aw_4 &\equiv 0 \mod p^{4n},\\
C(w_2+w_3)-Bw_4 &\equiv 0 \mod p^{4n},\\
A(w_2+w_3)-Bw_1 &\equiv 0 \mod p^{4n}.
\end{aligned}
\end{equation}

We take the following substitutions:
\begin{equation}\label{eq:xysubstitution}
    x_i=p^{2n}y_i+w_i
\end{equation}
for $1\leq i\leq 4$, where 
\begin{align}\label{eq:boundonyi's}
|y_1|\ll p^{-2n}A,\ |y_2|,|y_3|\ll p^{-2n}\sqrt{D},\ |y_4|\ll p^{-2n}C.
\end{align}
The systems of congruence equations \eqref{eq:originalcongeqforx} and \eqref{eq:congeqforw} imply the following system of congruence equations in the new variables:
\begin{equation}\label{eq:congeqfory}
    \begin{aligned}
w_4y_1+w_1y_4-w_2y_3-w_3y_2 &\equiv 0 \mod p^{2n},\\
Cy_1-Ay_4 &\equiv 0 \mod p^{2n},\\
C(y_2+y_3)-By_4 &\equiv 0 \mod p^{2n},\\
A(y_2+y_3)-By_1 &\equiv 0 \mod p^{2n}.
\end{aligned}
\end{equation}
We have thus `linearized' the quadratic equation in \eqref{eq:originalcongeqforx}.

For each $\mathbf{w}$ let $\Lambda_{\mathbf{w}}$ denote the sublattice of $\mathbb{Z}^4$ consisting of $(y_1,y_2,y_3,y_4)\in\mathbb{Z}^4$ satisfying \eqref{eq:congeqfory}. 

\begin{lemma}\label{lem:covolume of lattice in y}
    For $\mathbf{w}\in \mathcal{W}$ we have $\operatorname{covol}(\Lambda_{\mathbf{w}})= p^{6n}$ with respect to the Euclidean norm.
\end{lemma}
\begin{proof}
    Since $\Lambda_{\mathbf{w}}$ contains $p^{2n}\mathbb{Z}^4$, it suffices to count the number of solutions to \eqref{eq:congeqfory} mod $p^{2n}$. The latter three conditions in \eqref{eq:congeqfory} are equivalent to $(y_1,y_2+y_3,y_4) \equiv (A,B,C)t \mod p^{2n}$ for some $t$, hence there are $p^{2n}$ many $(y_1,y_2+y_3,y_4)$ mod $p^{2n}$. 
    We may write the first equation of \eqref{eq:congeqfory} as follows:
    $$(w_2-w_3)y_2 \equiv -w_4y_1-w_1y_4+w_2(y_2+y_3) \mod p^{2n}.$$

    We claim that there are $p^{2n}$ many solutions to \eqref{eq:congeqfory}, i.e. $[\Lambda_{\mathbf{w}}:p^{2n}\mathbb{Z}^4]= p^{2n}$. 
    Let $0 \leq \nu'\leq 2n$ be maximal such that $p^{\nu'}\mid (w_2-w_3)$.

    If $\nu'=0$, then $y_2$ is determined mod $p^{2n}$ for each $(y_1,y_2+y_3,y_4)$. Thus, there are $p^{2n}$ many solutions to \eqref{eq:congeqfory}, i.e. $[\Lambda_{\mathbf{w}}:p^{2n}\mathbb{Z}^4]= p^{2n}$.

    Suppose $\nu'>0$. 
    Write $(y_1,y_2+y_3,y_4) \equiv (A,B,C)t \mod p^{2n}$ for some $t$ using the latter three conditions in \eqref{eq:congeqfory}. For the first congruence condition in \eqref{eq:congeqfory} to have a solution we require
\begin{align*}
(-w_4 A -w_1 C + w_2 B)t \equiv 0 \mod p^{\nu'}
\end{align*}
Also, there is $s$ static with
$(A,B,C)s \equiv (w_1,w_2+w_3,w_4) \equiv (w_1,2w_2,w_4) \mod p$ by \eqref{eq:congeqforw} and using $w_2 \equiv w_3 \mod p$.  
Inserting this, $-w_4 A -w_1 C + w_2 B \equiv (-2AC + \frac{1}{2}B^2)s \mod p$. 
If $p\mid s$ then $p \mid w_1,w_4,w_2+w_3$ and as $p\mid w_2- w_3$ we have $p\mid w$, which is impossible by the quadratic congruence condition in \eqref{eq:congeqforw}.
This shows $p \nmid s$, hence $p^{\nu'}\mid t$. Thus, $t$ has $p^{2n-\nu'}$ choices and $y_2$ is determined exactly mod $p^{2n-\nu'}$ and so has exactly $p^{\nu'}$ choices. It follows that there are $p^{2n}$ solutions mod $p^{2n}$. Hence, we have $[\Lambda_{\mathbf{w}}:p^{2n}\mathbb{Z}^4]= p^{2n}$ in this case as well. The claim is proved.
    
    We conclude from the claim that 
    $$\operatorname{covol}(\Lambda_{\mathbf{w}})=[\mathbb{Z}^4:p^{2n}\mathbb{Z}^4][\Lambda_{\mathbf{w}}:p^{2n}\mathbb{Z}^4]^{-1}=  [\mathbb{Z}^4:p^{2n}\mathbb{Z}^4](p^{2n})^{-1}= p^{6n}$$
    as desired.
\end{proof}

We now apply Minkowski's second theorem and we do so with respect to a `distorted' norm to account for the possibility that $A$ could be slightly smaller than $C$.
Consider the norm $\norm{\cdot}_q$ on $\R^4$ given by
\begin{align*}
\norm{t}_q^2 = \big(\tfrac{|t_1|}{p^{-2n}A}\big)^2 + \big(\tfrac{|t_2|}{p^{-2n}D^{1/2}}\big)^2 +\big(\tfrac{|t_3|}{p^{-2n}D^{1/2}}\big)^2 + \big(\tfrac{|t_4|}{p^{-2n}C}\big)^2.
\end{align*}
The covolume of $\Lambda_{\mathbf{w}}$ with respect to this norm is $\gg p^{14n}D^{-2}$ by Lemma~\ref{lem:covolume of lattice in y}.
By Minkowski's second theorem, there exists a reduced $\mathbb{Z}$-basis $\{\mathbf{v}_1,\mathbf{v}_2,\mathbf{v}_3,\mathbf{v}_4\}\subset\mathbb{Z}^4$ of $\Lambda_{\mathbf{w}}$ such that $\|\mathbf{v}_1\|_q\|\mathbf{v}_2\|_q\|\mathbf{v}_3\|_q\|\mathbf{v}_4\|_q\gg p^{14n}D^{-2}$, $\|\mathbf{v}_1\|_q\leq \|\mathbf{v}_2\|_q\leq \|\mathbf{v}_3\|_q\leq \|\mathbf{v}_4\|_q$ and, moreover,
\begin{align}\label{eq:almost orth Minkowski}
\norm{z_1\mathbf{v}_1+z_2\mathbf{v}_2+z_3\mathbf{v}_3+z_4\mathbf{v}_4}_q^2 \asymp \norm{\mathbf{v}_1}_q^2 z_1^2 +\norm{\mathbf{v}_2}_q^2 z_2^2+\norm{\mathbf{v}_3}_q^2 z_3^2+\norm{\mathbf{v}_4}_q^2 z_4^2.
\end{align}

Given any integral point $x \in \mathcal{S}(\mathbf{w})$, the associated point $(y_1,y_2,y_3,y_4)$ (via the substitution \eqref{eq:xysubstitution}) is an element of $\Lambda_{\mathbf{w}}$ by \eqref{eq:congeqfory}.
In particular, we may write
\begin{align}\label{eq:yi coordinates}
(y_1,y_2,y_3,y_4) = z_1 \mathbf{v}_1+z_2\mathbf{v}_2+z_3\mathbf{v}_3+z_4\mathbf{v}_4
\end{align}
for some (unique) $z_1,z_2,z_3,z_4 \in \Z$. 
The bounds \eqref{eq:boundonyi's} imply $\norm{(y_1,y_2,y_3,y_4)}_q\ll 1$.
Combining this with the almost orthogonality property of the vectors $\mathbf{v}_1,\mathbf{v}_2,\mathbf{v}_3,\mathbf{v}_4$ in \eqref{eq:almost orth Minkowski} we obtain
\begin{align*}
|z_i|^2 \norm{\mathbf{v}_i}_q^2 \ll \norm{z_1 \mathbf{v}_1+z_2\mathbf{v}_2+z_3\mathbf{v}_3+z_4\mathbf{v}_4}_q^2 \ll 1
\end{align*}
and hence 
\begin{align}
|z_i| \leq B_i
\end{align}
for some $B_i=B_i(\mathbf{w})>0$ with $B_i \asymp \frac{1}{\norm{\mathbf{v}_i}_q}$.
Notice that $B_1B_2B_3B_4 \asymp D^{2}p^{-14n}$ by construction.

\begin{lemma}\label{lem:Bsizes}
For any $\mathbf{w}\in \mathcal{W}$, we have
\begin{align*}
B_1 \gg \tfrac{C}{p^{4n}},\ B_2,B_3 \gg \tfrac{\sqrt{D}}{p^{4n}},\ B_4 \gg \tfrac{A}{p^{4n}}.
\end{align*}
\end{lemma}

We caution the reader that the above lower bound on $B_4$ might be strictly less than $1$.
However, we will choose $p^{4n}\asymp D^{\frac{1}{2}-\eta}$ later so that $\delta_0'\leq \eta\leq \frac{1}{126}$. Thus, in the rest of the section we may assume $p^{4n}\ll D^{\frac{1}{2}-\delta_0'} \leq A$, hence $B_4\gg 1$.

\begin{proof}
Notice that the lattice $\Lambda_\mathbf{w}$ contains the vectors $p^{2n}\mathbf{e}_i$ for $1\leq i\leq 4$
which satisfy
\begin{align*}
\norm{p^{2n}\mathbf{e}_1}_q = \tfrac{p^{4n}}{C},\
\norm{p^{2n}\mathbf{e}_2}_q = \norm{p^{2n}\mathbf{e}_3}_q = \tfrac{p^{4n}}{\sqrt{D}},\
\norm{p^{2n}\mathbf{e}_4}_q \asymp \tfrac{p^{4n}}{A}.
\end{align*}
and are linearly independent.
In particular, $\norm{p^{2n}\mathbf{e}_i}_q \geq \norm{\mathbf{v}_i}_q$ for $1\leq i\leq 4$. This proves the lemma.
\end{proof}

\begin{lemma}\label{lem:smallB2}
    For any $\delta>0$ and $n\in\mathbb{N}$ with $D^{3\delta+\delta_0'}<p^{2n}$, we have
    $$\#\big\{\mathbf{w}\in \mathcal{W}: B_2(\mathbf{w})\leq D^{\delta}\big\} 
    \ll \left(p^{12n}D^{-\frac{3}{2}+3\delta}\right)^3.$$
\end{lemma}
\begin{proof}
    Recall that $\mathbf{v}_1=\mathbf{v}_1(\mathbf{w})\in\Lambda_{\mathbf{w}}$ denotes a shortest vector in $\Lambda_{\mathbf{w}}$. 
    If $B_2\leq D^{\delta}$, then $B_1\asymp (B_2B_3B_4)^{-1}D^2p^{-14n}\gg D^{2-3\delta}p^{-14n}$, hence
    $$\|\mathbf{v}_1\|_q\asymp  B_1^{-1} \ll p^{14n}D^{-2+3\delta}.$$
    Write $\mathbf{v}_1=(t_1,t_2,t_3,t_4)\in\Lambda_{\mathbf{w}}$. Then we have
    \begin{align*}
|t_1|,|t_2|,|t_3|&\leq \big(p^{-2n}D^{\frac{1}{2}}\big)p^{14n}D^{-2+3\delta}=p^{12n}D^{-\frac{3}{2}+3\delta},\\
|t_4|&\leq \big(p^{-2n}C\big)p^{14n}D^{-2+3\delta}\ll p^{12n}D^{-\frac{3}{2}+3\delta+\delta_0'}.
    \end{align*}
    In particular, we have $|t_i|< p^{2n}$ for all $1\leq i\leq 4$. 
    Note that once $t_2,t_3$ are fixed, the congruence equations in \eqref{eq:congeqfory} determines $t_1,t_4$ modulo $p^{2n}$, provided that $p \nmid B$.
    By the above bounds, this implies that $t_1,t_4$ take $O(1)$ many values.
    When $p\mid B$, we instead fix either $t_1,t_2$ or $t_1,t_3$, in which case the remaining coordinates are again determined up to $O(1)$ many options.
    Thus, the number of integer vectors $(t_1,t_2,t_3,t_4)$ satisfying these conditions is $\ll (p^{12n}D^{-\frac{3}{2}+3\delta})^2$.

We now claim that for each $(t_1,t_2,t_3,t_4)$ there are $\ll p^{12n}D^{-\frac{3}{2}+3\delta}$ many possible $\mathbf{w}\in\mathcal{W}$ with $\mathbf{v}_1(\mathbf{w})=(t_1,t_2,t_3,t_4)$. 
The lemma immediately follows from the claim. 
By \eqref{eq:congeqforw}, there exists $s \in \{1,\ldots,p^{2n}\}$ such that $(w_1,w_2+w_3,w_4) \equiv (A,B,C)s \mod p^{2n}$. 
By \eqref{eq:congeqfory}, we have
\begin{equation}\label{eq:congeqq1}
    \begin{aligned}
    (t_3-t_2)w_2&\equiv t_4w_1+t_1w_4-t_2(w_2+w_3) \equiv (At_4+Ct_1-Bt_2)s \mod p^{2n}
\end{aligned}
\end{equation}
and by the quadratic congruence condition in \eqref{eq:congeqforw}
\begin{equation}\label{congeqq2}
w_2^2 -Bw_2s+ACs^2 \equiv D \mod p^{2n}
\end{equation}
Let $r\geq 0$ be such that $p^r \Vert (t_3-t_2)$. Since $|t_2-t_3|\ll p^{12n}D^{-\frac{3}{2}+3\delta}$, we have $p^r\ll p^{12n}D^{-\frac{3}{2}+3\delta}$ and in particular $r < 2n$.
For simplicity, set $\alpha = t_3-t_2$ and $\beta = (At_4+Ct_1-Bt_2)$.
Multiplying \eqref{congeqq2} by $\alpha^2$ we have
\begin{align*}
(\beta^2 -B\alpha\beta + AC \alpha^2) s^2 \equiv D \alpha^2 \mod p^{2n+2r}
\end{align*}
which has either no solutions (if the coefficient of $s^2$ has order bigger than $2r$, we conclude in this case) or determines $s$ uniquely (modulo $p^{2n}$).
In particular, $(w_1,w_2+w_3,w_4)$ is uniquely determined modulo $p^{2n}$.
By \eqref{eq:congeqq1}, $w_2$ is uniquely determined modulo $p^{2n-r}$ and we conclude using $p^r\ll p^{12n}D^{-\frac{3}{2}+3\delta}$.
\end{proof}

\subsection{Proof of Linnik basic lemma}
In this subsection we prove Proposition~\ref{prop:maincountingestimate} when $ A\geq D^{\frac{1}{2}-\delta_0'}$. Let us fix $0<\delta\leq \frac{1}{13}$, to be determined later.
Let $\delta_0'\leq \eta\leq \frac{1}{126}$ and choose the scale $n$ maximal with $p^{4n}\leq D^{\frac{1}{2}-\eta}$. 
With this choice, we have $p^{4n}\leq D^{\frac{1}{2}-\delta_0'} \leq A$, hence $B_4(\mathbf{w})\gg 1$ for any $\mathbf{w}\in \mathcal{W}$. 
Moreover, this choice satisfies the assumption of Lemma~\ref{lem:smallB2}.

We split $\mathcal{W}$ into
\begin{align*}
\mathcal{W}_{\operatorname{bad}}:=\{\mathbf{w}\in\mathcal{W}: B_2(\mathbf{w})\leq D^{\delta}\} \quad \text{and}\quad \mathcal{W}_{\operatorname{good}}:=\mathcal{W}\setminus\mathcal{W}_{\operatorname{bad}}.
\end{align*}
We shall estimate
$$\#\mathcal{S}_0=\sum_{\mathbf{w}\in\mathcal{W}}\#\mathcal{S}(\mathbf{w})=\sum_{\mathbf{w}\in\mathcal{W}_{\operatorname{good}}}\#\mathcal{S}(\mathbf{w})+\sum_{\mathcal{W}_{\operatorname{bad}}}\#\mathcal{S}(\mathbf{w}).$$

Recall that we have $$\#\mathcal{W}_{\operatorname{bad}}\ll \left(p^{12n}D^{-\frac{3}{2}+3\delta}\right)^3$$by Lemma \ref{lem:smallB2}. Using the trivial bound $$\mathcal{S}(\mathbf{w})\ll B_1B_2B_3B_4\asymp D^2p^{-14n},$$ we get
\begin{equation}\label{eq:W2contribution}
\begin{split}
\sum_{\mathbf{w}\in\mathcal{W}_{\operatorname{bad}}}\#\mathcal{S}(\mathbf{w})
&\ll (\#\mathcal{W}_{\operatorname{bad}})D^{2}p^{-14n}\\
&\ll \left(p^{12n}D^{-\frac{3}{2}+3\delta}\right)^3D^2p^{-14n}.
\end{split}
\end{equation}

We now estimate the contribution from $\mathcal{W}_{\operatorname{good}}$. 
Following \eqref{eq:degree4 equation} set
\begin{align*}
f(x_1,x_2,x_3,x_4) &= (x_1x_4-x_2x_3 -D)^2 - (Cx_1-Ax_4)^2 \\
&\qquad- (Bx_1-A(x_2+x_3))(Bx_4-C(x_2+x_3)).
\end{align*}
For $\mathbf{w}\in \mathcal{W}$, with the substitutions \eqref{eq:xysubstitution} and \eqref{eq:yi coordinates}, let $F_\mathbf{w}$ be the polynomial such that $F_{\mathbf{w}}(z_1,z_2,z_3,z_4)=f(x_1,x_2,x_3,x_4)$. In particular, we have
$$F_{\mathbf{w}}(z_1+t,z_2,z_3,z_4)=f\big((x_1,x_2,x_3,x_4)+p^{2n}t\mathbf{v}_1\big),$$
where $\mathbf{v}_1=\mathbf{v}_1(w)$ is the shortest vector in $\Lambda_{\mathbf{w}}$. Let us write $\mathbf{x}=(x_0,x_1,x_2,x_3,x_4)$, $\mathbf{x}'=(x_1,x_2,x_3,x_4)$, and $\mathbf{z}=(z_1,z_2,z_3,z_4)$ for simplicity.

Let $\mathbf{w}\in\mathcal{W}_{\operatorname{good}}$ and set $\mathsf{B}(\mathbf{w})=[-B_1,B_1]\times [-B_2,B_2]\times [-B_3,B_3]\times [-B_4,B_4]$ (a box in the $z$-coordinates). 
We consider
$$\mathcal{Z}(\mathbf{w}):=\big\{\mathbf{z}\in\mathsf{B}(\mathbf{w}):\operatorname{disc}(Q)x_0^2=F_{\mathbf{w}}(\mathbf{z}) \textrm{ for some }x_0\in\mathbb{Z}\big\} $$
and observe that $\#\mathcal{S}(\mathbf{w})\leq 2\#\mathcal{Z}(\mathbf{w})$.

We fix $(\tau_3,\tau_4)\in[-B_3,B_3]\times [-B_4,B_4]$ and set
\begin{align*}
\mathcal{T}^{\square}(\mathbf{w},\tau_3,\tau_4)
&=\big\{\tau_2\in[-B_2,B_2]: F_{\mathbf{w}}(z_1,\tau_2,\tau_3,\tau_4) \textrm{ is a square in }\mathbb{R}[z_1]\big\},\\
\mathcal{T}^{\boxtimes}(\mathbf{w},\tau_3,\tau_4)&=\big\{\tau_2\in[-B_2,B_2]: F_{\mathbf{w}}(z_1,\tau_2,\tau_3,\tau_4) \textrm{ is not a square in }\mathbb{R}[z_1]\big\}.
\end{align*}

\begin{lemma}\label{lem:squarefiniteness}
    For any $\mathbf{w}\in\mathcal{W}_{\operatorname{good}}$, $(\tau_3,\tau_4)\in [-B_3,B_3]\times[-B_4,B_4]$, we have $$\#\mathcal{T}^{\square}(\mathbf{w},\tau_3,\tau_4)\ll 1.$$
\end{lemma}
\begin{proof}
   Let $\mathbf{x}_0\in\mathcal{S}(\mathbf{w})$ be a vector whose last two $z$-coordinates are $(\tau_3,\tau_4)$. 
   Observe that $\tau_2\in \mathcal{T}^{\square}(\mathbf{w},\tau_3,\tau_4)$ if and only if $$F_{\mathbf{w}}(z_1,\tau_2,\tau_3,\tau_4)=f(\mathbf{x}_0'+p^{2n}(z_1-c_1)\mathbf{v}_1+p^{2n}(\tau_2-c_2)\mathbf{v}_2)$$ is square over $\mathbb{R}[z_1]$.
    Here, $c_1,c_2$ are determined by $\mathbf{x}_0'$ and satisfy $|c_1| \leq B_1$ as well as $|c_2| \leq B_2$ (in fact, $c_1$ is the $z_1$-coordinate of $\mathbf{x}_0$ and similarly for $c_2$).
   
    Let $P_{\mathbf{w}}\subset \mathbf{A}^4$ be the affine plane $\{\mathbf{x}_0'+z_1\mathbf{v}_1+z_2\mathbf{v}_2:z_1,z_2\in\mathbb{R}\}$. We shall apply Lemma \ref{eq:squarelemma}(2) for this plane $P_{\mathbf{w}}$. 
    Note that $P_{\mathbf{w}}$ is not contained in the linear subspace given by $x_2-x_3=Cx_1+Ax_4-Bx_2=0$.
    In fact, one may verify that $\mathcal{S}$ intersects that subspace trivially merely using the congruence conditions \eqref{eq:originalcongeqforx}.
    Recall that $\mathcal{W}$ is defined a reduction of $\mathcal{S}_0$, and $\mathcal{S}_0$ is disjoint from the bad set $\mathcal{S}_{\operatorname{deg}}$. It follows that $P_{\mathbf{w}}$ is not contained in the linear subspace given by $Cx_1-Ax_4=Bx_4-C(x_2+x_3)=0$. Thus Lemma \ref{eq:squarelemma}(2) implies that $f|_{P_{\mathbf{w}}}$ is not a square of a real polynomial of $z_1,z_2$. 
    
    We claim that there is a polynomial $p\in \mathbb{R}[z_2]$ of absolutely bounded degree such that $F_{\mathbf{w}}(z_1,\tau_2,\tau_3,\tau_4)$ is square over $\mathbb{R}[z_1]$ only if $p(\tau_2)=0$. This claim is an immediate consequence of Lemma \ref{lem:square-part1}. Indeed, to see the claim we view $F_{\mathbf{w}}(z_1,\tau_2,\tau_3,\tau_4)$ as a univariate polynomial of $z_1$ and apply Lemma \ref{lem:square-part1}. Then $F_{\mathbf{w}}(z_1,\tau_2,\tau_3,\tau_4)$ is square over $\mathbb{C}[z_1]$ if and only if the coefficients satisfy at least one of the polynomial relations \eqref{eq:square-rel1} or \eqref{eq:square-rel2}, where the coefficients are also polynomial in $\tau_2$. Thus we may express the relations \eqref{eq:square-rel1} and \eqref{eq:square-rel2} by some polynomials in $\tau_2$, hence we can find a polynomial $p\in \mathbb{Z}[z_2]$ so that $F_{\mathbf{w}}(z_1,\tau_2,\tau_3,\tau_4)$ is square over $\mathbb{C}[z_1]$ if and only $p(\tau_2)=0$. Moreover, it is straightforward to check that the degree of $p$ is bounded above by an absolute constant. This completes the proof of the claim.
    
   Since $f|_{P_{\mathbf{w}}}$ is not a square of a polynomial of $z_1,z_2$, the polynomial $p$ is not identically zero. Hence, $\#\mathcal{T}^{\square}(\mathbf{w},\tau_3,\tau_4)\leq \deg p=O(1)$. This completes the proof.
\end{proof}

For $(\tau_2,\tau_3,\tau_4)\in [-B_2,B_2]\times [-B_3,B_3]\times [-B_4,B_4]$ we define
$$\mathcal{Z}(\mathbf{w},\tau_2,\tau_3,\tau_4):=\{\mathbf{z}\in\mathsf{B}(\mathbf{w}):(z_2,z_3,z_4)=(\tau_2,\tau_3,\tau_4)\}.$$

For $\mathbf{w}\in \mathcal{W}_{\operatorname{good}}$ we estimate $\mathcal{S}(\mathbf{w})$ as follows:
$$ \tfrac{1}{2}\#\mathcal{S}(\mathbf{w})
\leq \#\mathcal{Z}(\mathbf{w})
=\sum_{|\tau_4|\leq B_4}\sum_{|\tau_3|\leq B_3}\sum_{|\tau_2|\leq B_2}\#\mathcal{Z}(\mathbf{w},\tau_2,\tau_3,\tau_4)=\mathcal{I}^{\square}+\mathcal{I}^{\boxtimes},$$
where
$$ \mathcal{I}^{\square}:=\sum_{|\tau_4|\leq B_4}\sum_{|\tau_3|\leq B_3}\sum_{\tau_2\in\mathcal{T}^{\square}(\mathbf{w},\tau_3,\tau_4)}\#\mathcal{Z}(\mathbf{w},\tau_2,\tau_3,\tau_4),$$
$$\mathcal{I}^{\boxtimes}:= \sum_{|\tau_4|\leq B_4}\sum_{|\tau_3|\leq B_3}\sum_{\tau_2\in\mathcal{T}^{\boxtimes}(\mathbf{w},\tau_3,\tau_4)}\#\mathcal{Z}(\mathbf{w},\tau_2,\tau_3,\tau_4).$$

By Lemma \ref{lem:squarefiniteness}, we have using $\mathbf{w}\in \mathcal{W}_{\operatorname{good}}$
\begin{equation}
    \mathcal{I}^{\square}\ll  B_1B_3B_4\ll D^{2-\delta}p^{-14n}.
\end{equation}

For $\mathcal{I}^{\boxtimes}$ we apply a version of a theorem of Bombieri and Pila \cite{BombieriPila} (proven in Theorem~\ref{thm:BP} below) to the equation $\operatorname{disc}(Q)x_0^2=F_{\mathbf{w}}(z_1,\tau_2,\tau_3,\tau_4)$ in the variables $x_0,z_1$. 
If $\tau_2\in \mathcal{T}^{\boxtimes}(\mathbf{w},\tau_3,\tau_4)$ then $F_{\mathbf{w}}(z_1,\tau_2,\tau_3,\tau_4)$ is not square over $\mathbb{R}[z_1]$. We first consider the case that $F_{\mathbf{w}}(z_1,\tau_2,\tau_3,\tau_4)$ is square over $\mathbb{C}[z_1]$. This means that $F_{\mathbf{w}}(z_1,\tau_2,\tau_3,\tau_4)=-g(z_1)^2$ for some real polynomial $g\in\mathbb{R}[z_1]$, hence the equation $\operatorname{disc}(Q)x_0^2=F_{\mathbf{w}}(\mathbf{z})=-g(z_1)^2$ indeed implies $x_0=F_{\mathbf{w}}(z_1,\tau_2,\tau_3,\tau_4)=0$. It follows that there are at most $\deg_{z_1}F_{\mathbf{w}}=O(1)$ many solutions $(x_0,z_1)$. We thus have $\#\mathcal{Z}(\mathbf{w},\tau_2,\tau_3,\tau_4)\ll 1$ in this case.

We may now assume that $F_{\mathbf{w}}(z_1,\tau_2,\tau_3,\tau_4)$ is not square over $\mathbb{C}[z_1]$. Then $F(x_0,z_1)= \operatorname{disc}(Q)x_0^2-F_{\mathbf{w}}(z_1,\tau_2,\tau_3,\tau_4)$ is an absolutely irreducible polynomial of $x_0$-degree $2$.
By Theorem~\ref{thm:BP}, for any $\varepsilon>0$ we have
$$\#\mathcal{Z}(\mathbf{w},\tau_2,\tau_3,\tau_4)\ll_{\varepsilon} B_1^{\frac{1}{2}+\varepsilon}.$$
It follows that
\begin{equation}
        \mathcal{I}^{\boxtimes}\ll_\varepsilon B_1^{\frac{1}{2}}B_2B_3B_4 D^\varepsilon
        \ll (B_1B_2B_3B_4)^{\frac{7}{8}}D^\varepsilon
        \ll \big(D^{2}p^{-14n}\big)^{\frac{7}{8}} D^\varepsilon.
\end{equation}

Combining these estimates, we get
$$\#\mathcal{S}(\mathbf{w})\leq \mathcal{I}^{\square}+\mathcal{I}^{\boxtimes}
\ll_\varepsilon \{D^{2-\delta}p^{-14n}+(D^{2}p^{-14n})^{\frac{7}{8}}\}D^\varepsilon$$
for any $\mathbf{w}\in\mathcal{W}_{\operatorname{good}}$, hence
\begin{equation}\label{eq:W1contribution}
\begin{aligned}
    \sum_{\mathbf{w}\in\mathcal{W}_{\operatorname{good}}}\#\mathcal{S}(\mathbf{w})
    &\ll_\varepsilon \sum_{\mathbf{w}\in\mathcal{W}_{\operatorname{good}}}\{D^{2-\delta}p^{-14n}+(D^{2}p^{-14n})^{\frac{7}{8}}\}D^{\varepsilon}\\
    &\ll D^2p^{-12n}\big(D^{-\delta}+(D^{2}p^{-14n})^{-\frac{1}{8}}\big)D^{\varepsilon}.
\end{aligned}
\end{equation}

Recall $p^{4n}\asymp D^{\frac{1}{2}-\eta}$. Combining \eqref{eq:W2contribution} and \eqref{eq:W1contribution}, we get
\begin{align*}
    \#\mathcal{S}_0&\ll \sum_{\mathbf{w}\in\mathcal{W}_{\operatorname{good}}}\#\mathcal{S}(\mathbf{w})+\sum_{\mathbf{w}\in\mathcal{W}_{\operatorname{bad}}}\#\mathcal{S}(\mathbf{w})\\
    &\ll_{\varepsilon} D^2p^{-12n}\left\{D^{-\delta}+(D^{2}p^{-14n})^{-\frac{1}{8}}+p^{-2n}\left(p^{12n}D^{-\frac{3}{2}+3\delta}\right)^3\right\}D^\varepsilon\\ 
    &\asymp D^{\frac{1}{2}+3\eta}\left\{D^{-\delta}+D^{-\frac{1}{8}\big(\frac{1}{4}+\frac{7}{2}\eta\big)}+D^{-\big(\frac{1}{4}+\frac{17}{2}\eta-9\delta\big)}\right\}D^\varepsilon.
\end{align*}

We may choose $\delta=4\eta$ with $\delta_0'\leq\eta\leq\frac{1}{126}$ so that $\#\mathcal{S}_0\ll D^{\frac{1}{2}-\eta}$.
In combination with \eqref{eq:degcount} we conclude that
$$\#\mathcal{S}\leq \#\mathcal{S}_0+\#\mathcal{S}_{\operatorname{deg}}\ll_{\varepsilon} D^{\frac{1}{2}-\eta+\varepsilon}+p^{-4n}D^{\frac{1}{2}}
\asymp D^{\frac{1}{2}-\eta+\varepsilon}+D^{\eta}\ll_{\varepsilon} D^{\frac{1}{2}-\eta+\varepsilon}. $$
This concludes the proof of Proposition~\ref{prop:maincountingestimate} and, therefore, of our main theorems.
\newpage

\begin{appendix}

\section{The determinant method}\label{sec:det method}

Variants of the determinant method have been developed by Bombieri-Pila \cite{BombieriPila}, Heath-Brown \cite{HeathBrown-annals}, and Salberger \cite{Salberger-PLMS}, among others.
Here, we establish a home-made variant of the result of Bombieri and Pila, which is surely known to experts, though does not (to authors' knowledge) appear in the literature yet.

Given a polynomial $F \in \Z[x_1,\ldots,x_r]$ and $\mathrm{B}=(B_1,\ldots,B_r) \in (1,\infty)^r$ we set
\begin{align*}
X_{F}(\mathrm{B})
= \big\{ x \in \Z^r: |x_i| \leq B_i \text{ for all }i \text{ and } F(x) =0\big\}.
\end{align*}
The determinant method is one of the tools at hand providing upper bounds on the cardinality of $X_{F}(\mathrm{B})$.
We note that some of the literature concerning the determinant method considers cubes instead of general (axis-aligned) rectangular boxes; this would be insufficient for our purposes.

    \begin{theorem}[Bombieri-Pila]\label{thm:BP}
    Let $F(x_1,x_2)\in\mathbb{Z}[x_1,x_2]$ be a polynomial irreducible over $\overline{\Q}$. Suppose that $d=\deg_{x_2}(F)\geq 1$. 
    Let $\mathrm{B}=(B_1,B_2)$ and suppose that $B_2\leq B_1^{E}$ and $\deg F \leq E$ for some $E> 1$.
    Then for any $\varepsilon>0$ we have
    $$\#\big(X_{F}(\mathrm{B})\big)\ll_{E,d,\varepsilon} B_1^{\frac{1}{d}+\varepsilon}.$$
    \end{theorem}

    As stated, the above theorem deviates from e.g.~\cite[Thm.~5]{BombieriPila} as we do not consider points in a square (i.e.~$E$ could be bigger than $1$). 
    For completeness, we reproduce a proof using the $p$-adic determinant method of Heath-Brown \cite{HeathBrown-annals}.

    The following variant of a non-archimedean implicit function theorem was proven e.g.~in \cite[Lemma 5]{HeathBrown-annals}:
    
    \begin{lemma}\label{lem:IFT}
    Let $p>2$ be a prime.
    Let $F \in \Z_p[x_1,\ldots,x_r]$ and $u \in \Z_p^r$ such that $F(u) = 0$ and $p \nmid\partial_{r}F(u)$.
    Then for any $N \geq 1$ there exists a polynomial $f_N \in \Z_p[x_1,\ldots,x_{r-1}]$ so that for any $v\in \Z_p^r$ with $F(v)=0$ and $v \equiv u \mod p$ we have
    \begin{align*}
    v_r \equiv f_N(v_1,\ldots,v_{r-1}) \mod p^N.
    \end{align*}
    \end{lemma}
    
    \begin{proof}[Proof of Theorem~\ref{thm:BP}]
    We give a proof using the $p$-adic determinant method of Heath-Brown \cite{HeathBrown-annals}.
    Correspondingly, we will be brief in some parts and refer to \cite{HeathBrown-annals} for details.
    
    The argument for \cite[Thm.~4]{HeathBrown-annals} based on Siegel's lemma and Bezout's theorem implies that we may assume $\height(F) \ll (B_1B_2)^{E^3}$.
    Throughout the argument we may also assume that $B_1$ is sufficiently large in terms of $d,E$; else, the theorem is trivial.

        For a prime $p$, we let
        $$X_F^{p-\text{ns}}(\mathrm{B}):=\big\{(x_1,x_2)\in\mathbb{Z}^2: |x_i|\leq B_i,\; F(x_1,x_2)=0,\;p\nmid\partial_{2}F(x_1,x_2)\big\}.$$

        \begin{claim*}
         Given $\varepsilon>0$, there are distinct primes $2B_1^{\frac{1}{d}+\varepsilon}\leq p_1,\ldots,p_r < 4B_1^{\frac{1}{d}+\varepsilon}$ with $r\ll E^5$ satisfying
        $$\#\left(X_{F}(\mathrm{B})\setminus \bigcup_{i=1}^{r}X_F^{p_i-\text{ns}}(\mathrm{B})\right)\leq E^2.$$
        \end{claim*}

        To prove the claim, suppose a point $(x_1,x_2)\in X_{F}(\mathrm{B})\setminus \bigcup_{i=1}^{r}X_F^{p_i-\text{ns}}(\mathrm{B})$ exists where the primes $p_1,\ldots,p_r$ will be chosen. 
        Then we have
        $$ p_1\cdots p_r\ll |\partial_{2}F(x_1,x_2)|\ll \operatorname{ht}(F)B_1^{E}B_2^{d}\leq B_1^{10E^4},$$
        unless $\partial_{2}F(x_1,x_2)=0$. Thus, if we choose $p_1,\ldots,p_r$ the first $\asymp E^5$ primes bigger than $2B_1^{\frac{1}{d}+\varepsilon}$ (they are smaller than $4B_1^{\frac{1}{d}+\varepsilon}$, provided $B_1\gg E^6$), then $(x_1,x_2)$ is on the intersection of $F=0$ and $\partial_{2}F=0$. As $F$ is irreducible over $\overline{\mathbb{Q}}$, by Bezout's theorem there are at most $E^2$ points on the intersection.
        This proves the claim.
        
        We now estimate $\#X_F^{p_i-\text{ns}}(\mathrm{B})$ for $1\leq i\leq r$. Let us fix one of these primes $p=p_i$ and write
        \begin{align*}
        \overline{X}_{F\mod p}=\big\{\mathrm{w}\in \mathbb{F}_p^2:F(\mathrm{w})=0,\;\partial_{2}F(\mathrm{w})\neq 0\big\}
        \end{align*}
        for the solutions mod $p$.
        We split the points in $X_F^{p-\text{ns}}(\mathrm{B})$ into their congruence classes mod $p$ and, thus, set for $\mathrm{w} \in \overline{X}_{F\mod p}$
        $$\mathcal{F}_{p}(\mathrm{w}):=\{(v_1,v_2)\in X_F^{p-\text{ns}}(\mathrm{B}): (v_1,v_2)\equiv\mathrm{w} \mod p\}.$$
        
        We fix $\mathrm{w} \in \overline{X}_{F\mod p}$ and estimate the size of $\mathcal{F}_p(\mathrm{w})$ (that is, the size of the fiber above $\mathrm{w}$ under the reduction map $X_F(\mathrm{B}) \to \overline{X}_{F\mod p}$).
        Fix $(u_1,u_2) \in \mathcal{F}_p(\mathrm{w})$ (if $\mathcal{F}_p(\mathrm{w})$ is empty, we conclude) and a sufficiently large $N\geq 1$ to be determined later. By the non-archimedean implicit function theorem in Lemma \ref{lem:IFT}, there exists a polynomial $f_{\mathrm{w}}(x_1)=f_{\mathrm{w},N}(x_1)\in\mathbb{Z}_p[x_1]$ such that $$v_2\equiv f_{\mathrm{w}}(v_1) \mod p^N$$ for any $\mathrm{v}=(v_1,v_2)\in \mathcal{F}_p(\mathrm{w})$.

    For $R\geq 1$ let $\mathcal{Q}(R)$ be a family of monomials
    $$\mathcal{Q}(R):=\{x_1^{m_1}x_2^{m_2}: 0\leq m_1\leq R-1,\;0\leq m_2\leq d-1\}$$
    and $s:=|\mathcal{Q}(R)|=dR$. 
    We pick an enumeration $F_1,\ldots,F_s$ of the monomials in $\mathcal{Q}(R)$.
    Let $\mathbf{\xi}_j=(x_{j,1},x_{j,2})$ for $j=1,\ldots,s$ be integral points on $\mathcal{F}_p(\mathrm{w})$ and set 
    \begin{align*}
    \Delta=\det(F_i(\xi_j))_{1\leq i,j\leq s}.
    \end{align*}
    
    A trivial upper bound on $\Delta$ is given by
    \begin{equation}\label{eq:determinantupperboundBP}
    \begin{aligned}
        |\Delta|&\leq (dR)! \prod_{x_1^{m_1}x_2^{m_2}\in \mathcal{Q}(R)}B_1^{m_1}B_2^{m_2}\\&= (dR)!B_1^{\frac{dR(R-1)}{2}}B_2^{\frac{d(d-1)R}{2}}\leq (dR)!B_1^{\frac{d}{2}R^2+\frac{1}{2}d^2E R}.
    \end{aligned}
    \end{equation}

    We choose $R,N\geq 1$ so that $\frac{R+dE}{dR-1}<\frac{1}{d}+\varepsilon$ and $N\geq (dR)^2$. Let $m_i \in \Z_{\geq 0}^2$ with $F_i(x_1,x_2)=x_1^{m_{i,1}}x_2^{m_{i,2}}$ for $1\leq i\leq s$ and write
    \begin{equation}
        \begin{aligned}
            \Delta&\equiv \det\big(x_{j,1}^{m_{i,1}}x_{j,2}^{m_{i,2}}\big)_{1\leq i,j\leq s}\,\mod p^N\\
            &\equiv\det\big(x_{j,1}^{m_{i,1}}f_{\mathrm{w}}(x_{j,1})^{m_{i,2}}\big)_{1\leq i,j\leq s}\,\mod p^N.
        \end{aligned}
    \end{equation} Let $y_{j,1}=x_{j,1}-u_1$ for $1\leq j\leq s$. Then we have
    $$\Delta\equiv\det\big((y_{j,1}+u_1)^{m_{i,1}}f_{\mathrm{w}}(y_{j,1}+u_1)^{m_{i,2}}\big)_{1\leq i,j\leq s}\; (\textrm{mod }p^N).$$
    By expanding $(y_{j,1}+u_1)^{m_{i,1}}f_{\mathrm{w}}(y_{j,1}+u_1)^{m_{i,2}}$'s and eliminating low-order in $y_{j,1}$ by column operations, we may arrange the columns so that $y_{j,1}^{i-1}$ divides the $i$-th column of the matrix in the determinant for $1\leq i\leq s$.
    Note that $p|y_{j,1}$ since $(x_{j,1},x_{j,2})\equiv (u_1,u_2) \; (\textrm{mod }p)$. Using this, we obtain $p^M|\Delta$ where
    $$M=0+1+\cdots+(s-1)=\tfrac{1}{2}dR(dR-1).$$
    We remark that the above divisibility relation has been verified at various points in the literature (and in fact in much more complicated settings).
    See for example \cite[Lemma 6]{HeathBrown-annals}, \cite[Proof of Lemma 3]{Broberg}, or \cite[Lemma 2.4]{Salberger-crelle}.
    
    It follows that
    \begin{equation}\label{eq:determinantlowerboundBP}
        \log|\Delta|\geq M\log p=(\log p)\left(\frac{d^2}{2}R^2-\frac{d}{2}R\right),
    \end{equation}
    unless $\Delta=0$. Combining \eqref{eq:determinantupperboundBP} and \eqref{eq:determinantlowerboundBP}, we have $\Delta=0$ if $$p>((dR)!)^{\frac{2}{dR(dR-1)}}B_1^{\frac{R+dE}{dR-1}}.$$
    Recall that the prime $p\geq 2B_1^{\frac{1}{d}+\varepsilon}$ satisfies this inequality as we chose $R$ sufficiently large so that $\frac{R+dE}{dR-1}<\frac{1}{d}+\varepsilon$.

    To summarize, we have shown that the rank of the matrix $(F_i(v))_{i \leq s, v \in \mathcal{F}_p(\mathsf{w})}$ is at most $s-1$ for $s$ as above.
    
    It follows that for each $\mathrm{w}\in \overline{X}_{F\mod p}$ there exists a polynomial $\widetilde{F}_{\mathrm{w}}(x_1,x_2)\in \mathbb{Z}[x_1,x_2]$ such that 
    \begin{align*}
    \deg (\widetilde{F}_{\mathrm{w}})=O_{d,E,\varepsilon}(1),\
    \deg_{x_2} (\widetilde{F}_{\mathrm{w}})< d,\ \text{and }
    \mathcal{F}_p(\mathrm{w})\subset X_{\widetilde{F}_{\mathrm{w}}}
    \end{align*}
    where $X_{\widetilde{F}_{\mathrm{w}}}$ is the zero locus of $\widetilde{F}_{\mathrm{w}}$.
    Since $\deg_{x_2} (\widetilde{F}_{\mathrm{w}})< d$, we have $\widetilde{F}_{\mathrm{w}}\nmid F$. Therefore, $X_{F}\cap X_{\widetilde{F}_{\mathrm{w}}}$ is a finite set of points and moreover 
    $$\#\mathcal{F}_p(\mathrm{w})\leq\#\big(X_{F}\cap X_{\widetilde{F}_{\mathrm{w}}}\big)\leq \deg(F)\deg (\widetilde{F}_{\mathrm{w}})=O_{d,E,\varepsilon}(1)$$ by B\'{e}zout's theorem. Hence,
    \begin{align*}
    \#\big(X_F^{p-\text{ns}}(\mathrm{B})\big)
    \leq \sum_{\mathrm{w}\in \overline{X}_{F\mod p}}\#\mathcal{F}_p(\mathrm{w})
    \ll_{d,E,\varepsilon}\#\overline{X}_{F\mod p}\ll p\ll B_1^{\frac{1}{d}+\varepsilon}
    \end{align*}
    which proves the theorem.
    \end{proof}

\section{An application of the Siegel mass formula}\label{sec:Siegel mass}

In this appendix we will prove upper bounds on the number of representations of a quadratic form by another when the codimension is either $0$ or $1$.
Generally, one expect representations in such small codimension to be very scarce.

Assume that $q,Q$ are two integral quadratic forms.
We denote by $n$ the number of variables of $q$ and by $M_q$ resp.~$M_Q$ a matrix representation of $q$ resp.~$Q$.
We suppose that both $q$ and $Q$ are non-degenerate i.e.~for each $v \in \Z^n$ there exists $w \in \Z^n$ with $\langle v,w\rangle_q \neq 0$ and similarly for $Q$ (equivalently, $\det(M_q) \neq 0 \neq \det(M_Q)$).

\begin{definition}\label{def:elementary divisors}
Given a prime $p$, the \emph{elementary divisors} of $q$ at $p$ are the unique integers $k_1(p,q),\ldots,k_n(p,q)$ such that for all $j \leq n$
\begin{align*}
k_1(p,q)+\ldots+k_j(p,q) = \min_{\Lambda_j} \ord_p(2^{j}\disc_q(\Lambda_j))
\end{align*}
where $\Lambda_j$ runs over all primitive sublattices of $\Z_p^n$ (of non-zero discriminant) of rank $j$.
\end{definition}

Recall \cite[Ch.~8]{cassels} that for any odd prime $p$ the form $q$ (or $Q$) is equivalent to a diagonal form
\begin{equation}\label{eq:diagonalize q}
a_1 p^{k_1} x_1^2 + \ldots + a_n p^{k_n}x_n^2
\end{equation}
where $a_i=a_i(p) \in \Z_p^\times$ and $0 \leq k_1 \leq \ldots \leq k_n$. 
It is easy to see that $k_j=k_j(p,q)$ in this case.
This alternative viewpoint is not applicable for $p=2$ where we caution readers that the sequences $k_1(2,q),\ldots,k_n(2,q)$ may not be increasing (though the defect is is very controlled, e.g.~$k_1(2,q) \leq k_2(2,q)+1$)

Define now for any $1\leq j \leq n$
\begin{align}\label{eq:m_j}
\mathsf{m}_j(q) = \prod_{p \text{ prime}}p^{k_j(p,q)}.
\end{align}
In particular, $\mathsf{m}_1(q) = \gcd(q)$.
All of the above definitions apply analogously to $Q$.

In this appendix, we establish the following.

\begin{theorem}[Equal rank $4$]\label{thm:app-equalrank}
Let $q,Q$ be two positive definite integral quadratic forms over $\Z$ of equal rank $4$.
Then the number of representations of $q$ by $Q$ is 
\begin{align*}
\ll_{Q,\varepsilon} \mathsf{m}_1(q) \mathsf{m}_2'(q) |\disc(q)|^\varepsilon
\end{align*}
where $\mathsf{m}_2'(q)$ is the largest integer with $\mathsf{m}_2'(q)^2 \mid \mathsf{m}_2(q)$.
\end{theorem}

In the codimension $1$ case we have the following.

\begin{cor}[Rank $3$ resp.~$4$]\label{cor:app-codim1}
Let $q,Q$ be two positive definite integral quadratic forms over $\Z$ of rank $3$ resp.~$4$.
Then the number of representations of $q$ by $Q$ is
\begin{align*}
\ll_{Q,\varepsilon} \mathsf{m}_1(q) \mathsf{m}_2'(q) |\disc(q)|^\varepsilon
\end{align*}
where $\mathsf{m}_2'(q)$ is the largest integer with $\mathsf{m}_2'(q)^2 \mid \mathsf{m}_2(q)$.
\end{cor}

The above results are by no means novel.
For instance, a different version of the above estimate under slightly stronger assumptions appears in work of Bourgain and Demeter~\cite{BourgainDemeter-Fourierrestrictionsphere,BourgainDemeter-Siegelmass}.
As in their work, we establish the above as a consequence of Siegel's mass formula \cite{Siegel-I} (see also e.g.~\cite[Ch.~6.8]{Kitaoka-book}).

In the context of proving Linnik-type bounds on self-correlations of toral measures, the analogue of Corollary~\ref{cor:app-codim1} for representations of binary by ternary forms was instrumental in \cite{ELMV-Ens,Linnikthmexpander,W-Linnik}.
Note that \cite[App.~A]{ELMV-Ens} contains a proof not reliant on the Siegel mass formula and following instead a combinatorial argument on the Bruhat-Tits tree of $\PGL_2$.

\subsection{Local estimates at $p\neq 2$}
As before, we write $n$ for the number of variables of $q$ and $Q$.
In the setting of Theorem~\ref{thm:app-equalrank} we have $n=4$, though the following discussion will also be required for $n<4$.

The Siegel mass formula expresses the average of representation numbers over the genus of $Q$ as a product of the local densities
\begin{align}\label{def:app-localdensities}
\alpha_p(q,Q) 
= \lim_{r\to\infty} (p^{-r})^{\frac{n(n-1)}{2}} 
\#\{X \in \Mat_n(\Z/p^r\Z): X^t M_Q X = M_q \mod p^r\}
\end{align}
where $p$ varies over all the primes and where the above sequence in $r$ is eventually constant.
For the current subsection, we fix the prime $p \neq 2$ and think of $q$ and $Q$ as quadratic forms over $\Z_p$.

The goal of the following will be to prove an inductive upper bound on the densities $\alpha_p(q,Q)$.
We begin with the following elementary counting estimate for points on level sets of $Q$ modulo a prime power $p^r$.

\begin{lemma}\label{lem:app-counting vectors}
Denote by $(l_1,\ldots,l_n)$ the elementary divisors of $Q$.
Let $k>0$ and $a\in \Z_p^\times$.
Fix a tuple $\underline{\nu}=(\nu_1,\ldots,\nu_n)\in (\Z_{\geq 0} \cup \{\infty\})^n$ and define for $r>0$
\begin{align*}
\mathcal{A}_r(a,k;\underline{\nu}) = \big\{x \in (\Z/p^r\Z)^n: Q(x) \equiv ap^{k} \mod p^r \text{ and } \ord_p(x_i) = \nu_i \text{ for all }i\big\}.
\end{align*}
For all $r$ large enough we have $\mathcal{A}_r(a,k;\underline{\nu}) = \emptyset$ if $k < \min\{l_i+2\nu_i\}$ and otherwise
\begin{align*}
\#\mathcal{A}_r(a,k;\underline{\nu}) \leq 2p^{r(n-1)} p^{\nu_{\mathrm{tot}}}
\end{align*}
where $\nu_{\mathrm{tot}} = \min_{i}\{l_i+2\nu_i\} - \sum_{i:\nu_i < \infty} \nu_i$.
\end{lemma}

\begin{proof}
If $\nu_i =\infty$, for some $i$ we may discard that coordinate and decrease the dimension by $1$. So assume $\nu_i < \infty$ for all $i$.

Estimating $\#\mathcal{A}_r(a,k;\underline{\nu})$ is equivalent to estimating the cardinality of the set of tuples $(x_1,\ldots,x_n)$ where $x_i \in \Z/p^{r-\nu_i}\Z$ with $\sum_{i}b_i p^{l_i+2\nu_i}x_i^2 \equiv a p^{k} \mod p^{r}$ and with $p \nmid x_i$ for all $i$.
By reordering the terms we may assume $l_1+2\nu_1 = \min_i\{l_i+2\nu_i\}=:l$.
If $k<l$, the conclusion is clear, so assume otherwise.

Fix the last $n-1$ coordinates arbitrarily; that amounts to $p^{r(n-1)-\sum_{i>1}\nu_i}$ choices. 
The first coordinate $x_1$ therefore satisfies an equation of the form $x_1^2 = c \mod p^{r-l}$ for some $c \in \Z_p$. 
This determines $x_1$ modulo $p^{r-l}$ up to a sign and there are $p^{\ell-\nu_1}$ choices for $x_1$.
Our total count is therefore at most $2p^{r(n-1)-\sum_{i}\nu_i+l}$ as claimed.
\end{proof}

\begin{lemma}\label{lem:app-orthcomplement}
Denote by $(l_1,\ldots,l_n)$ the elementary divisors of $Q$.
Let $k>0$ and $a\in \Z_p^\times$. Suppose that $x\in \Z_p^n$ satisfies $Q(x) = ap^k$ and $\ord_p(x_i) = \nu_i$ for all $i$.
Then
\begin{align*}
\ord_p(\disc_Q(\{y \in \Z_p^n: y \perp x\}))
= (k - 2\min_i\{l_i+\nu_i\}) + \sum_i l_i.
\end{align*}
\end{lemma}

\begin{proof}
We reorder the coordinates so that $l_1+\nu_1 \leq l_2+\nu_2 \leq \ldots$.
In particular, if $\nu_i = \infty$ then $\infty = \nu_{i+1}=\nu_{i+2}=\ldots$ and
\begin{align*}
\{y \in \Z_p^n: y \perp x\} \simeq  \{y' \in \Z_p^{i-1}: y' \perp x\} \oplus \Z_p^{n-i+1}.
\end{align*}
We may thus assume $\nu_i <\infty$ for every $i$.
Equivalently, the orthogonal complement of $x$ does not contain any of the coordinate axes.
We also write $\tilde{x}_i = p^{-\nu_i}x_i$.

We first construct a suitable basis of the orthogonal complement (strongly relying on the fact that the construction is a local one).
Let $\Lambda=\{y \in \Z_p^n: y \perp x\}$.
As $\Lambda$ has co-rank $1$, any $m$-dimensional primitive lattice not contained in $\Lambda$ intersects $\Lambda$ in an $(m-1)$-dimensional primitive lattice.
In view of our assumptions we thus have a complete flag of primitive sublattices of $\Lambda$
\begin{align*}
\{0\} < \Lambda \cap \langle e_1,e_2\rangle < \Lambda \cap \langle e_1,e_2,e_3\rangle < \ldots < \Lambda \cap \langle e_1,\ldots,e_{n-1}\rangle < \Lambda.
\end{align*}
We may therefore pick a $\Z_p$-basis $w_1,\ldots, w_{n-1}$ of $\Lambda$ with $w_i \in \Lambda \cap \langle e_1,\ldots,e_{i+1}\rangle$ for every $i$.

We now manipulate this basis further and do so in two steps.
By assumption we have $w_1 \perp x$ or equivalently
\begin{align*}
a_1 p^{l_1}x_1 w_{11}+ a_2 p^{l_2}x_2 w_{12} =0
\end{align*}
where $a_1,a_2,\ldots$ are units as in \eqref{eq:diagonalize q} for $Q$.
Thus, $w_1$ must be a multiple of $(-a_2 p^{l_2+\nu_2-l_1-\nu_1}\tilde{x}_2,a_1\tilde{x}_1,0,\ldots,0)$.
Since $w_1$ is primitive, we may assume $w_1 = (-a_2 p^{l_2+\nu_2-l_1-\nu_1}\tilde{x}_2,a_1\tilde{x}_1,0, \ldots)$ and, in particular, $p \nmid w_{12}$.
By replacing $w_2$ with an appropriate translation by $w_1$ we may thus assume $w_{22}=0$. By the same argument as for $w_1$ we may suppose that $w_2 = -a_3 p^{l_3+\nu_3-l_1-\nu_1}\tilde{x}_3 e_1 + a_1\tilde{x}_1 e_3$.
Iterating this procedure we have shown that
\begin{align*}
w_i = -a_{i+1}p^{l_{i+1}+\nu_{i+1}-l_1-\nu_1}\tilde{x}_{i+1} e_1 + a_1\tilde{x}_1 e_{i+1}
\end{align*}
is a basis of $\Lambda$.
Note that $w_{11}\mid w_{21}\mid \cdots$ in view of our ordering of the coordinates.
Thus, replacing $w_2$ by a suitable translation with $w_1$ and so forth we may in fact assume that
\begin{align*}
w_i = -a_{i+1}p^{l_{i+1}+\nu_{i+1}-l_i-\nu_i}\tilde{x}_{i+1} e_i + a_i\tilde{x}_i e_{i+1}
\end{align*}
where we used the same argument as before. We remark that while it is clear that these $w_i$'s belong to $\Lambda$, it is only now understood that they form a basis of $\Lambda$.

We turn to computing the discriminant of $\Lambda$ using the above basis. Let $\Lambda'$ be the sublattice of $\Lambda$ spanned by $-a_{i+1}p^{l_{i+1}}x_{i+1}e_i + a_ip^{l_i}x_i e_{i+1}$. Then we have
\begin{align}\label{eq:orthcompl-sublattice}
\disc(\Lambda)
= p^{-2 \sum_{i < n}(l_i+\nu_i)} \disc(\Lambda').
\end{align}
By orthogonality of $e_1,\ldots,e_n$
\begin{align*}
\disc(\Lambda')
&= \prod_{i< n} \big(a_{i+1}p^{l_{i+1}} (a_ip^{l_i}x_i)^2\big)\\
&\quad +\sum_{1 \leq j < n} \big(\prod_{1 \leq i \leq j} a_ip^{l_i}(a_{i+1}p^{l_{i+1}}x_{i+1})^2\big)
\big(\prod_{j < i < n}a_{i+1}p^{l_{i+1}} (a_ip^{l_i}x_i)^2\big).
\end{align*}
Note that in each of the above summands is divisible by $\prod_{i<n}(a_{i}a_{i+1}p^{l_i+l_{i+1}}) =:A$ and
\begin{align*}
\disc(\Lambda') = A \Big(\prod_{i< n} \big( a_ip^{l_i}x_i^2\big)
 +\sum_{1 \leq j < n} \big(\prod_{1 \leq i \leq j} a_{i+1}p^{l_{i+1}}x_{i+1}^2\big)
\big(\prod_{j < i < n} a_ip^{l_i}x_i^2\big)\Big).
\end{align*}
Furthermore, each term is divisible by $\prod_{1<i< n} a_ip^{l_i}x_i^2$ and therefore
\begin{align*}
\disc(\Lambda') &= A \big(\prod_{1<i< n}  a_ip^{l_i}x_i^2\big)
\big(a_1p^{l_1}x_1^2 + \ldots + a_n p^{l_n}x_n^2\big)\\
&= \big(\prod_{i<n}a_{i}a_{i+1}p^{l_i+l_{i+1}}\big) \big(\prod_{1<i< n}  a_ip^{l_i}x_i^2\big) ap^k.
\end{align*}
Applying \eqref{eq:orthcompl-sublattice} we obtain
\begin{align*}
\disc(\Lambda) &= \big(\prod_{i<n}a_{i}a_{i+1}p^{l_i+l_{i+1}}\big) \big(\prod_{1<i< n}  a_i\tilde{x}_i^2\big) ap^{k-2(l_1+\nu_1)-\sum_{1<i<n}l_i}\\
&\in  \disc(Q) p^{k-2(l_1+\nu_1)}\Z_p^\times
\end{align*}
which proves the lemma.
\end{proof}

With these preparations at hand, we turn to stating the claimed inductive upper bound on local densities.

\begin{proposition}\label{prop:app-inductiveclaim}
Let $(l_1,\ldots,l_n)$ resp.~$(k_1,\ldots,k_n)$ be the elementary divisors of $Q$ resp.~$q$ at $p$ and write $q$ as in \eqref{eq:diagonalize q}.
Then
\begin{align*}
\alpha_p(q,Q)
\ll \sum_{\underline{\nu}} p^{\min_{i}\{l_i+2\nu_i\} - \sum_{i} \nu_i + (n-1)\min_{i}\{l_i+\nu_i\}} 
\max_{Q_{\underline{\nu}}} \alpha_p\big(\sum_{i\geq 2}a_ip^{k_i}x_i^2,Q_{\underline{\nu}}\big)
\end{align*}
where $\underline{\nu}$ runs over all $\underline{\nu} \in (\Z_{\geq 0} \cup \{\infty\})^n$ with $\min_i\{l_i+2\nu_i\} \leq k_1$ and
$Q_{\underline{\nu}}$ runs over all forms over $\Z_p$ in $n-1$ variables so that $(k_1 - 2\min_i\{l_i+\nu_i\}) + \sum_i l_i$ is the valuation of the discriminant.
Here, the implicit constant is independent of $p$.
\end{proposition}

\begin{proof}
We may split the count involved in computing the density $\alpha_p(q,Q)$ (cf.~ \eqref{def:app-localdensities}) according to the set of valuations we see in the `first' vector.
Thus, let $\underline{\nu}=(\nu_1,\ldots,\nu_n) \in (\Z_{\geq 0} \cup \{\infty\})^n$ with $\min_i\{l_i+2\nu_i\} \leq k_1$. Let $i_0$ be the least integer with $l_{i_0}+\nu_{i_0}=\min_i\{l_i+\nu_i\}$.

We write the coefficients of $q$ as in \eqref{eq:diagonalize q} and write $Q= \sum_i b_ip^{l_i}x_i^2$.
By Lemma \ref{lem:app-counting vectors} we have 
\begin{align*}
\#\mathcal{A}_r(a_1,k_1;\underline{\nu}) \leq 2p^{r(n-1)} p^{\nu_{\mathrm{tot}}}.
\end{align*}
Now fix a vector $v_1 \in \mathcal{A}_r(a_1,k_1;\underline{\nu})$. 
When $r>k_1$ may lift $v_1$ to a vector in $\Z_p^n$ with $Q(v_1) = a_1 p^{k_1}$ and $\ord_p(v_{1i}) = \nu_i$ for all $i$ where we wrote again $v_1$ for the lift by abuse of notation.
Lemma~\ref{lem:app-orthcomplement} yields all the information we need on the orthogonal complement $\Lambda = \{w\in \Z_p^n: w \perp v_1\}$. We fix a basis $w_1,\ldots,w_{n-1}$ of it.
Note that $w_1,\ldots,w_{n-1},e_{i_0}$ is a basis of $\Z_p^n$ by the proof of Lemma~\ref{lem:app-orthcomplement}.
Any vector can thus be expressed in the coordinates of this basis; to distinguish them from the standard coordinates we write $x = \hat{x}_1 w_1 + \ldots + \hat{x}_{n-1}w_{n-1} + \hat{x}_n e_{i_0}$ for $x \in \Z_p^n$.
As in the proof of Lemma~\ref{lem:app-orthcomplement}, we ensure that $w_2,\ldots,w_{n-1}\perp e_{i_0}$.

We need to estimate the number of $(v_2,\ldots,v_n)$ in $(\Z/p^r\Z)^n$ with
\begin{align}\label{eq:app-afterfixingv1}
\langle v_i,v_j\rangle \equiv \delta_{ij}a_jp^{k_j} \mod p^{r} \quad \text{ for all }1 \leq i \leq j\leq n,\, (i,j) \neq (1,1).
\end{align}
Expanding this condition for $i=1$ and using the above new coordinates we obtain
$b_{i_0}p^{l_{i_0}} \hat{v}_{jn} v_{i_0}
\equiv 0 \mod p^r$ for any $j>1$ or equivalently
\begin{align*}
\hat{v}_{jn} 
\equiv 0 \mod p^{r-l_{i_0}-\nu_{i_0}}.
\end{align*}
Therefore, $\hat{v}_{jn}$ can take at most $p^{l_{i_0}+\nu_{i_0}}$ many values for each $j$. We fix one such value for every $j$ at the cost of a factor of $p^{(n-1)\min_i\{l_i+\nu_i\}}$ in the final count.
With this choice, \eqref{eq:app-afterfixingv1} for all $i,j\geq 2$ is now an inhomogeneous equation in the remaining coordinates $\hat{v}_{j1},\ldots,\hat{v}_{j(n-1)}$.
We perform an affine coordinate change to homogenize the resulting equation and use the following claim to that end.

\begin{claim*}
For any $\alpha \in p^{r-l_{i_0}-\nu_{i_0}}\Z_p$ there exists $w \in \Lambda$ so that $(w+\alpha e_{i_0},v) \equiv 0 \mod p^r$ for all $v \in \Z_p^n$ with $\hat{v}_n \in p^{r-l_{i_0}-\nu_{i_0}}\Z_p$.
\end{claim*}

To prove the claim, we may replace, for convenience, our basis of $\Lambda$ with an orthogonal basis $w_1',\ldots,w_{n-1}'$ (cf.~\eqref{eq:diagonalize q}).
The vector $w = \sum_i \beta_i w_i'$ we wish to find needs to satisfy
\begin{align*}
0 \equiv (w+\alpha e_{i_0},w_i') \equiv \beta_i (w_i',w_i') + \alpha (e_{i_0},w_i') \mod p^r
\end{align*}
or equivalently $\beta_i \equiv -\frac{\alpha(e_{i_0},w_i')}{(w_i',w_i')} \mod p^{r-\star}$. 
Here, recall that $r$ is assumed sufficiently large so that the quotient is well-defined.
We fix such choices $\beta_{1},\ldots, \beta_{n-1}$ so $w$ is determined with $w \equiv 0 \mod p^{r-\star}$.
Therefore, for any $v\in \Z_p^n$ with $\hat{v}_n\in p^{r-l_{i_0}-\nu_{i_0}}\Z_p$
\begin{align*}
(w+\alpha e_{i_0},v) \equiv \hat{v}_n (w+\alpha e_{i_0},e_{i_0}) \equiv 0 \mod p^r
\end{align*}
using $w \equiv 0 \mod p^{r-\star}$ and $\alpha\equiv 0 \mod p^{r-\star}$.
This proves the claim.

We apply the above claim for $\alpha = \hat{v}_{jn}$ and find $c_j \in \Lambda$ with the given property. Set $u_{j-1} = \hat{v}_{j1}w_1+\ldots+\hat{v}_{j(n-1)}w_{n-1}-c_j$.
With this, note that
\begin{align*}
(v_i,v_j) 
&\equiv (\hat{v}_{i1}w_1+\ldots+\hat{v}_{i(n-1)}w_{n-1},v_j) + (\hat{v}_{in}e_{i_0},v_j) \\
&\equiv (\hat{v}_{i1}w_1+\ldots+\hat{v}_{i(n-1)}w_{n-1},v_j) - (c_i,v_j)\\
&\equiv (u_{i-1},v_j) \equiv (u_{i-1},u_{j-1}) \mod p^r.
\end{align*}
Thus, after the coordinate change by translation with the $c_i$'s we are left with counting the number of solutions $(u_1,\ldots,u_{n-1}) \in \Lambda/p^r\Lambda$ with $(u_i,u_j) \equiv \delta_{ij}a_jp^{k_{j+1}} \mod p^r$.
This constitutes the claimed reduction by Lemma~\ref{lem:app-orthcomplement}.
We have accumulated a factor $p^{\min_i\{l_i+2\nu_i\} -\sum_i \nu_i +(n-1) \min_{i}\{l_i+\nu_i\}}$; the proposition follows.
\end{proof}

\begin{remark}\label{rem:app-twoindices}
The proofs of Proposition~\ref{prop:app-inductiveclaim} and Lemma~\ref{lem:app-counting vectors} show that the maximum in Proposition~\ref{prop:app-inductiveclaim} can in fact be taken over $\underline{\nu}$ so that the minimum of $l_i+2\nu_i$ is attained at least at two indices $i$.
\end{remark}

We apply Proposition~\ref{prop:app-inductiveclaim} in low dimensional cases to obtain the desired estimate for local densities.

\begin{cor}[Estimates for local densities]\label{cor:app-estimatelocdensity}
Write $(k_1,\ldots)$ and $(l_1,\ldots)$  respectively for the elementary divisors of $q$ and $Q$ at $p$. 
Then
\begin{align*}
\alpha_p(q,Q)
\ll \prod_{i < n}(k_i+1)
\begin{cases}
p^{\lfloor\frac{k_1+l_1}{2}\rfloor} &\text{if } n=1,\\
 p^{\lfloor\frac{k_1+k_2}{2}\rfloor + l_1 +l_2} &\text{if } n=2,\\
p^{\lfloor \frac{1}{2}k_1\rfloor + \frac{1}{2}(k_1+k_2+k_3) + \frac{3}{2}(l_1 +l_2+l_3)}&\text{if } n=3,\\
p^{k_1+\lfloor \frac{1}{2}k_2\rfloor + \frac{1}{2}(k_1+k_2+k_3+k_4) + 2(l_1 +l_2+l_3+l_4)}&\text{if } n=4.
\end{cases}
\end{align*}
\end{cor}

\begin{proof}
The corollary is proven by `induction' on $n$.
Throughout, we have $q = \sum_i a_i p^{k_i}x_i^2$ and $Q = \sum_i b_i p^{l_i}x_i^2$.

\textsc{The case $n=1$:}
Write $q(x) = a_1p^{k_1}x^2$ and $Q(x) = b_1 p^{l_1}x^2$.
If $l_1>k_1$ or $k_1-l_1$ is odd, we have $\alpha_p(q,Q) =0$ and we conclude. So assume otherwise and set $k= \frac{1}{2}(k_1-l_1)$.
For $r$ large enough, compute
\begin{align*}
\alpha_p(q,Q) &=\#\{x \in \Z/p^r\Z: b_1 p^{l_1} x^2 = a_1p^{k_1} \mod p^r\}\\
&=p^{l_1}\#\{x \in \Z/p^{r-l_1}\Z: b_1 x^2 = a_1p^{k_1-l_1} \mod p^{r-l_1}\}\\
&= p^{l_1}\#\{x \in \Z/p^{r-l_1-k}\Z: b_1 x^2 = a_1 \mod p^{r-l_1-2k}\}\\
&= p^{l_1+k}\#\{x \in \Z/p^{r-l_1-2k}\Z: b_1 x^2 = a_1 \mod p^{r-l_1-2k}\} = 2p^{\lfloor\frac{1}{2}(l_1+k_1)\rfloor}.
\end{align*}

\textsc{The case $n=2$:}
We apply Proposition~\ref{prop:app-inductiveclaim}.
By the previous case $n=1$,
\begin{align*}
\max_{Q_{\underline{\nu}}} \alpha_p\big(a_2p^{k_2}x^2,Q_{\underline{\nu}}\big) \leq 2 p^{\lfloor \frac{k_1+k_2+l_1+l_2}{2}\rfloor - \min\{l_i+\nu_i\}}
\end{align*}
and so
\begin{align*}
\alpha_p(q,Q)
\ll \sum_{\underline{\nu}} p^{\min_{i}\{l_i+2\nu_i\} - \nu_1 - \nu_2+\lfloor \frac{k_1+k_2+l_1+l_2}{2}\rfloor}.
\end{align*}
Now notice also that in view of Remark~\ref{rem:app-twoindices}
\begin{align*}
\min_{i}\{l_i+2\nu_i\} - \nu_1 - \nu_2
&= \lfloor\tfrac{1}{2}(l_1+2\nu_1+l_2+2\nu_2)\rfloor - \nu_1-\nu_2= \lfloor \tfrac{1}{2}(l_1+l_2)\rfloor
\end{align*}
and the above summation over $\underline{\nu}$ is restricted to $l_1+2\nu_1,l_2+2\nu_2 \leq k_1$. In particular, $\nu_1,\nu_2 \leq \tfrac{1}{2}k_1$ and so as claimed
\begin{align*}
\alpha_p(q,Q) 
\ll (k_1+1) p^{\lfloor \frac{1}{2}(l_1+l_2)\rfloor+\lfloor \frac{k_1+k_2+l_1+l_2}{2}\rfloor}
\leq (k_1+1) p^{\lfloor \frac{k_1+k_2}{2}\rfloor+l_1+l_2}.
\end{align*}

\textsc{The case $n=3$:} 
By the case $n=2$
\begin{align*}
\max_{Q_{\underline{\nu}}} \alpha_p\big(a_2p^{k_2}x_1^2+a_3p^{k_3}x_2^2,Q_{\underline{\nu}}\big)
\ll (k_2+1) p^{\lfloor \frac{k_2+k_3}{2}\rfloor + l_1+l_2+l_3 + k_1-2\min_i\{l_i+\nu_i\}}
\end{align*}
and so by Proposition~\ref{prop:app-inductiveclaim}
\begin{align*}
\alpha_p(q,Q)\ll
(k_2+1) p^{\lfloor \frac{k_2+k_3}{2}\rfloor + l_1+l_2+l_3 + k_1}\sum_{\underline{\nu}} p^{\min_i\{l_i+2\nu_i\}-\nu_1-\nu_2-\nu_3}.
\end{align*}
In view of Remark~\ref{rem:app-twoindices} denote by $i_1,i_2 \in \{1,2,3\}$ two distinct indices with $l_{i_1}+2\nu_{i_1}=l_{i_2}+2\nu_{i_2}=\min_i\{l_{i}+2\nu_{i}\}\leq k_1$.
Then $\min_i\{l_i+2\nu_i\} -\nu_{i_1}-\nu_{i_2} = \lfloor \frac{1}{2}(l_{i_1}+l_{i_2})\rfloor$
and
\begin{align*}
\sum_{\underline{\nu}} p^{\min_i\{l_i+2\nu_i\}-\nu_1-\nu_2-\nu_3}
&\ll (k_1+1) p^{\lfloor \frac{1}{2}(l_{2}+l_{3})\rfloor} \sum_{\nu_i, i\neq i_1,i_2} p^{-\nu_i}\\
&\ll (k_1+1) p^{\lfloor \frac{1}{2}(l_{2}+l_{3})\rfloor}.
\end{align*}
To summarize,
\begin{align*}
\alpha_p(q,Q)
&\ll (k_1+1)(k_2+1) p^{k_1+\lfloor \frac{k_2+k_3}{2}\rfloor + l_1+l_2+l_3 + \lfloor \frac{1}{2}(l_{2}+l_{3})\rfloor}\\
&\ll (k_1+1)(k_2+1) p^{k_1+\lfloor \frac{k_2+k_3}{2}\rfloor + \lfloor \frac{3}{2}(l_1+l_2+l_3)\rfloor}.
\end{align*}
It follows from parity considerations and $k_1+k_2+k_3-l_1-l_2-l_3$ even that
\begin{align*}
\alpha_p(q,Q)\ll (k_1+1)(k_2+1) p^{\lfloor \frac{1}{2}k_1 \rfloor +\frac{1}{2}(k_1+k_2+k_3) + \frac{3}{2}( l_1+l_2+l_3)}.
\end{align*}

\textsc{The case $n=4$:}
By the case $n=3$
\begin{align*}
\max_{Q_{\underline{\nu}}} \alpha_p\big(&\sum_{i\geq 2} a_ip^{k_i}x_i^2,Q_{\underline{\nu}}\big)\\
&\ll (k_2+1)(k_3+1) p^{k_2+\lfloor \frac{k_3+k_4}{2}\rfloor + \lfloor \frac{3}{2}(l_1+l_2+l_3 + l_4+k_1)\rfloor -3\min_i\{l_i+\nu_i\}}
\end{align*}
and so by Proposition~\ref{prop:app-inductiveclaim}
\begin{align*}
\alpha_p(q,Q)\ll
(k_2+1)(k_3+1) p^{k_2+\lfloor \frac{k_3+k_4}{2}\rfloor + \lfloor \frac{3}{2}\sum_il_i\rfloor}\sum_{\underline{\nu}} p^{\min_i\{l_i+2\nu_i\}-\sum_i\nu_i}.
\end{align*}
In view of Remark~\ref{rem:app-twoindices} denote by $i_1,i_2 \in \{1,2,3,4\}$ two distinct indices with $l_{i_1}+2\nu_{i_1}=l_{i_2}+2\nu_{i_2}=\min_i\{l_{i}+2\nu_{i}\}\leq k_1$. 
Then $\min_i\{l_i+2\nu_i\} -\nu_{i_1}-\nu_{i_2} = \lfloor \frac{1}{2}(l_{i_1}+l_{i_2})\rfloor$ and
\begin{align*}
\sum_{\underline{\nu}} p^{\min_i\{l_i+2\nu_i\}-\nu_1-\nu_2-\nu_3-\nu_4}
\ll (k_1+1) p^{\lfloor \frac{1}{2}(l_{3}+l_{4})\rfloor}.
\end{align*}
Thus,
\begin{align*}
\alpha_p(q,Q)&\ll
\Big(\prod_{i=1}^3(k_i+1)\Big) p^{k_1+k_2+\lfloor \frac{k_3+k_4}{2}\rfloor + \lfloor \frac{3}{2}(l_1+l_2+l_3 + l_4)+\frac{1}{2}k_1\rfloor+\lfloor \frac{1}{2}(l_{3}+l_{4})\rfloor}\\
&\ll \Big(\prod_{i=1}^3(k_i+1)\Big) p^{k_1+ \lfloor \frac{1}{2}k_2\rfloor+\frac{1}{2}(k_1+k_2+k_3+k_4)+2(l_1+l_2+l_3 + l_4)}
\end{align*}
as claimed.
\end{proof}

\subsection{Local estimates at $p=2$}

We require a bound similar to Corollary~\ref{cor:app-estimatelocdensity} for $p=2$.

\begin{lemma}[Estimate for local densities at $p=2$]\label{lem:app-localestimateat2}
Assume that $q,Q$ are non-degenerate quadratic forms over $\Z_2$ in $4$ variables.
Write $(k_1,\ldots)$ and $(l_1,\ldots)$ respectively for the elementary divisors of $q$ and $Q$ at $2$. 
Then
\begin{align*}
\alpha_p(q,Q)
\ll \Big(\prod_{i < n}(k_i+1)\Big) 2^{\frac{3}{2}k_1 + k_2 + \frac{1}{2}(k_3+k_4) + 2 \sum_i l_i}.
\end{align*}
\end{lemma}

\begin{proof}[Sketch of proof]
The proof follows the same strategy as the proof of Corollary~\ref{cor:app-estimatelocdensity} as well as the following simple observation:

Whenever $q'$ is a quadratic form over $\Z_2$ in $m$ variables, then there exists a sublattice $\Lambda$ of $\Z_2^m$ of index dividing $2^{\lfloor \frac{m}{2}\rfloor}$ such that $q'|_\Lambda$ is `diagonalizable'.
That is, there exists a basis of $\Lambda$ on which $q'$ is diagonal.
In particular, this implies for any form $Q'$ over $\Z_2$ in $\geq m$ variables
\begin{align*}
\alpha_2(q',Q') \ll_m \alpha_2(q'|_\Lambda,Q');
\end{align*}
see for example \cite[Thm.~5.6.4]{Kitaoka-book}.

The above observation reduces our arguments to diagonal forms. Thus, the previous strategy of proof for $p \neq 2$ is applicable and the lemma follows.
\end{proof}

\subsection{Proof of the global estimates}
We turn to proving the main results of this appendix.

\begin{proof}[Proof of Theorem~\ref{thm:app-equalrank}]

Let $r(q,Q)$ be the number of representations of $q$ by $Q$.
By the Siegel mass formula we have
\begin{align*}
\frac{1}{\omega} \sum_i \frac{r(q,Q_i)}{|\mathrm{O}_{Q_i}(\Z)|}
= C \disc(Q)^{-2} \disc(q)^{-\frac{1}{2}} \prod_p \alpha_p(q,Q)
\end{align*}
where $\{Q_i\}$ runs over a set of representatives of the genus of $Q$, $\omega = \sum_i \frac{1}{|\mathrm{O}_{Q_i}(\Z)|}$, and the constant $C$ is absolute (it depends on the dimension, which is $4$ here).
As $\omega$ depends only on $Q$, this shows in particular
\begin{align*}
r(q,Q) \ll \disc(q)^{-\frac{1}{2}} \prod_p \alpha_p(q,Q).
\end{align*}
By \cite[Prop.~5.6.2]{Kitaoka-book}, we have 
\begin{align*}
\prod_{2 \neq p\, \nmid\, \disc(q)\disc(Q)} \alpha_p(q,Q)\ll 1
\end{align*}
where the implicit constant is absolute.
For all other primes, we apply the estimates on the local densities given in Corollary~\ref{cor:app-estimatelocdensity} and Lemma~\ref{lem:app-localestimateat2}. Thus,
\begin{align*}
r(q,Q) \ll_\varepsilon \disc(q)^\varepsilon \prod_{p}p^{k_1(p,q)} \prod_p p^{\lfloor \frac{1}{2}k_2(p,q)\rfloor}
\ll \disc(q)^\varepsilon \mathsf{m}_1(q) \mathsf{m}_2'(q).
\end{align*}
Here, to absorb rounding errors we used that whenever $\sum_{i} k_i(p,q)$ is odd for some odd prime $p$, then $p \mid \disc(Q)$ as $\sum_{i} (k_i(p,q)-k_i(p,Q))$ is even (or no representation exists and we conclude).
This finishes the proof.
\end{proof}

\begin{proof}[Proof of Corollary~\ref{cor:app-codim1}]
Let $\iota$ be any representation of $q$ by $Q$ and let $V = \iota(\Q^{3})$. 
The primitive lattice $V \cap \Z^4$ contains $\iota(\Z^{3})$ and the square of the index $[(V \cap \Z^4):\iota(\Z^{3})]$ is a divisor of $8\disc(q)$.
In particular, the index can take $\ll_\varepsilon |\disc(q)|^\varepsilon$ many values; fix one such value $d$.
Let $v \in V^\perp \cap \Z^4$ be a primitive vector (one of two).
For instance by~\cite[Prop.~5.1]{AMW-higherdim}, its quadratic value $Q(v)$ satisfies $Q(v) = \tfrac{m_1}{m_2}d$ for two divisors $m_1,m_2$ of $16\disc(Q)$.
Given that we fixed $d$, $Q(v)$ can thus take $\ll_\varepsilon |\disc(Q)|^\varepsilon$ many values; we fix one such value $d_0$. 
We have a representation
\begin{align*}
\iota': \Z^{4} \to \Z^4,\ x' = (x,x_0) \mapsto \iota(x) + x_0 v
\end{align*}
of the quadratic form $q'(x,x_0) = q(x) + d_0 x_0^2$ by $Q$. 
Note also that 
\begin{align*}
&|\disc(q')| = |\disc(q)|d_0 \leq |\disc(Q)| |\disc(q)|^2,\\
&\mathsf{m}_1(q')\mid 16\disc(Q)\mathsf{m}_1(q),\quad \mathsf{m}_2(q') \mid 16\disc(Q)\mathsf{m}_2(q).
\end{align*}
Overall, the above argument shows that the number of representations of $q$ by $Q$ is bounded by a sum of the number of representations of certain $4$-variable forms $q'$ by $Q$ where $q'$ ranges over $\ll_{Q,\varepsilon} |\disc(q)|^\varepsilon$ many forms of discriminant at most $|\disc(Q)| |\disc(q)|^2$ and with $\mathsf{m}_1(q')\mid 16\disc(Q)\mathsf{m}_1(q)$ as well as $\mathsf{m}_2(q') \mid 16\disc(Q)\mathsf{m}_2(q)$.
Thus, the corollary indeed follows from Theorem~\ref{thm:app-equalrank}.
\end{proof}

\section{A uniform version of Duke's theorem}\label{sec:Duke}
Let $\Hbf=\mathbf{PB}^\times$ where $B$ is a definite quaternion algebra over $\Q$.
Let $\widetilde{\Hbf}=\Bbf^1$ be the simply connected cover of $\Hbf$ and let $\rho:\widetilde{\Hbf}\to \Hbf$ be the natural isogeny. In this subsection, we choose and fix a maximal order $R$ in $B$. For each prime $p$ and natural number $m$, let $K_p^H[m]$ be the compact open subgroup of $\Hbf(\Q_p)$ defined by $K_p^H[m]=(1+p^mR_p)/(1+p^m\Z_p)$ if $m\geq 1$, and $K_p^H[0]=R_p^\times/\Z_p^\times$. Let $\disc(B)$ be the product of ramified primes in $B$.

Let $N$ be a positive integer and write $N=\prod_p p^{m_p}$. 
Let $\Av_p[m_p]$ be the averaging projection on the $K^H_p[m_p]$-invariant vectors, and put $\pr_p[0]=\Av_p[0]$, $\pr_p[m_p]=\Av_p[m_p]-\Av[m_p-1]$ for $m_p\geq1$. Let $\pr[N]=\prod_p\pr_p[m_p]$.
For $d\geq0$, we define a Sobolev norm on $C_c^\infty([\Hbf(\bA)])$ as follows:
\begin{equation}\label{eq:sob H}
    \Scal_d^H(f)^2 := \sum_{N} \left( N^d\sum_{\Dcal}\norm{\pr[N]\Dcal f(x)}_2^2 \right)
\end{equation}
where the inner sum is over degree at most $d$ monomials with respect to a fixed basis of the Lie algebra of $\Hbf(\R)$. Note that \cite{MV10} uses an alternative way to define the Sobolev norm, and it is comparable to the Sobolev norm defined above; see e.g. \cite[Thm 2.29]{Wu14}.

Let $L^2_{00}([\Hbf(\bA)])$ denotes the closed subspace of $L^2([\Hbf(\bA)])$ orthogonal to $\rho(\widetilde{\Hbf}(\bA))$-invariant functions. Let $h\in\Hbf(\bA)$. Let $E$ be a imaginary quadratic field admitting an embedding $\iota:E\hookrightarrow B$, and let $\Tbf\simeq \Res_{E/\Q}(\Gm)/\Gm$ be the associated torus in $\Hbf$. 
Let $\Ocal$ be the quadratic order in $E$ uniquely determined by the condition that $\Ocal_p := \Ocal\otimes_\Z\Z_p = \iota^{-1}(h_pR_ph_p^{-1})$.

\begin{theorem} \label{thm:uniform Duke}
    There exists $\delta,\theta>0$ such that for every $f\in C_c^\infty([\H(\bA)])\cap L^2_{00}([\Hbf(\bA)])$ and $0<\varepsilon<\frac{1}{4}$,
    \begin{equation} \label{eq:uniform Duke}
        \abs{\int_{[\Tbf(\bA) h]}f} \ll_{\varepsilon} (\disc B)^{\frac{3}{4}+\varepsilon} \left( \frac{\disc\Ocal}{D_E} \right)^{-(\frac{1}{4}-\frac{\theta}{2})+\varepsilon}D_E^{-\delta+\varepsilon}  \mathcal{S}_{10}^H(f).
    \end{equation}
\end{theorem}

The theorem can be viewed as a uniform version of Duke's theorem \cite{duke88}, in particular treating all aspects in which the discriminant can vary. We refer the readers to Michel's survey \cite{Mic21} for a beautiful introduction to Duke's theorem and recent developments surrounding it.
See also the discussions in \cite[\S4]{ELMVAnn} for references.
Theorem~\ref{thm:uniform Duke} and its proof are relatively standard knowledge in the community; we include a proof here merely for completeness.
We combine Waldspurger's formula \cite{Waldspurger} with subconvex estimates \cite{MV10} and estimates for local torus integrals to deduce Theorem~\ref{thm:uniform Duke}.
The estimates on local torus integrals proven below are contained in much greater generality in the literature, see for instance \cite[\S5]{CU} or \cite{wu2018}.
We are very thankful to Philippe Michel for so generously sharing his ideas and insights into the proof of the above variant of Duke's theorem.

Let $\delta>0$ be an absolute constant as in \cite[Thm.1.2]{MV10} (for subconvexity), see also \eqref{eq:subconvex bound} below.

\begin{theorem}\label{thm:uniform Duke fixed level}
    Let $\pi'=\otimes_v\pi'_v$ be an infinite dimensional cuspidal automorphic representation of $\Hbf(\bA)$ and let $\phi=\otimes_v\phi_v\in\pi'$ be a $K^H[N]$-fixed vector such that $(\phi,\phi)=1$. Then for every $\varepsilon>0$,
    \begin{equation*}
        \abs{\int_{[\Tbf(\bA) h]}\phi}^2 \ll_\varepsilon (\dim\pi'_\infty)^{1+\varepsilon}(\disc B)^{\frac{1}{2}+\varepsilon} N^{5+\varepsilon}\left( \frac{\disc\Ocal}{D_E} \right)^{-(\frac{1}{2}-\theta)+\varepsilon}D_E^{-2\delta+\varepsilon},
    \end{equation*}
    where
    $\theta=\frac{7}{64}$ is a constant toward the Ramanujan-Petersson conjecture \cite{Kim03}, $\dim\pi'_\infty$ is the dimension of the finite dimensional representation $\pi'_\infty$, and $D_E$ is the discriminant of $E$.
\end{theorem}

\begin{proof}[Proof of Theorem~\ref{thm:uniform Duke} assuming Theorem~\ref{thm:uniform Duke fixed level}] 

    Consider the real manifold $Y_N = \H(\Q)\backslash \H(\A) / K[N]$ on which $\H(\R)\simeq \SO(3)$ acts.
The number of orbits is $\ll [K[0]:K[N]] \ll N^3$ (cf.~\cite[Lemma 3.5]{Nori}) times the number of orbits on $Y_0$.
Moreover, the number of orbits on $Y_0$ is $\ll \disc(B)$ by \cite[Thm.~25.1.1]{voight} and so $Y_N$ is a union of $\ll N^3 \disc(B)$ many $\H(\R)$-orbits.
For $N$ sufficiently large, the stabilizer of each $\H(\R)$ is trivial and so, in particular, $L^2(Y_N)$ is a direct sum of $\ll N^3 \disc(B)$ many copies of $L^2(\H(\R))$ by Eichler mass formula.
We have
\begin{align*}
L^2(Y_N) = \bigoplus_\lambda \mathcal{H}_\lambda
\end{align*}
where $\lambda$ runs over all Casimir eigenvalues, $\mathcal{H}_\lambda$ is the eigenspace for Casimir eigenvalue $\lambda$, and $\dim(\mathcal{H}_\lambda) \ll N^3 \disc(B) (1+4\lambda)$ by Peter-Weyl theorem.

Write $I(\lambda)=\dim(\mathcal{H}_\lambda)$. For each $\lambda$, let $\{f_{\lambda,i}\}_{i \leq I(\lambda)}$ be an orthonormal basis of $\mathcal{H}_\lambda$ consisting of cusp forms coming from cuspidal automorphic representations of $\H(\A)$ and by Theorem~\ref{thm:uniform Duke fixed level},
\begin{align*}
\abs{\int_{[\Tbf(\bA) h]}f_{\lambda,i}}^2 
&\ll_\varepsilon \sqrt{1+4\lambda}^{1+\varepsilon}(\disc B)^{\frac{1}{2}+\varepsilon} N^{5+\varepsilon}\left( \frac{\disc\Ocal}{D_E} \right)^{-(\frac{1}{2}-\theta)+\varepsilon}D_E^{-2\delta+\varepsilon}.
\end{align*}
Let $f \in L^2(Y_N)\cong L^2(\H(\Q)\backslash\H(\bA))^{K[N]}$ and write $f = \sum_{\lambda}f_\lambda$ where $ \mathcal{H}_\lambda \ni f_{\lambda} = \sum_{i \leq I(\lambda)} a_{\lambda,i}f_{\lambda,i}$ for some $a_{\lambda,i}\in \C$.
Then for each $\lambda$
\begin{align*}
\abs{\int_{[\Tbf(\bA) h]}f_\lambda}
\ll_\varepsilon \sum_{i} |a_{\lambda,i}| \lambda^{\frac{1}{4}+\varepsilon}(\disc B)^{\frac{1}{4}+\varepsilon} N^{\frac{5}{2}+\varepsilon}\left( \frac{\disc\Ocal}{D_E} \right)^{-(\frac{1}{4}-\frac{\theta}{2})+\varepsilon}D_E^{-\delta+\varepsilon}
\end{align*}
and summing over $\lambda$
\begin{align*}
\abs{\int_{[\Tbf(\bA) h]}f}
\ll_\varepsilon C(f,\varepsilon) (\disc B)^{\frac{1}{4}+\varepsilon} N^{\frac{5}{2}+\varepsilon}\left( \frac{\disc\Ocal}{D_E} \right)^{-(\frac{1}{4}-\frac{\theta}{2})+\varepsilon}D_E^{-\delta+\varepsilon}
\end{align*}
where $C(f,\varepsilon) = \sum_{\lambda}\sum_{i\leq I(\lambda)} |a_{\lambda,i}| \lambda^{\frac{1}{4}+\varepsilon}$ for $\varepsilon\leq \frac{1}{4}$.
Observe that
\begin{align*}
C(f,\varepsilon) 
&\ll \sum_{\lambda}\Big(\sum_{i\leq I(\lambda)} |a_{\lambda,i}|^2 \Big)^{\frac{1}{2}} \lambda^{\frac{1}{4}+\varepsilon} \sqrt{I(\lambda)}\\
&\ll \disc(B)^{\frac{1}{2}} N^{\frac{3}{2}} \sum_{\lambda}\Big(\sum_{i\leq I(\lambda)} |a_{\lambda,i}|^2 \Big)^{\frac{1}{2}} \lambda\\
&= \disc(B)^{\frac{1}{2}} N^{\frac{3}{2}} \sum_{\lambda} \norm{\Delta^2 f_{\lambda}} \lambda^{-1}\\
&\ll \disc(B)^{\frac{1}{2}} N^{\frac{3}{2}} \Big(\sum_{\lambda} \norm{\Delta^2 f_{\lambda}}^2\Big)^{\frac{1}{2}} 
= \disc(B)^{\frac{1}{2}} N^{\frac{3}{2}} \norm{\Delta^2 f}.
\end{align*}
Therefore,
\begin{align*}
\abs{\int_{[\Tbf(\bA) h]}f}
\ll_{\varepsilon} (\disc B)^{\frac{3}{4}+\varepsilon} N^{4+\varepsilon}\left( \frac{\disc\Ocal}{D_E} \right)^{-(\frac{1}{4}-\frac{\theta}{2})+\varepsilon}D_E^{-\delta+\varepsilon}\norm{\Delta^2 f}.
\end{align*}

Lastly, for general $f$ apply the above to $\mathrm{pr}[N]f$.
Apply Cauchy-Schwarz again
\begin{align*}
\abs{\int_{[\Tbf(\bA) h]}f} \ll_{\varepsilon} (\disc B)^{\frac{3}{4}+\varepsilon} \left( \frac{\disc\Ocal}{D_E} \right)^{-(\frac{1}{4}-\frac{\theta}{2})+\varepsilon}D_E^{-\delta+\varepsilon}  \sum_N N^{4+\varepsilon}\norm{\pr[N]\Delta^2f}
\end{align*}
where
\begin{align*}
\sum_N N^{4+\varepsilon}\norm{\pr[N]\Delta^2f}
\ll \big(\sum_N N^{10+\varepsilon}\norm{\pr[N] \Delta^2f}^2\Big)^{\frac{1}{2}}
\ll \mathcal{S}_{10}^H(f)
\end{align*}
using that $\Delta$ is a differential operator of degree two.
Hence
\begin{equation*}
    \abs{\int_{[\Tbf(\bA) h]}f} \ll_{\varepsilon} (\disc B)^{\frac{3}{4}+\varepsilon} \left( \frac{\disc\Ocal}{D_E} \right)^{-(\frac{1}{4}-\frac{\theta}{2})+\varepsilon}D_E^{-\delta+\varepsilon}  \mathcal{S}_{10}^H(f)
\end{equation*}
as claimed.
\end{proof}

The rest of the subsection is devoted to proving Theorem~\ref{thm:uniform Duke fixed level}.
\subsubsection{Waldspurger formula}
Let $\pi$ be the Jacquet-Langlands transfer of $\pi'$. For $\phi=\otimes_v \phi_v$ in $\pi'=\otimes_v\pi'_v$ we define
    \[
    \alpha(\phi, h) = \prod_v \alpha_v(\phi_v, h_v),
    \]
where
\[
\alpha_v(\phi_v, h_v)=\frac{L(1,\eta_v)L(1,\pi_v,\Ad)}{\zeta_v(2)L(\frac{1}{2},\pi_{E,v})}\int_{\Tbf(\Q_v)}\frac{(\pi'_v(h_v^{-1}th_v)\phi_v,\phi_v)}{(\phi_v,\phi_v)}dt.
\]
Here $dt$ denotes the quotient Haar measure on $\Tbf(\Q_v)\cong \Q_v^\times\backslash E_v^\times$ under Tate's measure normalization, see \cite[Sect.1.6]{YZZ13}.
Note that $\alpha_v(\phi_v, h_v)=1$ for all but finitely many $v$'s.

Recall the Waldspurger formula from \cite[Prop.~7]{Waldspurger} or \cite[Thm.~1.4.2]{YZZ13}
\begin{equation} \label{eq:Waldspurger formula}
    \frac{\lvert \int_{[\Tbf(\bA) h]}\phi \rvert^2}{(\phi,\phi)} = \frac{\zeta(2)L(\frac{1}{2},\pi)L(\frac{1}{2},\pi\otimes\eta)}{8L(1,\eta)^2L(1,\pi,\Ad)}\alpha(\phi,h).
\end{equation}

\subsubsection{Local estimates}
Let $\phi=\otimes_v\phi_v$ be a $K^H[N]$-fixed vector in $\pi$.
By \cite{CHH88} and \cite[Lemma~9.1]{Venkatesh-sparse} we have
\begin{equation}\label{eq:first bound of matrix coefficients}
    ( h_p\phi_p,\phi_p ) \leq [K_{p} \colon K_p[m_p]](\phi_p,\phi_p)\Xi_p(h_p)^{1-2\theta}.
\end{equation}
where $h\in\Hbf(\Q_p)$, $\Xi_p$ is the Harish-Chandra spherical function \cite[p.~505]{bump}, and we can take $\theta=\frac{7}{64}$ by the work of Kim and Sarnak \cite[App.~2]{Kim03}.
The index $[K_{p} \colon K_p[m_p]]$ has the following bound:
\begin{equation}\label{eq:index upper bound}
    [K_{p} \colon K_p[m_p]] \ll p^{3m_p},
\end{equation}
see e.g.~\cite[(A.8)]{EMMV}.

Combining \eqref{eq:first bound of matrix coefficients} and \eqref{eq:index upper bound} we get
\begin{equation}\label{eq:matrix coefficients bound}
    ( h_p\phi_p,\phi_p ) \ll p^{3m_p}(\phi_p,\phi_p)\Xi_p(h_p)^{1-2\theta}.
\end{equation}

We also recall the following consequence of Macdonald's formula (see e.g. \cite[Thm.~4.6.6]{bump}).
Let $a=\diag(p,1)$, we have
\begin{equation}\label{eq:Harish-Chandra spherical function}
    \Xi_p(a^k) = \frac{(1-p^{-1})k+1+p^{-1}}{1+p^{-1}}p^{-\frac{k}{2}}.
\end{equation}

Following from the definition of local $L$-functions (see e.g. \cite{MW09}), we have
\begin{equation}\label{eq:local L-factor bound}
    \frac{L(1,\eta_v)L(1,\pi_v,\Ad)}{\zeta_v(2)L(\frac{1}{2},\pi_{E,v})}\ll 1,
\end{equation}
where the implied constant is absolute.

In the following subsections, we give local estimates for $\alpha_p(\phi_p,h_p)$ in different cases.

\subsubsection{$p\nmid\disc(\Ocal)\disc(B)N$}
In this spherical case, we have $\alpha(\phi_p,h_p)=1$; see \cite[Sect.~1.4.1]{YZZ13}.

\subsubsection{$p\nmid\disc(\Ocal)$, $p\mid\disc(B)N$} 
    In this case
    \begin{equation} \label{eq:p divides disc(R) but not disc(O)}
        \alpha_p(\phi_p,h_p) \ll 1,
    \end{equation}
where the implied constant is absolute.
This follows from the trivial bound $(h_p\phi_p,\phi_p)\leq(\phi_p,\phi_p)$ and \eqref{eq:local L-factor bound}.

\subsubsection{$p \nmid \disc(B)$, $p\mid\disc(\Ocal)$, $p$ split in $E$}
We recall the following integral version of the Skolem-Noether theorem \cite[p.~44, Thm.~3.2]{Vigneras} or \cite[Prop.~3.7]{Linnikthmexpander}, which will be useful in this subsection and the subsequent two subsections.

\begin{lemma}\label{lem:integral Skolem-Noether}
    Let $\Ocal_p'$ be an order in $E_p$. The group $\GL_2(\Z_p)$ acts transitively by conjugation on the set of optimal embeddings $\iota:\Ocal_p'\hookrightarrow M_2(\Z_p)$.
    Here, $\iota$ is optimal if $\iota(E_p) \cap M_2(\Z_p) = \iota(\Ocal_p')$.
\end{lemma}

Let $\alpha=\frac{D_E+\sqrt{D_E}}{2}$ so that the ring of integers in $E$ is $\Ocal_E=\Z[\alpha]$. 
Any order in the quadratic \'etale algebra $E_p=E\otimes \Q_p\simeq \Q_p \times \Q_p$ is of the form $\Ocal_{E_p,r}=\Z_p+p^r\Ocal_{E_p}=\Z_p+p^r\Z_p\alpha$ for some nonnegative integer $r$ where $\Ocal_{E_p}$ is the maximal order.
Let $\iota_r: E_p\to M_2(\Q_p)$ be the embedding
\begin{equation*}
    a+b\alpha \mapsto aI_2+b\begin{pmatrix}
        0 & p^{-2r}\Nr\alpha \\
        -p^{2r} & \Tr\alpha
    \end{pmatrix}.
\end{equation*} 
As $p \nmid D_E$, it follows that $\iota_r^{-1}(M_2(\Z_p)) = \Ocal_{E_p,r}$ so that $\iota_r$ is an opimal embedding for $\Ocal_{E_p,r}$.
By abuse of notation we will also let $\iota_r$ denote the induced embedding $E_p^\times\to\GL_2(\Q_p)$.

Let
\begin{equation*}
    \Ocal_{E_p,r}^{(k)}:=\{ t\in\Ocal_{E_p,r}\setminus p\Ocal_{E_p,r}\mid \Nr_{E_p/\Q_p}(t)\in p^k\Z_p^\times \}.
\end{equation*}
Let $K_{p,\max} = \GL_2(\Z_p)$. Then for every $r\geq0$,
\begin{equation*}
    \iota_r^{-1}(K_{p,\max}a^kK_{p,\max}) = \Z_p^{\times}\backslash\Ocal_{E_p,r}^{(k)}.
\end{equation*}
Let $\vol$ denote the Haar measure on $\Tbf(\Q_p)\cong \Q_p^\times\backslash E_p^{\times}$ with $\vol(\Z_p^\times\backslash\Ocal_{E_p}^\times)=1$. 
An explicit calculation shows
\begin{equation*}
    \vol(\Z_p^{\times}\backslash\Ocal_{E_p,r}^{(k)}) = \begin{cases}
        1 & k=r=0,\\
        (p-1)^{-1}p^{-(r-1)} & k=0, r\geq1,\\
        p^{-(r-\frac{k}{2})} & 0<k<2r, 2\mid k,\\
        \frac{p}{p-1} & k=2r>0,\\
        2 & k>2r, \\
        0 & \text{otherwise.}
    \end{cases}
\end{equation*}
Using additionally \eqref{eq:matrix coefficients bound}, \eqref{eq:Harish-Chandra spherical function}, and \eqref{eq:local L-factor bound}, we have
\begin{equation} \label{eq:E_p split}
    \begin{split}
        \alpha_p(\phi_p,h_p) & \ll \int_{\Tbf(\Q_p)}\frac{(\pi_p(h_p^{-1}th_p)\phi_p,\phi_p)}{(\phi_p,\phi_p)} dt \\
        & = \sum_{k=0}^{\infty}\int_{\iota_r^{-1}(K_{p,\max}a^kK_{p,\max})}\frac{(\pi_p(h_p^{-1}th_p)\phi_p,\phi_p)}{(\phi_p,\phi_p)} dt \\
        & \ll p^{3m_p}\sum_{k=0}^{\infty}(k+1)p^{-(\frac{1}{2}-\theta)k}\vol(\Z_p^{\times}\backslash\Ocal_{E_p,r}^{(k)}) \\
        & \ll p^{3m_p}(r+1)^2p^{-r(1-2\theta)},
    \end{split}
\end{equation}
where the implied constant is absolute.

\subsubsection{$p \nmid \disc(B)$, $p\mid\disc(\Ocal)$, $p$ inert in $E$}
If $p$ is inert in $E$, then $E_p/\Q_p$ is an unramified quadratic field extension. 
Let $\iota_r: E_p\to M_2(\Q_p)$ be as in the previous section.
In this case we have
\begin{equation*}
    \vol(\Z_p^\times\backslash\Ocal_{E_p,r}^{(k)}) = 
    \begin{cases}
        1 & k=r=0,\\
        \frac{1}{p^{r-1}(p+1)} & k=0, r>0,\\
        \frac{p-1}{p+1}p^{-(r-\frac{k}{2})} & 0< k< 2r, 2\mid k,\\
        \frac{p}{p+1} & k=2r>0,\\
        0 & \text{otherwise}.
    \end{cases}
\end{equation*}

It follows that
\begin{equation} \label{eq:E_p unramified}
    \begin{split}
        \alpha_p(\phi_p, h_p) & \ll \int_{\Tbf(\Q_p)}\frac{(\pi_p(h_p^{-1}th_p)\phi_p,\phi_p)}{(\phi_p,\phi_p)} dt \\
        & = \sum_{k=0}^{\infty}\int_{\iota_r^{-1}(K_{p,\max}a^kK_{p,\max})}\frac{(\pi_p(h_p^{-1}th_p)\phi_p,\phi_p)}{(\phi_p,\phi_p)} dt \\
        & \ll p^{3m_p}\sum_{k=0}^{\infty}(k+1)p^{-(\frac{1}{2}-\theta)k}\vol(\Z_p^{\times}\backslash\Ocal_{E_p,r}^{(k)}) \\
        &\ll p^{3m_p}(r+1)^2p^{-r(1-2\theta)},
    \end{split}
\end{equation}
where the implied constant is absolute.

\subsubsection{$p \nmid \disc(B)$, $p\mid\disc(\Ocal)$, $p$ ramified in $E$}
In this case $E_p$ is a quadratic ramified extension of $\Q_p$. 
Taking Tate's normalization of the Haar measure on $\Tbf(\Q_p)$, we have $\vol(\Tbf(\Q_p))= 2\vol(\Z_p^\times\backslash \Ocal_{E_p}^\times) = 2p^{-\frac{1}{2}}$.
Let $\iota_r$ be as in the previous subsection. Then the optimal embedding of $\Ocal_p$ into $R_p$ is integrally conjugated to a unique $\iota_r$ for some $r\geq1$. In this case we have
\begin{equation*}
    \iota_r^{-1}(K_{p,\max}a^kK_{p,\max}) = \Z_p^{\times}\backslash\Ocal_{E_p,r-1}^{(k)}.
\end{equation*}
We calculate
\begin{equation*}
    \vol(\Z_p^\times\backslash\Ocal_{E_p,r}^{(k)}) = 
    \begin{cases}
        p^{-(r+\frac{1}{2})} & k=0\\
        p^{-(r-\frac{k}{2}+\frac{1}{2})}-p^{-(r-\frac{k}{2}+\frac{3}{2})} & 0< k\leq 2r, 2\mid k,\\
        p^{-\frac{1}{2}} & k=2r+1,\\
        0 & \text{otherwise}.
    \end{cases}
\end{equation*}
It follows that
\begin{equation} \label{eq:E_p ramified}
    \begin{split}
        \alpha_p(\phi_p,h_p) & \ll \int_{\Tbf(\Q_p)}\frac{(\pi_p(h_p^{-1}th_p)\phi_v,\phi_v)}{(\phi_v,\phi_v)} dt \\
        & = \sum_{k=0}^{\infty}\int_{\iota_r^{-1}(K_{p,\max}a^kK_{p,\max})}\frac{(\pi_p(h_p^{-1}th_p)\phi_v,\phi_v)}{(\phi_v,\phi_v)} dt \\
        & \ll p^{3m_p}\sum_{k=0}^{\infty}(k+1)p^{-(\frac{1}{2}-\theta)k}\vol(\Z_p^{\times}\backslash\Ocal_{E_p,r-1}^{(k)}) \\
        & \ll p^{3m_p}r^2p^{-(r-1)(1-2\theta)-\frac{1}{2}} ,
    \end{split}
\end{equation}
where the implied constant is absolute.

\subsubsection{$p \mid \disc(B)$ ramified, $p\mid\disc(\Ocal)$}
In this case, $B_p = B \otimes \Q_p$ is the unique quaternion divison algebra over $\Q_p$. Since $E_p$ is embedded into $B_p$, $E_p$ has to be a field extension of $\Q_p$. There is a unique maximal order $R_{p,\max}$ in $B_p$, and thus $\Ocal_{E_p}$ is the only order that can be optimally embedded into $R_{p,\max}$.
In this case we have the trivial bound (using $(h_p\phi_p,\phi_p)\leq(\phi_p,\phi_p)$)
\begin{equation} \label{eq:B_p ramified}
    \alpha_p(\phi_p,h_p) \ll 1,
\end{equation}
where the implied constant is absolute.

\subsubsection{$v=\infty$}
As for the ramified primes, we have the trivial bound for the local factor
\begin{equation} \label{eq:archimedean}
    \alpha_\infty(\phi_\infty,h_\infty) \ll 1,
\end{equation}
where the implied constant is absolute.

\subsubsection{Combining the local estimates}
\begin{lemma}\label{lem:combine local estimates}
    
    \begin{equation}
        \alpha(\phi,h) \ll_\varepsilon (\disc B)^\varepsilon N^3\left( \frac{\disc\Ocal}{D_E} \right)^{-(\frac{1}{2}-\theta)+\varepsilon}D_E^{-\frac{1}{2}}.
    \end{equation}
\end{lemma}
\begin{proof}
    This follows from \eqref{eq:p divides disc(R) but not disc(O)}\eqref{eq:E_p split}\eqref{eq:E_p unramified}\eqref{eq:E_p ramified}\eqref{eq:B_p ramified}\eqref{eq:archimedean}.
\end{proof}

We can now finish the proof of Theorem~\ref{thm:uniform Duke fixed level}.
\begin{proof}[Proof of Theorem~\ref{thm:uniform Duke fixed level}]
    We have the subconvexity bound \cite[Thm.1.2]{MV10}
    \begin{equation}\label{eq:subconvex bound}
        L\big(\frac{1}{2},\pi_E\big)=L\big(\frac{1}{2},\pi\big)L\big(\frac{1}{2},\pi\otimes\eta_E\big)\ll C(\pi)^{\frac{1}{2}}D_E^{\frac{1}{2}-2\delta+\varepsilon},
    \end{equation}
    where $C(\pi)$ denotes the analytic conductor of $\pi$ as defined in \cite{IS00}.
    We also have the zero-free region type estimate by the appendix of \cite{HL94}:
    \begin{equation*}
        L(1,\pi,\Ad)\gg_\varepsilon C(\pi)^{-\varepsilon}.
    \end{equation*}

    Since $\zeta_E(s)=\zeta(s)L(s,\eta_E)$, we have $L(1,\eta_E)=\Res_{s=1}\zeta_E(s)$. Therefore, by Brauer-Siegel theorem, for any $\varepsilon>0$ we have
    \begin{equation*}
        D_E^{-\varepsilon}\ll_{\varepsilon} L(1,\eta_E) \ll_{\varepsilon} D_E^\varepsilon.
    \end{equation*}
    Since $\pi'$ has a nonzero $K^H[N]$-fixed vector, by \cite[footnote 19]{MV10} and \cite[Prop.~3.3]{wu2018} we have 
    \begin{equation*}
       \disc(B)N(\dim\pi'_\infty)^2 \ll C(\pi)\ll \disc(B)N^4(\dim\pi'_\infty)^2,
    \end{equation*}
    where the implied constants are absolute and we have used $C(\pi_\infty)\asymp (\dim\pi'_\infty)^2$.
    
    Now, in view of the Waldspurger formula \eqref{eq:Waldspurger formula} and Lemma~\ref{lem:combine local estimates},
    \begin{equation*}
        \begin{split}
            \abs{\int_{[\Tbf(\bA)h]}\phi}^2 &= \frac{\zeta(2)L(\frac{1}{2},\pi)L(\frac{1}{2},\pi\otimes\eta)}{8L(1,\eta)^2L(1,\pi,\Ad)}\alpha(\phi,h)\\
            &\ll_\varepsilon \frac{C(\pi)^{\frac{1}{2}}D_E^{\frac{1}{2}-2\delta+\epsilon}}{D_E^{-\varepsilon}C(\pi)^{-\varepsilon}}(\disc B)^\varepsilon N^3\left( \frac{\disc\Ocal}{D_E} \right)^{-(\frac{1}{2}-\theta)+\varepsilon}D_E^{-\frac{1}{2}} \\
            &= C(\pi)^{\frac{1}{2}+\varepsilon}(\disc B)^\varepsilon N^3\left(\frac{\disc\Ocal}{D_E} \right)^{-(\frac{1}{2}-\theta)+\varepsilon}D_E^{-2\delta+\varepsilon} \\
            &\ll (\dim\pi'_\infty)^{1+\varepsilon}(\disc B)^{\frac{1}{2}+\varepsilon} N^{5+\varepsilon}\left( \frac{\disc\Ocal}{D_E} \right)^{-(\frac{1}{2}-\theta)+\varepsilon}D_E^{-2\delta+\varepsilon}.
        \end{split}
    \end{equation*}
    The implied constant depends merely on $\varepsilon$. 
\end{proof}

\end{appendix}

\bibliographystyle{plain}
\bibliography{Bibliography}

@book {Knus,
    AUTHOR = {Knus, Max-Albert},
     TITLE = {Quadratic forms, {C}lifford algebras and spinors},
    SERIES = {Semin\'arios de Matem\'atica [Seminars in Mathematics]},
    VOLUME = {1},
 PUBLISHER = {Universidade Estadual de Campinas, Instituto de Matem\'atica,
              Estat\'istica e Ci\^encia da Computa\c c\~ao, Campinas},
      YEAR = {1988},
     PAGES = {vi+135},
   MRCLASS = {11Exx (11E88)},
  MRNUMBER = {1099376},
MRREVIEWER = {Manuel\ Ojanguren},
}

@article {Nori,
    AUTHOR = {Nori, Madhav V.},
     TITLE = {On subgroups of {${\rm GL}_n({\bf F}_p)$}},
   JOURNAL = {Invent. Math.},
  FJOURNAL = {Inventiones Mathematicae},
    VOLUME = {88},
      YEAR = {1987},
    NUMBER = {2},
     PAGES = {257--275},
      ISSN = {0020-9910,1432-1297},
   MRCLASS = {20G40 (20G30 20J06)},
  MRNUMBER = {880952},
MRREVIEWER = {James\ E.\ Humphreys},
       DOI = {10.1007/BF01388909},
       URL = {https://doi.org/10.1007/BF01388909},
}

@book {Vigneras,
    AUTHOR = {Vign\'eras, Marie-France},
     TITLE = {Arithm\'etique des alg\`ebres de quaternions},
    SERIES = {Lecture Notes in Mathematics},
    VOLUME = {800},
 PUBLISHER = {Springer, Berlin},
      YEAR = {1980},
     PAGES = {vii+169},
      ISBN = {3-540-09983-2},
   MRCLASS = {12A80 (16A18)},
  MRNUMBER = {580949},
MRREVIEWER = {J.\ H. H. Chalk},
}

@article{KWY-2in4announcement,
      title={Representations of binary forms by quaternary quadratic forms}, 
      author={Wooyeon Kim and Andreas Wieser and Pengyu Yang},
      year={2025},
    journal={arXiv preprint 2511.22877},
      eprint={2511.22877},
      archivePrefix={arXiv},
      primaryClass={math.NT},
      url={https://arxiv.org/abs/2511.22877}, 
}

@article {Salberger-crelle,
    AUTHOR = {Salberger, Per},
     TITLE = {On the density of rational and integral points on algebraic
              varieties},
   JOURNAL = {J. Reine Angew. Math.},
  FJOURNAL = {Journal f\"ur die Reine und Angewandte Mathematik. [Crelle's
              Journal]},
    VOLUME = {606},
      YEAR = {2007},
     PAGES = {123--147},
      ISSN = {0075-4102,1435-5345},
   MRCLASS = {11G35 (11G50 14G05)},
  MRNUMBER = {2337644},
MRREVIEWER = {Ulrich\ Derenthal},
       DOI = {10.1515/CRELLE.2007.037},
       URL = {https://doi.org/10.1515/CRELLE.2007.037},
}

@article {Broberg,
    AUTHOR = {Broberg, Niklas},
     TITLE = {A note on a paper by {R}. {H}eath-{B}rown: ``{T}he density of
              rational points on curves and surfaces'' [{A}nn. of {M}ath.
              (2) {\bf 155} (2002), no. 2, 553--595; MR1906595]},
   JOURNAL = {J. Reine Angew. Math.},
  FJOURNAL = {Journal f\"ur die Reine und Angewandte Mathematik. [Crelle's
              Journal]},
    VOLUME = {571},
      YEAR = {2004},
     PAGES = {159--178},
      ISSN = {0075-4102,1435-5345},
   MRCLASS = {11G35 (11G05 14G05)},
  MRNUMBER = {2070148},
MRREVIEWER = {Timothy\ D.\ Browning},
       DOI = {10.1515/crll.2004.039},
       URL = {https://doi.org/10.1515/crll.2004.039},
}

@incollection {BourgainDemeter-Siegelmass,
    AUTHOR = {Bourgain, Jean and Demeter, Ciprian},
     TITLE = {Three applications of the {S}iegel mass formula},
 BOOKTITLE = {Geometric aspects of functional analysis. {V}ol. {I}},
    SERIES = {Lecture Notes in Math.},
    VOLUME = {2256},
     PAGES = {99--111},
 PUBLISHER = {Springer, Cham},
      YEAR = {[2020] \copyright 2020},
      ISBN = {978-3-030-36020-7; 978-3-030-36019-1},
   MRCLASS = {11E45 (11L07)},
  MRNUMBER = {4175745},
MRREVIEWER = {Anton\ Deitmar},
       DOI = {10.1007/978-3-030-36020-7\_6},
       URL = {https://doi.org/10.1007/978-3-030-36020-7_6},
}

@article {BourgainDemeter-Fourierrestrictionsphere,
    AUTHOR = {Bourgain, Jean and Demeter, Ciprian},
     TITLE = {New bounds for the discrete {F}ourier restriction to the
              sphere in 4{D} and 5{D}},
   JOURNAL = {Int. Math. Res. Not. IMRN},
  FJOURNAL = {International Mathematics Research Notices. IMRN},
      YEAR = {2015},
    NUMBER = {11},
     PAGES = {3150--3184},
      ISSN = {1073-7928,1687-0247},
   MRCLASS = {42B10 (11E04)},
  MRNUMBER = {3373047},
MRREVIEWER = {Juan\ Luis\ Varona},
       DOI = {10.1093/imrn/rnu036},
       URL = {https://doi.org/10.1093/imrn/rnu036},
}

@article {Siegel-I,
    AUTHOR = {Siegel, Carl Ludwig},
     TITLE = {\"{U}ber die analytische {T}heorie der quadratischen {F}ormen},
   JOURNAL = {Ann. of Math. (2)},
  FJOURNAL = {Annals of Mathematics. Second Series},
    VOLUME = {36},
      YEAR = {1935},
    NUMBER = {3},
     PAGES = {527--606},
      ISSN = {0003-486X,1939-8980},
   MRCLASS = {99-04},
  MRNUMBER = {1503238},
       DOI = {10.2307/1968644},
       URL = {https://doi.org/10.2307/1968644},
}

@incollection {Cogdell-squares,
    AUTHOR = {Cogdell, James W.},
     TITLE = {On sums of three squares},
      NOTE = {Les XXII\`emes Journ\'ees Arithmetiques (Lille, 2001)},
   JOURNAL = {J. Th\'eor. Nombres Bordeaux},
  FJOURNAL = {Journal de Th\'eorie des Nombres de Bordeaux},
    VOLUME = {15},
      YEAR = {2003},
    NUMBER = {1},
     PAGES = {33--44},
      ISSN = {1246-7405,2118-8572},
   MRCLASS = {11F66 (11E25 11F30)},
  MRNUMBER = {2018999},
MRREVIEWER = {Emmanuel\ P.\ Royer},
       DOI = {10.5802/jtnb.385},
       URL = {https://doi.org/10.5802/jtnb.385},
}

@article {DukeSchulzePillot,
    AUTHOR = {Duke, William and Schulze-Pillot, Rainer},
     TITLE = {Representation of integers by positive ternary quadratic forms
              and equidistribution of lattice points on ellipsoids},
   JOURNAL = {Invent. Math.},
  FJOURNAL = {Inventiones Mathematicae},
    VOLUME = {99},
      YEAR = {1990},
    NUMBER = {1},
     PAGES = {49--57},
      ISSN = {0020-9910,1432-1297},
   MRCLASS = {11E36 (11D85 11E20 11H55 11P21 11P55)},
  MRNUMBER = {1029390},
MRREVIEWER = {Mark\ Sheingorn},
       DOI = {10.1007/BF01234411},
       URL = {https://doi.org/10.1007/BF01234411},
}

@article {Iwaniec-halfintegral,
    AUTHOR = {Iwaniec, Henryk},
     TITLE = {Fourier coefficients of modular forms of half-integral weight},
   JOURNAL = {Invent. Math.},
  FJOURNAL = {Inventiones Mathematicae},
    VOLUME = {87},
      YEAR = {1987},
    NUMBER = {2},
     PAGES = {385--401},
      ISSN = {0020-9910,1432-1297},
   MRCLASS = {11F37 (11F30)},
  MRNUMBER = {870736},
MRREVIEWER = {Marvin\ I.\ Knopp},
       DOI = {10.1007/BF01389423},
       URL = {https://doi.org/10.1007/BF01389423},
}

@article{effectivesemisimple,
      title={Effective equidistribution of semisimple adelic periods and representations of quadratic forms}, 
      author={Manfred Einsiedler and Elon Lindenstrauss and Amir Mohammadi and Andreas Wieser},
      year={2025},
    journal={arXiv preprint 2503.21068}
}

@article {Salberger-PLMS,
    AUTHOR = {Salberger, Per},
     TITLE = {Counting rational points on projective varieties},
   JOURNAL = {Proc. Lond. Math. Soc. (3)},
  FJOURNAL = {Proceedings of the London Mathematical Society. Third Series},
    VOLUME = {126},
      YEAR = {2023},
    NUMBER = {4},
     PAGES = {1092--1133},
      ISSN = {0024-6115,1460-244X},
   MRCLASS = {11D45 (11D72 11G35 14G05 14G40)},
  MRNUMBER = {4574827},
MRREVIEWER = {Paul\ M.\ Voutier},
       DOI = {10.1112/plms.12508},
       URL = {https://doi.org/10.1112/plms.12508},
}

@article {HeathBrown-annals,
    AUTHOR = {Heath-Brown, D. R.},
     TITLE = {The density of rational points on curves and surfaces},
   JOURNAL = {Ann. of Math. (2)},
  FJOURNAL = {Annals of Mathematics. Second Series},
    VOLUME = {155},
      YEAR = {2002},
    NUMBER = {2},
     PAGES = {553--595},
      ISSN = {0003-486X,1939-8980},
   MRCLASS = {11G35 (11G50 14G05 14G40)},
  MRNUMBER = {1906595},
MRREVIEWER = {Carlo\ Gasbarri},
       DOI = {10.2307/3062125},
       URL = {https://doi.org/10.2307/3062125},
}

@article {BombieriPila,
    AUTHOR = {Bombieri, E. and Pila, J.},
     TITLE = {The number of integral points on arcs and ovals},
   JOURNAL = {Duke Math. J.},
  FJOURNAL = {Duke Mathematical Journal},
    VOLUME = {59},
      YEAR = {1989},
    NUMBER = {2},
     PAGES = {337--357},
      ISSN = {0012-7094,1547-7398},
   MRCLASS = {11P21 (11D99)},
  MRNUMBER = {1016893},
MRREVIEWER = {Ulrich\ Rausch},
       DOI = {10.1215/S0012-7094-89-05915-2},
       URL = {https://doi.org/10.1215/S0012-7094-89-05915-2},
}

@article {Waldspurger,
    AUTHOR = {Waldspurger, J.-L.},
     TITLE = {Sur les valeurs de certaines fonctions {$L$} automorphes en
              leur centre de sym\'{e}trie},
   JOURNAL = {Compositio Math.},
  FJOURNAL = {Compositio Mathematica},
    VOLUME = {54},
      YEAR = {1985},
    NUMBER = {2},
     PAGES = {173--242},
      ISSN = {0010-437X},
   MRCLASS = {11F70 (11F67 22E55)},
  MRNUMBER = {783511},
MRREVIEWER = {Stephen Gelbart},
       URL = {http://www.numdam.org/item?id=CM_1985__54_2_173_0},
}

@incollection {Linnikthmexpander,
    AUTHOR = {J. S. Ellenberg and Ph. Michel and A. Venkatesh},
     TITLE = {Linnik's ergodic method and the distribution of integer points
              on spheres},
 BOOKTITLE = {Automorphic representations and {$L$}-functions},
    SERIES = {Tata Inst. Fundam. Res. Stud. Math.},
    VOLUME = {22},
     PAGES = {119--185},
 PUBLISHER = {Tata Inst. Fund. Res., Mumbai},
      YEAR = {2013},
   MRCLASS = {11K36 (11E12)},
  MRNUMBER = {3156852},
MRREVIEWER = {Vladimir S. Anashin},
}

@article {IlyaAnnals,
    AUTHOR = {Khayutin, Ilya},
     TITLE = {Joint equidistribution of {CM} points},
   JOURNAL = {Ann. of Math. (2)},
  FJOURNAL = {Annals of Mathematics. Second Series},
    VOLUME = {189},
      YEAR = {2019},
    NUMBER = {1},
     PAGES = {145--276},
      ISSN = {0003-486X},
   MRCLASS = {11G18 (37A17)},
  MRNUMBER = {3898709},
MRREVIEWER = {Thomas Ward},
       DOI = {10.4007/annals.2019.189.1.4},
       URL = {https://doi.org/10.4007/annals.2019.189.1.4},
}

@article {MVIHES,
    AUTHOR = {Ph. Michel and A. Venkatesh},
     TITLE = {The subconvexity problem for {${\rm GL}_2$}},
   JOURNAL = {Publ. Math. Inst. Hautes \'{E}tudes Sci.},
  FJOURNAL = {Publications Math\'{e}matiques. Institut de Hautes \'{E}tudes
              Scientifiques},
    NUMBER = {111},
      YEAR = {2010},
     PAGES = {171--271},
      ISSN = {0073-8301},
   MRCLASS = {11F67 (11F37 11F70 22E55 58J51)},
  MRNUMBER = {2653249},
MRREVIEWER = {Andre Reznikov},
       DOI = {10.1007/s10240-010-0025-8},
       URL = {https://doi.org/10.1007/s10240-010-0025-8},
}

@article {CU,
    AUTHOR = {Clozel, Laurent and Ullmo, Emmanuel},
     TITLE = {\'{E}quidistribution de mesures alg\'{e}briques},
   JOURNAL = {Compos. Math.},
  FJOURNAL = {Compositio Mathematica},
    VOLUME = {141},
      YEAR = {2005},
    NUMBER = {5},
     PAGES = {1255--1309},
      ISSN = {0010-437X},
   MRCLASS = {22E40 (11F70)},
  MRNUMBER = {2157138},
MRREVIEWER = {A. Raghuram},
       DOI = {10.1112/S0010437X0500148X},
       URL = {https://doi.org/10.1112/S0010437X0500148X},
}

@article {ELMV-Ens,
    AUTHOR = {M. Einsiedler and E. Lindenstrauss and Ph. Michel and A. Venkatesh},
     TITLE = {The distribution of closed geodesics on the modular surface,
              and {D}uke's theorem},
   JOURNAL = {Enseign. Math. (2)},
  FJOURNAL = {L'Enseignement Math\'{e}matique. Revue Internationale. 2e S\'{e}rie},
    VOLUME = {58},
      YEAR = {2012},
    NUMBER = {3-4},
     PAGES = {249--313},
      ISSN = {0013-8584},
   MRCLASS = {11F11 (11E16 11F37 14C22 37A45 53C22)},
  MRNUMBER = {3058601},
MRREVIEWER = {Thomas R. Shemanske},
       DOI = {10.4171/LEM/58-3-2},
       URL = {https://doi.org/10.4171/LEM/58-3-2},
}

@Article{AES3D,
  Title                    = {Integer points on spheres and their orthogonal lattices},
  Author                   = {M. Aka and M. Einsiedler and U. Shapira},
  Journal                  = {Invent. Math.},
  Year                     = {2016},
  Volume				   = {206},
  Number				   = {2},
  Pages					   = {379--396}
}

@article {AMW-higherdim,
    AUTHOR = {Aka, Menny and Musso, Andrea and Wieser, Andreas},
     TITLE = {Equidistribution of rational subspaces and their shapes},
   JOURNAL = {Ergodic Theory Dynam. Systems},
  FJOURNAL = {Ergodic Theory and Dynamical Systems},
    VOLUME = {44},
      YEAR = {2024},
    NUMBER = {8},
     PAGES = {2009--2062},
      ISSN = {0143-3857,1469-4417},
   MRCLASS = {37A17 (11H06 11H55)},
  MRNUMBER = {4798789},
MRREVIEWER = {Shucheng\ Yu},
       DOI = {10.1017/etds.2023.107},
       URL = {https://doi.org/10.1017/etds.2023.107},
}

@Article{AEW-2in4,
  Title                    = {Planes in four space and four associated {CM} points},
  Author                   = {M. Aka and M. Einsiedler and A. Wieser},
  Journal                  = {Duke Math. J.},
  pages = {1469--1529},
  volume = {171},
  number = {7},
  Year			 = {2020}
}

@book {bump,
    AUTHOR = {Bump, Daniel},
     TITLE = {Automorphic forms and representations},
    SERIES = {Cambridge Studies in Advanced Mathematics},
    VOLUME = {55},
 PUBLISHER = {Cambridge University Press, Cambridge},
      YEAR = {1997},
     PAGES = {xiv+574},
      ISBN = {0-521-55098-X},
   MRCLASS = {11F70 (11F41 11R39 22E50 22E55)},
  MRNUMBER = {1431508},
MRREVIEWER = {Solomon\ Friedberg},
       DOI = {10.1017/CBO9780511609572},
       URL = {https://doi.org/10.1017/CBO9780511609572},
}

@article{BlomerBrumleyKhayutin,
      title={The mixing conjecture under {GRH}}, 
      author={Valentin Blomer and Farrell Brumley and Ilya Khayutin},
    journal={arXiv preprint 2212.06280},
      year={2022},
      eprint={2212.06280},
      archivePrefix={arXiv},
      primaryClass={math.NT},
      url={https://arxiv.org/abs/2212.06280}, 
}

@article {BlomerBrumley-mixing,
    AUTHOR = {Blomer, Valentin and Brumley, Farrell},
     TITLE = {Simultaneous equidistribution of toric periods and fractional
              moments of {$L$}-functions},
   JOURNAL = {J. Eur. Math. Soc. (JEMS)},
  FJOURNAL = {Journal of the European Mathematical Society (JEMS)},
    VOLUME = {26},
      YEAR = {2024},
    NUMBER = {8},
     PAGES = {2745--2796},
      ISSN = {1435-9855,1435-9863},
   MRCLASS = {11F67 (11E45 11R52)},
  MRNUMBER = {4756945},
MRREVIEWER = {Andre\ Reznikov},
       DOI = {10.4171/jems/1324},
       URL = {https://doi.org/10.4171/jems/1324},
}

@article{BlomerBrumleyRadziwill,
      title={{Joint Linnik problems}}, 
      author={Valentin Blomer and Farrell Brumley and Maksym Radiwiłł},
    journal={arXiv preprint 2603.05609},
      year={2026},
      eprint={2603.05609},
      archivePrefix={arXiv},
      primaryClass={math.NT},
      url={https://arxiv.org/abs/2603.05609}, 
}

@Book{cassels,
  Title                    = {Rational quadratic forms},
  Author                   = {J.W.S. Cassels},
  Publisher                = {Academic Press Inc.},
  Year                     = {1978},
  Series                   = {London Mathematical Society Monographs},
  Volume                   = {13}
}

@Article{duke88,
  Title                    = {Hyperbolic distribution problems and half-integral weight \uppercase{M}aass forms},
  Author                   = {W. Duke},
  Journal                  = {Inventiones mathematicae},
  Year                     = {1988},
  Number                   = {1},
  Pages                    = {73-90},
  Volume                   = {92}
}

@article {EL-joiningsPIHES,
    AUTHOR = {Einsiedler, Manfred and Lindenstrauss, Elon},
     TITLE = {Joinings of higher rank torus actions on homogeneous spaces},
   JOURNAL = {Publ. Math. Inst. Hautes \'Etudes Sci.},
  FJOURNAL = {Publications Math\'ematiques. Institut de Hautes \'Etudes
              Scientifiques},
    VOLUME = {129},
      YEAR = {2019},
     PAGES = {83--127},
      ISSN = {0073-8301,1618-1913},
   MRCLASS = {22E40 (22D40 37A05)},
  MRNUMBER = {3949028},
MRREVIEWER = {Thomas\ Ward},
       DOI = {10.1007/s10240-019-00103-y},
       URL = {https://doi.org/10.1007/s10240-019-00103-y},
}

@Article{EL-nonmaximal,
  Title                    = {RIGIDITY OF NON-MAXIMAL TORUS ACTIONS, UNIPOTENT QUANTITATIVE RECURRENCE, AND DIOPHANTINE APPROXIMATIONS},
  Author                   = {M. Einsiedler and E. Lindenstrauss},
  Journal                  = {arXiv preprint, arXiv:2307.04163},
  Year					   = {2023},
  url					   = {https://arxiv.org/pdf/2307.04163}
}

@Article{ELMVAnn,
  Title                    = {Distribution of periodic torus orbits and \uppercase{D}uke's theorem for cubic fields},
  Author                   = {M. Einsiedler and E. Lindenstrauss and Ph. Michel and A. Venkatesh},
  Journal                  = {Annals of Mathematics},
  Year                     = {2011},
  Pages                    = {815-885},
  Volume                   = {173}
}

@Article{ELMV-DukeJ,
  Title                    = {Distribution of periodic torus orbits on homogeneous spaces},
  Author                   = {M. Einsiedler and E. Lindenstrauss and P. Michel and A. Venkatesh},
  Journal                  = {Duke Math. J.},
  Year                     = {2009},
  Pages                    = {119-174},
  Volume                   = {148},
  Number		   = {1}
}

@article {EMMV,
    AUTHOR = {Einsiedler, M. and Margulis, G. and Mohammadi, A. and
              Venkatesh, A.},
     TITLE = {Effective equidistribution and property {$(\tau)$}},
   JOURNAL = {J. Amer. Math. Soc.},
  FJOURNAL = {Journal of the American Mathematical Society},
    VOLUME = {33},
      YEAR = {2020},
    NUMBER = {1},
     PAGES = {223--289},
      ISSN = {0894-0347},
   MRCLASS = {11K36 (11E20 11E57 22E55 37A17 37A46)},
  MRNUMBER = {4066475},
MRREVIEWER = {Herbert Abels},
       DOI = {10.1090/jams/930},
       URL = {https://doi.org/10.1090/jams/930},
}

@misc{wu2018,
      title={Subconvex bounds for compact toric integrals}, 
      author={Han Wu},
      year={2018},
      eprint={1604.01902},
      archivePrefix={arXiv},
      primaryClass={math.NT},
   note={        \url{https://arxiv.org/abs/1604.01902}}, 
}

@book {Kitaoka-book,
    AUTHOR = {Kitaoka, Yoshiyuki},
     TITLE = {Arithmetic of quadratic forms},
    SERIES = {Cambridge Tracts in Mathematics},
    VOLUME = {106},
 PUBLISHER = {Cambridge University Press, Cambridge},
      YEAR = {1993},
     PAGES = {x+268},
      ISBN = {0-521-40475-4},
   MRCLASS = {11E12 (11E08 11H55)},
  MRNUMBER = {1245266},
MRREVIEWER = {Rainer\ Schulze-Pillot},
       DOI = {10.1017/CBO9780511666155},
       URL = {https://doi.org/10.1017/CBO9780511666155},
}

@Book{linnik,
  Title                    = {Ergodic properties of algebraic fields},
  Author                   = {Yu. V. Linnik},
  Publisher                = {Springer-Verlag},
  Year                     = {1968},
  Address                  = {New York},
  Note                     = {Translated from the Russian by M.S. Keane},
  Series                   = {Ergebnisse der Mathematik und ihrer Grenzgebiete},
  Volume                   = {45}
}

@Book{platonov,
  Title                    = {Algebraic groups and number theory},
  Author                   = {V. Platonov and A. Rapinchuk},
  Publisher                = {Academic Press, Inc.},
  Year                     = {1994},
  Note                     = {Translated from the 1991 Russian original by R. Rowen},
  Series                   = {Pure and Applied Mathematics},
  Volume                   = {139}
}

@incollection {Rapinchuk-survey,
    AUTHOR = {Rapinchuk, Andrei S.},
     TITLE = {Strong approximation for algebraic groups},
 BOOKTITLE = {Thin groups and superstrong approximation},
    SERIES = {Math. Sci. Res. Inst. Publ.},
    VOLUME = {61},
     PAGES = {269--298},
 PUBLISHER = {Cambridge Univ. Press, Cambridge},
      YEAR = {2014},
      ISBN = {978-1-107-03685-7},
   MRCLASS = {20G30 (11F06 20G25 20G35)},
  MRNUMBER = {3220894},
MRREVIEWER = {Herbert\ Abels},
}

@incollection {SchulzePillot-survey04,
    AUTHOR = {Schulze-Pillot, Rainer},
     TITLE = {Representation by integral quadratic forms---a survey},
 BOOKTITLE = {Algebraic and arithmetic theory of quadratic forms},
    SERIES = {Contemp. Math.},
    VOLUME = {344},
     PAGES = {303--321},
 PUBLISHER = {Amer. Math. Soc., Providence, RI},
      YEAR = {2004},
      ISBN = {0-8218-3441-X},
   MRCLASS = {11E12 (11E25 11E45)},
  MRNUMBER = {2060206},
MRREVIEWER = {Hidenori\ Katsurada},
       DOI = {10.1090/conm/344/06226},
       URL = {https://doi.org/10.1090/conm/344/06226},
}

@article {SchulzePillot-2in4,
    AUTHOR = {Schulze-Pillot, Rainer},
     TITLE = {Averages of {F}ourier coefficients of {S}iegel modular forms
              and representation of binary quadratic forms by quadratic
              forms in four variables},
   JOURNAL = {Math. Ann.},
  FJOURNAL = {Mathematische Annalen},
    VOLUME = {368},
      YEAR = {2017},
    NUMBER = {3-4},
     PAGES = {923--943},
      ISSN = {0025-5831,1432-1807},
   MRCLASS = {11E12 (11E45 11F27 11F30 11F46)},
  MRNUMBER = {3673640},
MRREVIEWER = {Dongxi\ Ye},
       DOI = {10.1007/s00208-016-1448-4},
       URL = {https://doi.org/10.1007/s00208-016-1448-4},
}

@Article{Venkatesh-sparse,
author = {A.~Venkatesh},
title = {Sparse equidistribution problems, period bounds and subconvexity},
journal = {Ann.~of Math.},
volume = {172},
year = {2010}, 
number = {2}, 
pages = {989--1094}
}

@book {voight,
    AUTHOR = {Voight, John},
     TITLE = {Quaternion algebras},
    SERIES = {Graduate Texts in Mathematics},
    VOLUME = {288},
 PUBLISHER = {Springer, Cham},
      YEAR = {[2021] \copyright 2021},
     PAGES = {xxiii+885},
      ISBN = {978-3-030-56692-0; 978-3-030-56694-4},
   MRCLASS = {11R52 (11-02 11E12 11F06 16H05 16U60 20H10)},
  MRNUMBER = {4279905},
MRREVIEWER = {Juliusz\ Brzezi\'nski},
       DOI = {10.1007/978-3-030-56694-4},
       URL = {https://doi.org/10.1007/978-3-030-56694-4},
}

@article {W-Linnik,
    AUTHOR = {Wieser, Andreas},
     TITLE = {Linnik's problems and maximal entropy methods},
   JOURNAL = {Monatsh. Math.},
  FJOURNAL = {Monatshefte f\"ur Mathematik},
    VOLUME = {190},
      YEAR = {2019},
    NUMBER = {1},
     PAGES = {153--208},
      ISSN = {0026-9255,1436-5081},
   MRCLASS = {37A17 (11E12 11K06)},
  MRNUMBER = {3998337},
MRREVIEWER = {Thomas\ Ward},
       DOI = {10.1007/s00605-019-01320-7},
       URL = {https://doi.org/10.1007/s00605-019-01320-7},
}

@article {Wu14,
    AUTHOR = {Wu, Han},
     TITLE = {Burgess-like subconvex bounds for
              {$\text{GL}_2\times\text{GL}_1$}},
   JOURNAL = {Geom. Funct. Anal.},
  FJOURNAL = {Geometric and Functional Analysis},
    VOLUME = {24},
      YEAR = {2014},
    NUMBER = {3},
     PAGES = {968--1036},
      ISSN = {1016-443X,1420-8970},
   MRCLASS = {11F67 (11F70 22E55)},
  MRNUMBER = {3213837},
MRREVIEWER = {Jannis\ A.\ Antoniadis},
       DOI = {10.1007/s00039-014-0277-4},
       URL = {https://doi.org/10.1007/s00039-014-0277-4},
}

@article {MV10,
    AUTHOR = {Michel, Philippe and Venkatesh, Akshay},
     TITLE = {The subconvexity problem for {${\rm GL}_2$}},
   JOURNAL = {Publ. Math. Inst. Hautes \'Etudes Sci.},
  FJOURNAL = {Publications Math\'ematiques. Institut de Hautes \'Etudes
              Scientifiques},
    NUMBER = {111},
      YEAR = {2010},
     PAGES = {171--271},
      ISSN = {0073-8301,1618-1913},
   MRCLASS = {11F67 (11F37 11F70 22E55 58J51)},
  MRNUMBER = {2653249},
MRREVIEWER = {Andre\ Reznikov},
       DOI = {10.1007/s10240-010-0025-8},
       URL = {https://doi.org/10.1007/s10240-010-0025-8},
}

@article {CHH88,
    AUTHOR = {Cowling, M. and Haagerup, U. and Howe, R.},
     TITLE = {Almost {$L^2$} matrix coefficients},
   JOURNAL = {J. Reine Angew. Math.},
  FJOURNAL = {Journal f\"ur die Reine und Angewandte Mathematik. [Crelle's
              Journal]},
    VOLUME = {387},
      YEAR = {1988},
     PAGES = {97--110},
      ISSN = {0075-4102,1435-5345},
   MRCLASS = {22D10 (22E46)},
  MRNUMBER = {946351},
MRREVIEWER = {Rebecca\ Herb},
       DOI = {10.1515/crll.1988.387.97},
       URL = {https://doi.org/10.1515/crll.1988.387.97},
}

@book {YZZ13,
    AUTHOR = {Yuan, Xinyi and Zhang, Shou-Wu and Zhang, Wei},
     TITLE = {The {G}ross-{Z}agier formula on {S}himura curves},
    SERIES = {Annals of Mathematics Studies},
    VOLUME = {184},
 PUBLISHER = {Princeton University Press, Princeton, NJ},
      YEAR = {2013},
     PAGES = {x+256},
      ISBN = {978-0-691-15592-0},
   MRCLASS = {11G18 (11F70 14G35)},
  MRNUMBER = {3237437},
MRREVIEWER = {Ernest\ Hunter\ Brooks},
}

@article {Kim03,
    AUTHOR = {Kim, Henry H.},
     TITLE = {Functoriality for the exterior square of {${\rm GL}_4$} and
              the symmetric fourth of {${\rm GL}_2$}},
      NOTE = {With appendix 1 by Dinakar Ramakrishnan and appendix 2 by Kim
              and Peter Sarnak},
   JOURNAL = {J. Amer. Math. Soc.},
  FJOURNAL = {Journal of the American Mathematical Society},
    VOLUME = {16},
      YEAR = {2003},
    NUMBER = {1},
     PAGES = {139--183},
      ISSN = {0894-0347,1088-6834},
   MRCLASS = {11F70 (11R39 22E46)},
  MRNUMBER = {1937203},
MRREVIEWER = {Mahdi\ Asgari},
       DOI = {10.1090/S0894-0347-02-00410-1},
       URL = {https://doi.org/10.1090/S0894-0347-02-00410-1},
}

@article {HL94,
    AUTHOR = {Hoffstein, Jeffrey and Lockhart, Paul},
     TITLE = {Coefficients of {M}aass forms and the {S}iegel zero},
      NOTE = {With an appendix by Dorian Goldfeld, Hoffstein and Daniel
              Lieman},
   JOURNAL = {Ann. of Math. (2)},
  FJOURNAL = {Annals of Mathematics. Second Series},
    VOLUME = {140},
      YEAR = {1994},
    NUMBER = {1},
     PAGES = {161--181},
      ISSN = {0003-486X,1939-8980},
   MRCLASS = {11F37 (11F12 11F30 11F66)},
  MRNUMBER = {1289494},
MRREVIEWER = {Lynne\ H.\ Walling},
       DOI = {10.2307/2118543},
       URL = {https://doi.org/10.2307/2118543},
}

@article {MW09,
    AUTHOR = {Martin, Kimball and Whitehouse, David},
     TITLE = {Central {$L$}-values and toric periods for {${\rm GL}(2)$}},
   JOURNAL = {Int. Math. Res. Not. IMRN},
  FJOURNAL = {International Mathematics Research Notices. IMRN},
      YEAR = {2009},
    NUMBER = {1},
     PAGES = {Art. ID rnn127, 141--191},
      ISSN = {1073-7928,1687-0247},
   MRCLASS = {11F67 (11F41 22E55)},
  MRNUMBER = {2471298},
MRREVIEWER = {Alexandru\ A.\ Popa},
       DOI = {10.1093/imrn/rnn127},
       URL = {https://doi.org/10.1093/imrn/rnn127},
}

@incollection {IS00,
    AUTHOR = {Iwaniec, H. and Sarnak, P.},
     TITLE = {Perspectives on the analytic theory of {$L$}-functions},
      NOTE = {GAFA 2000 (Tel Aviv, 1999)},
   JOURNAL = {Geom. Funct. Anal.},
  FJOURNAL = {Geometric and Functional Analysis},
      YEAR = {2000},
     PAGES = {705--741},
      ISSN = {1016-443X,1420-8970},
   MRCLASS = {11M26 (11F67 11F70 11F72 11G40)},
  MRNUMBER = {1826269},
MRREVIEWER = {Henry\ H.\ Kim},
       DOI = {10.1007/978-3-0346-0425-3\_6},
       URL = {https://doi.org/10.1007/978-3-0346-0425-3_6},
}

@incollection {Mic21,
    AUTHOR = {Michel, Philippe},
     TITLE = {Some recents advances on {D}uke's equidistribution theorems},
 BOOKTITLE = {Nine mathematical challenges---an elucidation},
    SERIES = {Proc. Sympos. Pure Math.},
    VOLUME = {104},
     PAGES = {107--132},
 PUBLISHER = {Amer. Math. Soc., Providence, RI},
      YEAR = {[2021] \copyright 2021},
      ISBN = {978-1-4704-5490-6},
   MRCLASS = {11F66 (11E20 11F41 37A44)},
  MRNUMBER = {4337419},
MRREVIEWER = {Shaoyun\ Yi},
       DOI = {10.1090/pspum/104/01873},
       URL = {https://doi.org/10.1090/pspum/104/01873},
}

\end{document}